\documentclass[11pt, a4paper]{amsart}

\usepackage[T1]{fontenc}
\usepackage[utf8]{inputenc}
\usepackage{amsmath, amssymb, amsfonts, amsthm}
\usepackage{mathrsfs, mathtools}
\usepackage{array}
\usepackage{csquotes}
\usepackage[dvipsnames]{xcolor}
\usepackage{a4wide}
\usepackage{tikz}
\usepackage{accents}
\usepackage{extarrows}
\usepackage{calc}
\usepackage{verbatim}

\definecolor{cadmiumgreen}{rgb}{0.0, 0.42, 0.24}
\definecolor{darkred}{rgb}{.85,0,0}
\usepackage[
  colorlinks,
  linkcolor=darkred,
  citecolor=cadmiumgreen,
  pdfstartview ={FitV},
]{hyperref}

\usepackage[
  alphabetic,
  msc-links,
  nobysame,
  lite,
]{amsrefs}

\allowdisplaybreaks

\usetikzlibrary{calc}
\usetikzlibrary{intersections}
\usetikzlibrary{decorations.markings}
\usetikzlibrary{decorations.pathmorphing}
\usetikzlibrary{arrows.meta}
\usetikzlibrary{cd}

\newcommand{\midarrow}{node[midway,sloped,allow upside down] {\tikz \draw[->,thick] (-1pt,0)--++(2pt,0);}}
\newcommand{\mmidarrow}[2]{postaction={decorate, decoration={markings, mark=at position .5 with {\arrow{>}}}}}

\newcommand{\Z}{\mathbb Z}
\newcommand{\Q}{\mathbb Q}
\newcommand{\R}{\mathbb R}

\newcommand{\zint}[2]{
  \if #11
    [#2]
  \else
    [#1\,\raisebox{.2ex}.\,\raisebox{.2ex}.\,#2] \fi }
\newcommand{\rquot}[2]{#1\big/#2}
\newcommand{\st}{\,\big|\,}
\newcommand{\ie}[1]{i.e.,~}
\newcommand{\eg}[1]{e.g.,~}
\newcommand{\cf}{cf.~}
\newcommand{\resp}{resp.\ }
\newcommand{\abs}[1]{\lvert #1\rvert}
\newcommand{\simto}{\xrightarrow{\raisebox{-1pt}{\small$\hspace{-1pt}\sim\hspace{0pt}$}}}
\newcommand{\wcdot}{{}\cdot{}}

\renewcommand{\~}{\widetilde}
\renewcommand{\hat}{\widehat}
\renewcommand{\leq}{\leqslant}
\renewcommand{\geq}{\geqslant}
\renewcommand{\d}{\partial}
\renewcommand{\epsilon}{\varepsilon}
\let\ssubset\subset
\renewcommand{\subset}{\subseteq}

\theoremstyle{plain}
\newtheorem{theorem}{Theorem}[section]
\newtheorem*{theorem*}{Theorem}
\newtheorem{proposition}[theorem]{Proposition}
\newtheorem{lemma}[theorem]{Lemma}
\newtheorem{corollary}[theorem]{Corollary}
\newtheorem{claim}[theorem]{Claim}

\theoremstyle{definition}
\newtheorem{defii}[theorem]{Definition}
\newenvironment{definition}
  {\pushQED{\qed}\defii}
  {\popQED\enddefii}
\newtheorem{remm}[theorem]{Remark}
\newenvironment{remark}
  {\pushQED{\qed}\remm}
  {\popQED\endremm}
\newtheorem{exx}[theorem]{Example}
\newenvironment{example}
  {\pushQED{\qed}\exx}
  {\popQED\endexx}

\newcommand{\M}{\mathcal M}
\newcommand{\Ma}{\mathfrak M}
\newcommand{\E}{E}
\newcommand{\F}{F}
\newcommand{\FF}{\mathbb F}
\newcommand{\Ind}{\mathfrak I}
\newcommand{\J}{\mathcal J}

\newcommand{\G}{\mathcal G}
\renewcommand{\H}{\mathcal H}
\newcommand{\U}{\mathcal U}
\newcommand{\V}{\mathcal V}
\newcommand{\W}{\mathcal W}
\newcommand{\gE}{\mathcal E}
\newcommand{\trees}{\mathcal T}

\newcommand{\RPII}{\R\mathrm{P}^2}

\newcommand{\bound}{\d}
\newcommand{\boundsing}{\bound_\mathrm{sing}}

\renewcommand{\S}{\mathcal S}
\newcommand{\smS}{\S^{\mathrm{sm}}}
\newcommand{\mom}{\mathfrak p}

\newcommand{\vG}{\mathfrak G}
\newcommand{\vH}{\mathfrak H}
\newcommand{\vV}{\mathfrak V}
\newcommand{\vE}{\mathfrak E}
\newcommand{\va}{\mathfrak a}
\newcommand{\vb}{\mathfrak b}
\newcommand{\ve}{\mathfrak e}

\newcommand{\A}{A}

\newcommand{\f}{z}
\newcommand{\fF}{Z}
\newcommand{\sR}{\mathscr R}
\newcommand{\smsR}{{\sR^{\mathrm{sm}}}}
\newcommand{\sZ}{\underline\Z}
\newcommand{\scR}{\underline\R}
\newcommand{\X}{\mathscr X}

\newcommand{\sign}{\mathrm{sign}}
\newcommand{\sgn}{\mathrm{\sigma}}
\let\Im\relax
\newcommand{\Im}{\mathrm{Im}}

\newcommand{\rk}{\mathrm{rk}}
\newcommand{\Kir}{\mathrm{Kir}}
\newcommand{\Sym}{\mathrm{Sym}}
\newcommand{\Ext}{\mathrm{\bigwedge}}
\newcommand{\ratSym}{\mathrm{\~\Sym}}
\newcommand{\diag}{\mathrm{diag}}
\let\det\relax
\newcommand{\det}{\mathrm{det}}
\newcommand{\Chain}{\mathrm{C}}
\newcommand{\Cyc}{\mathrm{Z}}
\newcommand{\Bound}{\mathrm{B}}
\newcommand{\Homl}{\mathrm{H}}

\newcommand{\For}{\mathrm{\mathcal F}}
\newcommand{\Bas}{\mathrm{\mathcal B}}
\let\P\relax
\newcommand{\P}{\mathrm{\mathcal P}}
\newcommand{\cl}{\mathrm{cl}}
\newcommand{\Fr}{\mathrm{fr}}
\newcommand{\MCP}{\mathrm{MCP}}
\newcommand{\PFE}{\mathrm{PFE}}

\newcommand{\size}{\mathrm{size}}
\newcommand{\mhyp}{\mathrm{\mathfrak C}}

\newcommand{\Id}{\mathrm{Id}}
\newcommand{\Tor}{\mathrm{Tor}}

\newcommand{\Jac}{\mathrm{Jac}}

\newcommand{\Trees}{\mathrm{\mathcal T}}
\newcommand{\SF}{\mathrm{\mathcal{SF}}}
\newcommand{\conv}{\mathrm{conv}}

\newcommand{\Vol}{\mathrm{Vol}}
\let\O\relax
\newcommand{\O}{\mathrm{\mathcal O}}
\let\o\relax
\newcommand{\o}{\mathrm{o}}
\newcommand{\CC}{\mathrm{CC}}
\newcommand{\SCC}{\mathrm{SCC}}
\newcommand{\OCC}{\mathrm{OCC}}
\newcommand{\interior}{\mathrm{int}}
\newcommand{\degs}{\mathrm{degs}}
\newcommand{\bighp}{\mathop{\raisebox{-.5em}{\scalebox{2}{$\boxplus$}}}\limits}

\newcommand{\concat}{\star}
\newcommand{\tran}{{^\intercal}}
\newcommand{\multicup}{\uplus}

\newcommand{\mmult}[1]{\cdot_{#1}}

\newcommand{\myand}{\text{ and }}

\newcommand{\compl}[1]{#1^c}
\newcommand{\del}{\backslash}
\newcommand{\mat}{\mathbf}

\newcommand{\rest}[1]{\big|_{#1}}
\newcommand{\mrest}{|}
\newcommand{\contr}[1]{/\!#1}
\newcommand{\card}[1]{|#1|}

\newcommand{\geomreal}[1]{|#1|}

\newcommand{\perm}[1]{\mathfrak S_{#1}}
\newcommand{\freeZmod}[1]{\Z\langle#1\rangle}

\newcommand{\norm}[1]{\lVert#1\rVert}

\newcommand{\Zmod}{\Z}
\newcommand{\vect}[1]{\langle#1\rangle}
\newcommand{\Zvect}[1]{\langle#1\rangle_\Z}
\newcommand{\supp}[1]{\mathrm{supp}(#1)}
\newcommand{\divprod}[2]{\langle #1,#2\rangle}
\newcommand{\hp}{\boxplus}

\def\Hsing_#1{\Homl_{#1,\mathrm{sing}}}
\def\Chainsing_#1{\Chain_{#1,\mathrm{sing}}}
\def\Chainsingreg_#1{\Chain_{#1,\mathrm{sing}}^{\mathrm{reg}}}
\def\Boundsing_#1{\Bound_{#1,\mathrm{sing}}}
\def\Cycsing_#1{\Cyc_{#1,\mathrm{sing}}}

\newcommand{\pbmatrix}[4]{\left(
  \,\raisebox{-.45\height}{\tikz[x=.7cm, y=.7cm]{
  \draw[gray!50] (0,0) rectangle ++(1,1) node[midway, black] {$#1$};
  \draw[gray!50] (1.1,0) rectangle ++(.8,1) node[midway, black] {$#2$};
  \draw[gray!50] (0,-.1) rectangle ++(1,-.8) node[midway, black] {$#3$};
  \draw[gray!50] (1.1,-.1) rectangle ++(.8,-.8) node[midway, black] {$#4$};
  }}
\right)}

\newcommand{\pbmatrixIV}[4]{\left(
  \,\raisebox{-.45\height}{\tikz[x=.7cm, y=.7cm]{
  \draw[gray!50] (0,0) rectangle ++(1.2,1.6) node[midway, black] {$#1$};
  \draw[gray!50] (1.3,0) rectangle ++(1,1.6) node[midway, black] {$#2$};
  \draw[gray!50] (0,-.1) rectangle ++(1.2,-1) node[midway, black] {$#3$};
  \draw[gray!50] (1.3,-.1) rectangle ++(1,-1) node[midway, black] {$#4$};
  }}
\right)}
\newcommand{\pbmatrixV}[4]{\left(\!
  \raisebox{-.4\height}{\tikz[x=.7cm, y=.7cm]{
  \draw[gray!50] (0,0) rectangle ++(1.3,1.3) node[midway, black] {$#1$};
  \draw[gray!50] (1.4,0) rectangle ++(.5,1.3) node[midway, black] {$#2$};
  \draw[gray!50] (0,-.1) rectangle ++(1.3,-.5) node[midway, black] {$#3$};
  \draw[gray!50] (1.4,-.1) rectangle ++(.5,-.5) node[midway, black] {$#4$};
  }}
\!\right)}

\begin{document}

\title{A multidimensional generalization of Symanzik polynomials} \label{chap:symanzik}
\author{Matthieu Piquerez}
\email{matthieu.piquerez@univ-nantes.fr}
\address{Nantes Université, École Centrale Nantes, CNRS, INRIA, LS2N, UMR 6004, France}
% \date{\today}

\begin{abstract}
Symanzik polynomials are defined on Feynman graphs and they are used in quantum field theory to compute Feynman amplitudes.
They also appear in mathematics from different perspectives. For example, recent results show that they allow to describe asymptotic limits of geometric quantities associated to families of Riemann surfaces.

In this paper, we propose a generalization of Symanzik polynomials to the setting of higher dimensional simplicial complexes and study their basic properties and applications. We state a duality relation between these generalized Symanzik polynomials and what we call Kirchhoff polynomials, which have been introduced in recent generalizations of Kirchhoff's matrix-tree theorem to simplicial complexes. Moreover, we obtain geometric invariants which compute interesting data on triangulable manifolds. As the name indicates, these invariants do not depend on the chosen triangulation.

We furthermore prove a stability theorem concerning the ratio of Symanzik polynomials which extends a stability theorem of Amini to higher dimensional simplicial complexes. In order to show that theorem, we will make great use of matroids, and provide a complete classification of the connected components of the exchange graph of a matroid, a graph which encodes the exchange properties between independent sets of the matroid. We hope that this result could be of independent interest.

Finally, we explain how to generalize Symanzik polynomials to the setting of matroids over hyperfields defined recently by Baker and Bowler.
\end{abstract}

\maketitle

% \setcounter{tocdepth}{1}

% \tableofcontents

%%%
\section{Introduction}

The aim of this paper is to introduce a family of polynomials that we call \emph{generalized Symanzik polynomials}, and to study their geometric and combinatorial properties.

Classical Symanzik polynomials arising in quantum field theory are associated to \emph{Feynman graphs} and used for computing Feynman amplitudes. They are defined as follows.

Let $G=(V,E)$ be a connected graph with vertex set $V$ and edge set $E$. Let $\mom=(\mom_v)_{v\in V}$ be a collection of vectors called the \emph{external momenta} such that each $\mom_v$, the external momentum of $v\in V$, is an element of the real vector space $\R^D$, for some positive integer $D$, which is endowed with a Minkowski bilinear form. Moreover, we suppose that the collection of external momenta verify the conservation of momenta hypothesis, namely that $\sum_{v\in V}\mom_v=0$. Such a pair $(G,\mom)$ is called a \emph{Feynman graph}. In what follows we will only consider the case $D=1$, but the results can be extended to the more general setting as in~\cite{ABBF}.

The first Symanzik polynomial denoted by $\psi_G$ is defined by
\begin{equation} \label{5:eqn:first_Symanzik_polynomial}
\psi_G(x):=\sum_{T\in\Trees}\prod_{e\not\in T}x_e,
\end{equation}
where $\Trees$ denotes the set of spanning trees of $G$, where $x=(x_e)_{e\in E}$ is a collection of variables indexed by the edges of the graph, and the product for a spanning tree $T \in \Trees$ is on all edges of $G$ which are not in $T$. (A spanning tree of $G=(V,E)$ is a connected subgraph of $G$ which has vertex set $V$ and which does not have any cycle.) Note in particular that the first Symanzik polynomial does not depend on the external momenta.

The second Symanzik polynomial denoted $\phi_G$ is defined by
\begin{equation} \label{5:eqn:second_Symanzik_polynomial}
\phi_G(\mom, x):=\sum_{F\in\SF_2}q(F)\prod_{e\not\in F}x_e.
\end{equation}
In this formula, $\SF_2$ denotes the set of spanning forests of $G$ which have two connected components. (A spanning forest of $G$ is a subgraph with vertex set $V$ and without any cycle.) For a spanning forest $F\in\SF_2$, the term $q(F)$ is defined by $q(F):=-\langle\mom_{F_1},\mom_{F_2}\rangle$ where $F_1$ and $F_2$ are the two connected components of $F$, and where, for $i\in\{1,2\}$, $\mom_{F_i}$ is the sum of the momenta of vertices in $F_i$. Using these polynomials, the Feynman amplitude can be computed then as a path integral of $\exp(-i\phi_G/\psi_G)$.

\medskip

The above polynomials have many known interesting properties, most of them are summarized in~\cite{BW10}. In this paper, we present natural generalizations of the above definitions. In particular, we are going to extend the setting to higher dimensional simplicial complexes. Other generalizations already exist, see \eg,~\cite{GR07} and~\cite{KRTW10}, though the purpose is different from ours. In addition to introducing a larger family of polynomials with interesting properties, we hope that this generalization will provide a better understanding of these polynomials and will widens their applications.

\medskip

Symanzik polynomials are known to have tight connection to other concepts in mathematics. Before going to the heart of our construction and presenting an overview of our results, we would like to start by highlighting some of these links and making a series of comments to motivate the study undertaken in this paper.

\medskip

First, Symanzik polynomials of graphs are known to be in a sense \emph{dual} to Kirchhoff polynomials. These latter polynomials are those obtained applying the well-known (weighted) Kirchhoff matrix-tree theorem, which counts the number of spanning trees of a graph as the determinant of a matrix linked to the Laplacian of the graph. Recently, the notion of a spanning tree has been extended to the case of higher dimensional simplicial complexes, and Kirchhoff's theorem has been generalized to this setting, c.f.~\cites{DKM09, Lyo09} and~\cite{BK16}. In order to generalize the definition of Symanzik polynomials to this setting, it is thus natural to exploit and formulate an appropriate version of the duality theorem in this setting. This is what we have done to find our definition of generalized Symanzik polynomials. In the classical case, this duality infers a determinantal formula for Symanzik polynomials. This is still true after the generalization.

\medskip

This leads us to our second point which concerns \emph{determinantal probability measures} (see for instance~\cite{Lyo03}). This tight relation between Kirchhoff polynomials, spanning trees and determinants is studied thoroughly in the paper~\cite{Lyo09} of Lyons on random complexes. In fact, what appears here are determinantal probabilities. For instance, the probability for a uniform (or suitably weighted) spanning tree of a graph to contain a specific edge can be easily computed because it is given by the determinant of some matrix. The Laplacian of the graph is used to define this matrix. In the same way, Kirchhoff polynomials are also obtained using this Laplacian. By analogy, the matrix whose determinant gives the first Symanzik polynomial allows us to compute the probability for a spanning tree not to contain a specific edge. In this paper, we will also generalize Kirchhoff and Symanzik polynomials in two different ways: first to \emph{spanning forests} and second to \emph{higher order determinants} (the classical definition is of order two). In~\cite{KL22}, some generalizations of Kirchhoff and Symanzik polynomials also emerge from the study of some families of determinantal random subgraphs. Concerning higher order determinants, they naturally leads to \emph{hyperdeterminants} which are also linked to determinantal probability measures as it was shown in~\cite{EG09}.

\medskip

This brings us to the third point, which is the link to \emph{the theory of matroids}. Matroids are mathematical objects obtained by abstracting and generalizing some of the main notions of linear algebra and, in particular, of graph theory, as spanning trees for example (we refer to~\cite{Oxl11} for further information). In~\cite{BW10}, the link between matroid theory and Symanzik polynomials is noticed. We go further than~\cite{BW10} by defining Kirchhoff and Symanzik polynomials for matroids and for their recent generalization \emph{matroids over hyperfields} obtained by Baker and Bowler in~\cite{BB16} (see also~\cite{BB19}). This generalization is needed to understand the generalized operations of deletion and contraction, which are less natural than in the case of graphs. These two operations are at the heart of many properties of classical Kirchhoff and Symanzik polynomials, in particular deletion-contraction formula and a partial factorization of the polynomials used, for instance, in~\cite{Bro15} (see below). Behind all this, there is the multivariate Tutte polynomial, which is also known as the Potts model partition function by physicists (\cf \cite{Sok05}*{(4.11) and (2.18)}). Unfortunately, a canonical Tutte polynomial for matroids over hyperfields does not seem to exist. Though, using the recent work of Dupont, Fink and Moci in~\cite{DFM18} (itself based on~\cite{KMT18}), we will see that Kirchhoff and Symanzik polynomials over hyperfields still preserve all the properties one could expect of Tutte polynomial and its specializations.

Apart from this generalization, we show that matroids play a major role in understanding some other interesting features of the Symanzik polynomials. These purely combinatorial results about matroids, which we hope should be of independent interest, concern the exchange properties between \emph{independent sets} of a general matroid, which generalize well-known exchange properties between bases. For a connected graph, it is well-known that any spanning tree can be obtained from any other by a sequence consisting of exchanging one edge at a time. The exchange properties between spanning trees and spanning 2-forests in a graph were studied by Amini in~\cite{Ami19} for the purpose of understanding the asymptotic of some geometrically defined quantities in a family of Riemann surfaces. Generalizing the definition of the exchange graph of a graph from~\cite{Ami19} to any general matroid, we provide a complete characterization of the connected components of this graph and use it to generalize the \emph{stability theorem} of~\cite{Ami19} to higher dimensional simplicial complexes.

Let us make two further remarks at this point. It is well-known that the connectivity properties of different kind of exchange graphs is the first step for creating simple and efficient random generation of bases based on Markov chains, c.f.~\cite{FM92} and the recent work of~\cite{ALOGV19} based on combinatorial Hodge theory for matroids~\cite{AHK}. What we prove in this paper is another geometric manifestation of the exchange properties in matroids. Furthermore, the results we obtain on the exchange graph of a matroid has very close ties to an old conjecture of White on toric ideals of matroids~\cites{Whi80, Bla08, LM14}. In fact, by the work of~\cite{Bla08}, a part of White's conjecture can be rephrased as whether a particular generalization of our exchange graph is always connected, c.f. Remark~\ref{5:rem:White_s_conjecture}. So our results and methods might lead to a new approach in that direction.

\medskip

As we mentioned previously, Symanzik polynomials for graphs originated from Physics, and recent advances has shown that their properties govern the mathematics of Feynman amplitudes, see \eg \cites{BB03, BEK06, Bro15, BS12, BY11}. In particular, Brown shows in~\cite{Bro15} how certain arithmetic factorization properties satisfied by the Symanzik polynomials of graphs and their minors play a crucial role in defining the \emph{cosmic Galois group}, a Galois group for Feynman amplitudes. We extend these properties to generalized Symanzik polynomials in this article.

From a point of view inclined towards Physics, we will show that our generalized polynomials present a certain stability (rather different from the stability theorem evoked above). This means that if we start with a sufficiently nice topological space $\S$, then the Symanzik polynomials of any triangulation of $\S$ come from (more precisely can be factorized into) the same Symanzik polynomial, which can be directly defined then as the Symanzik polynomial of $\S$. In other words, this stability indicates that these polynomials might be useful for computing interesting features of continuous objects by approximating them with discrete ones. For instance, path integrals are fundamental in the formulation of quantum field theory. But they are generally ill-defined. A way to go through this difficulty is to discretize the paths, and Regge calculus gives a way to do this (originally in the framework of quantum gravity), see, \eg, the original article of Regge~\cite{Reg61}, the introductory paper~\cite{Wil92}, and a more recent approach~\cite{Ori07} in the framework of group field theory. We have found striking similarities between the geometry of generalized Symanzik polynomials and Regge calculus, although at this point we are not still able to formalize this.

About the geometrical data computed by generalized Symanzik polynomials, in simple cases as compact orientable manifolds endowed with a volume form, the Symanzik polynomial simply computes the volume of the manifold (setting the good values to variables). There is not always a simple interpretation for the value computed in more complex cases. However, for general metric graphs, the first Symanzik polynomial evaluated at positive reals computes the volume of the Jacobian torus of the graph, which is defined in~\cite{KS00}.

The Symanzik polynomials of graphs has been generalized in the context of non-commuta\-tive quantum field theory~\cite{KRTW10}. The article~\cite{KRVT11} shows how these new polynomial invariants derive from the Bollob\'as-Riordan polynomial~\cite{BR02}, which is a generalization of the Tutte polynomial for ribbon graphs. The Bollob\'as-Riordan polynomial itself has been generalized to higher rank graphs (which are, under some conditions, dual to some pseudo simplicial complexes of higher dimension) in~\cite{AGH13} and~\cite{Avo16}. Two recent papers~\cite{KMT18} and~\cite{DFM18} provides a generalization of these polynomials in the context of Hopf algebras. Another recent paper~\cite{BGC17} studies this kind of generalization for a different purpose. It concerns the geometric Brownian motion, also linked to quantum theory. As far as we know, multidimensional generalizations of Symanzik polynomials have not been studied prior to this work, and it would be certainly valuable to compare our approach to the above cited works.

\medskip

Symanzik polynomials of graphs form a family of graph polynomials with very interesting properties. As already seen, they are specializations of the multivariate Tutte polynomial. Recent papers study graph polynomials which usually reveal surprising geometric properties. It is not hard to prove that the first Symanzik polynomials of the connected graphs verify the half-plane property, \ie, they are nonvanishing whenever all the variables lie in the open right complex half-plane. The papers~\cites{COSW04, Bra07} and~\cite{BBL07} study invariant polynomials verifying this property and other linked properties. They might be extended in some cases to our generalization. Another work that might be linked to ours is \cite{BH20}. In this article, the authors prove that some specializations of the homogenized multivariate Tutte polynomial of a matroid are Lorentzian, which implies numerous important properties similar to thoses mentioned above.

Moreover, in~\cite{Ami19}, it has been proved that a bounded geometrical deformation of the graph only induces a variation of the ratio of the two Symanzik polynomials that is bounded independently of the initial geometry. At the end of this article, we will generalize this stability theorem. The initial goal of Amini in~\cite{Ami19} was to give a combinatorial proof of a result he obtained with Bloch, Burgos Gil and Fres\'an in~\cite{ABBF}. In that work, they expressed the asymptotic of the Archimedean height pairing between two $0$-divisors on degenerating families of Riemann surfaces as to the ratio of the two Symanzik polynomials. The generalization presented in this paper should have links to the generalization of that work to higher dimensional varieties. In particular, as in~\cite{ABBF}, it should have links with the nilpotent orbit theorem and the SL2 orbit theorem. In \cite{AN-hybrid-moduli} which studies this time the asymptotic of canonical measures on curves, some part of the measure obtained at the limit can be defined using the Kirchhoff polynomial. We do not know if our generalized polynomial might appear in some generalization of this result to higher dimensions.

In another direction, Scott and Sokal study in~\cite{SS14} strict monotonicity properties of the inverse powers of the Kirchhoff polynomials of graphs and matroids. They give a complete characterization of the exponents in the strict monotonicity range for series-parallel graphs, and raise several questions and conjectures. One might naturally wonder if such results could be generalized to the setting of generalized Symanzik polynomials studied in this paper.

\medskip

In the rest of this introduction, we give a quick overview of the main results of this paper.

\subsection{The generalization} \label{5:subsec:intro:generalization}

The idea of the generalization comes from Kirchhoff's matrix-tree theorem (observed by Kirchhoff~\cite{Kir1847}) in its weighted form. We briefly sketch this idea here. Example~\ref{5:ex:Kirchhoff_s_thm_1} contains all the details. Let $G=(V,E)$ be a connected graph of vertex set $V$ and edge set $E=\{e_1, \dots, e_n\}$. Set a weight $y_e\in\R$ on each edge $e\in E$.
Set $Q$ the incident matrix of $G$, and $Y$ the diagonal matrix $\diag(y_{e_1},\dots, y_{e_n})$. Then, the theorem states that
\begin{equation} \label{5:eqn:intro:Kirchhoff_s_thm}
\sum_{T\in\Trees}\prod_{e\in T}y_e=\det(\~QY{\~Q}^\tran),
\end{equation}
where we removed a row from $Q$ to get $\~Q$. Let $E'\subset E$ be a subset of $\card V-1$ edges. We restrict the matrix $\~Q$ to the columns indexed by the elements of $E'$. We denote this square matrix by $\~Q_{E'}$. A step of the proof is to notice that $\det(\~Q_{E'})^2$ equals $1$ if $E'$ corresponds to a spanning tree of $G$, $0$ otherwise. Thus, the first Symanzik polynomial verifies
\[\psi_G(x)=\sum_{\substack{E'\subset E \\ \card{E'}=\card{V}-1}}\det(\~Q_{E'})^2\prod_{e\not\in E'}x_e. \]
In the right-hand member, the graph $G$ is almost absent. The main data we use is its incident matrix. Thus, one can define the Symanzik polynomial of any matrix.

However, we needed to delete a row from the matrix. When the matrix does not come from graphs, we could have to remove several rows, and the result could depend on the chosen rows. To avoid any problem, we introduce a very useful tool: the standard inner product on the exterior algebra. This is used by Lyons in~\cite{Lyo03} to deal with a subject not far from ours, as explained above.

\bigskip

Before going through details of our constructions, we need to introduce some terminology which will be used all through the paper.

\vspace{-.5em}

\subsection*{Notations and conventions}

If $p, q$ are integers with $p\leq q$, then the set $\{p, p+1, \dots, q\}$ will be written $\zint pq$, or more simply $\zint1q$ if $p=1$.
Let $I$ be a finite set. Then $\card I$ is its cardinality and $\P(I)$ is its power set. If $J\in\P(I)$ is a subset of $I$ and if there is no ambiguity, then $\compl J:=I\setminus J$ denotes the complement of $J$.
If $i\in\compl I$, then $I+i:=I\cup\{i\}$, and if $i\in I$, then $I-i:=I\setminus\{i\}$ (using these notations means, respectively, that $i\in\compl I$ and that $i\in I$). If there is an ordering on the elements $I$, we will also use $I$ to denote the ascending family of elements in $I$ (see below).

Let $n$ be a positive integer. In the whole article, $x=(x_1, \dots, x_n)$ will be a family of variables. $\Z[x]$ is the ring of polynomials over $\Z$ with variables $x_1, \dots, x_n$. Following a usual notation, if $I\subset\zint1n$,
\[ x^I:= \prod_{i\in I}x_i. \]

If $\Delta$ is a set, $\freeZmod\Delta$ is the free $\Z$-module on $\Delta$.
By convention, if $a\in\Z$, then
\[ a^0=\left\{ \begin{array}{ll}
0 \qquad&\text{if $a=0$,} \\
1 \qquad&\text{otherwise.}
\end{array} \right. \]

\smallskip

Let $n, p$ be positive integers. Then $\M_{p,n}(\Z)$ will denote the set of matrices with $p$ rows and $n$ columns over $\Z$ and $\M_n(\Z)$ is a simpler notation for $\M_{n,n}(\Z)$. If $P\in\M_{n,p}(\Z)$, then $P^\tran\in\M_{p,n}(\Z)$ denotes the transpose of $P$.

Let $n, p, q$ be positive integers. Let $u=(u_1, \dots, u_n)$ be a family of $n$ vectors in $\R^p$. Then, when there is no ambiguity, the same letter in uppercase, $U$, will denote the matrix in $\M_{p,n}(\R)$ whose columns are the coordinates of the vectors of $u$. Reciprocally, if $U$ is a matrix, then one can associate to $U$ a family of vectors $u$. If $I=(i_1, \dots, i_q)$ is a family of elements of $\zint1n$, then $u_I$ will be the family $(u_{i_1}, \dots, u_{i_q})$, and $U_I\in\M_{p,q}(\R)$ will be the associated matrix. We will sometimes use this notation for variables: $x_I:=(x_{i_1}, \dots, x_{i_q})$. Moreover, $u^\tran$ will be the family of $p$ vectors in $\R^n$ associated to $U^\tran$. If $I$ is an (unordered) subset of $\zint1n$ of size $q$, then $u_I$, \resp $U_I$, will denote the family $u_{(i_1, \dots, i_q)}$, \resp the matrix $U_{(i_1, \dots, i_q)}$, where $I=\{i_1, \dots, i_q\}$ and $i_1<i_2<\dots<i_q$.

If $u$ and $v$ are two families of elements of a same set of respective size $n$ and $m$, then $u\concat v$ will denote the concatenation of the two families, \ie,
\[ u\concat v=(u_1,\dots,u_n, v_1, \dots, v_m). \]
By extension, if $U\in\M_{p,n}(\R)$ and $V\in\M_{p,m}(\R)$ are two matrices, then $U\concat V$ is the matrix of $\M_{p,n+m}(\R)$ associated to $u\concat v$.

Let $\Ext \R^p=\bigoplus_{l\geq0}\Ext^l \R^p$ be the exterior algebra over $\R^p$. With the notations of the previous paragraphs, we will denote by $\mat u\in\Ext^n \R^p$, \resp $\mat u_I\in\Ext^q \R^p$, the exterior product $u_1\wedge\dots\wedge u_l$, \resp $u_{i_1}\wedge\dots\wedge u_{i_q}$. We endow $\Ext\R^p$ with the standard inner product: if $n$ and $m$ are two nonnegative integers, if $u$, $v$ are two families of respective sizes $n$ and $m$ in $\R^p$, then
\[ (\mat u, \mat v)=\begin{cases}
    \det(U^\tran V) & \text{if $n=m$,} \\
    0 & \text{otherwise.}
  \end{cases} \]
We have the associated norm $\norm{\mat u}:=(\mat u, \mat u)^{1/2}$.

Note that, if $e$ is the standard basis of $\R^p$, then the family $(\mat e_I)_{I\subset\zint1p, \card I=l}$ forms an orthonormal basis of $\Ext^l\R^p$.

We have the standard inclusion $\Z^p\ssubset\R^p$. If $F$ is a vector subspace of $\R^p$, then we denote by $F_\Z$ the sublattice $F\cap\Z^p$. If $u$ is a family of vectors in $\Z^p\ssubset\R^p$, then we denote by $\Zmod u$ the sub-$\Z$-module generated by $u$, and by $\vect{u}$ the subspace of $\R^p$ spanned by $u$. If $A$ is a sub-$\Z$-module of $\Z^p$, we will denote by $\norm A$ the covolume of $A$, \ie, the value $\norm{\mat f}$ where $f$ is any basis of $A$.

If $U\in\M_{p,n}(\Z)$, then $\Im(U)$, \resp $\ker(U)$, will always denote vector subspaces, \ie, the image, \resp kernel, of $U$ in $\R^p$, \resp $\R^n$.

\bigskip

Let $n$ and $p$ be two positive integers. Let $u$ be a family of $n$ vectors of $\Z^p\ssubset\R^p$. Let $r$ be the rank of $u$. We define the Symanzik polynomial of $u$ by
\[ \Sym_2(u;x):=\frac1{{\norm{\Zvect{u}}}^2}\sum_{\substack{ I\subset\zint1n \\ \card I=r}}\norm{\mat u_I}^2 x^{\compl I}. \]

Actually, one can replace the exponents $2$ by any nonnegative even integer $k$. We will call $k$ the \emph{order} of the polynomial (see Definition~\ref{5:defi:Symanzik_polynomials}). The order $0$ has some theoretical interests, as we will see for matroids. Moreover, Theorem~\ref{5:thm:variation} about stability holds for every orders. That is why, we will study the case of all the orders in this article. This widens the set of interesting polynomials. However, in this introduction, we set $k=2$.

The main result of Section~\ref{5:sec:duality} is the close link between Symanzik polynomials and the Kirchhoff theorem. To enlight this link, we introduce the Kirchhoff polynomial of $u$:
\[ \Kir_2(u;x):=\frac1{{\norm{\Zvect{u}}}^2}\sum_{\substack{ I\subset\zint1n \\ \card I=r}}\norm{\mat u_I}^2 x^I. \]
Then, the link is given by the duality theorem (Theorem~\ref{5:thm:duality}). If $v$ is a family of maximal rank of $\ker(U)_\Z$, then the theorem states that there exists an explicit factor $c\in\Q$ such that
\[ \Kir_2(v^\tran; x)=c\Sym_2(u; x). \]

Here is the generalization of the second Symanzik polynomial in a specific case (see Definition~\ref{5:defi:Symanzik_polynomials_with_parameters} for the general case). If $a$ is an element of $\Zmod u$, then the Symanzik polynomial of $u$ with parameter $a$ is
\[ \Sym_2(u; (a); x):=\Sym_2(u\concat(a); x\concat (0)). \]

The justification of the definitions of Symanzik polynomials without parameters and of Kirchhoff polynomials are given respectively in Examples~\ref{5:ex:first_Symanzik_polynomial} and~\ref{5:ex:Kirchhoff_s_thm_1}. However, the link between Symanzik polynomials with parameters and second Symanzik polynomials is harder to establish. We need Proposition~\ref{5:prop:Symanzik_and_orientations}. In this proposition, some orientations hidden in Symanzik polynomials with parameters appear. These orientations are already well-known. They are called \emph{chirotopes} in the theory of oriented matroid (see~\cites{FL78, GRA13}). One can think about them in the following way. We take a spanning $2$-forests $(F_1, F_2)$ of a graph. Then we can orient all edges that link $F_1$ and $F_2$, from $F_1$ to $F_2$. Much detail is given in Example~\ref{5:ex:second_Symanzik_polynomial}.

Propositions~\ref{5:prop:determinantal_formula},~\ref{5:prop:Symanzik_determinantal_formula} and~\ref{5:prop:Symanzik_with_parameters_determinantal_formula} state determinantal formul\ae\ similar to that of Equation~\eqref{5:eqn:intro:Kirchhoff_s_thm}. For instance, let $v$ be a basis of $\ker(U)$ belonging to $\Z^n$, $a\in\Zmod u$ and $b\in\Z^n$ such that $Ub=a$. Set $\~v:=v\concat(b)$. Then,
\[ \Sym_2(u; (a); x)=c\det(\~V^\tran X\~V), \]
for some factor $c$ we explicit in Proposition~\ref{5:prop:Symanzik_determinantal_formula}, where $X=\diag(x_1,\dots,x_n)$.

\subsection{Deletion-contraction formula and partial factorization} \label{5:subsec:intro:deletion-contraction_and_partial_factorization}

Among the many interesting properties of classical Symanzik polynomials, we will generalize two of them of particular interest.

The first one is the deletion-contraction formula (see, \eg,~\cites{BW10, Bro15}). In the classical case, if $G=(V,E)$ is a graph and if $e$ is edge of $G$, then
\begin{equation} \label{5:eqn:deletion_contraction_graph}
\psi_G(x)=x_e\psi_{G-e}(x_{E-e})+\psi_{G\contr\{e\}}(x_{E-e}),
\end{equation}
where $G-e$ is the graph without $e$ and $G\contr\{e\}$ is the graph obtained from $G$ by contracting $e$ (\ie, we removed the edge $e$ and identify the two extremities of $e$). We will provide a generalized version of this property for polynomials defined in this paper. In the context of families of vectors, the result will be the same except that deletion and contraction are more involved to define. For this reason, we do not give more details in this introduction and refer to Section~\ref{5:sec:deletion_contraction}.

The second formula we generalize is a partial factorization, used for example in~\cite{Bro15}. Let $G=(V,E)$ be a graph and let $\gamma=(V,F)$ be a spanning subgraph of $G$. If $G\contr\gamma$ is the graph obtained from $G$ by contracting every edge in $\gamma$, then
\begin{equation} \label{5:eqn:partial_factorization}
\psi_G(x)=\psi_\gamma(x_F)\psi_{G\contr\gamma}(x_{E\setminus F})+R^\psi_{\gamma,G}(x),
\end{equation}
where the degree of $R^\psi_{\gamma,G}(x)$ in $x_F$ is greater than the degree of $\psi_\gamma(x_F)$. See Proposition~\ref{5:prop:partial_factorization} to see precisely what we mean here. This time, the generalized formula is the same up to a factor. Once more, we refer to Section~\ref{5:sec:deletion_contraction} for more details.

\subsection{Geometrical aspects of Symanzik polynomials}

Following Kalai's idea in~\cite{Kal83}, Duval, Klivans and Martin stated in~\cite{DKM09} a generalization of Kirchhoff's theorem for simplicial complexes. This theorem counts the number of generalized spanning trees with some weights. In Section~\ref{5:subsec:forest}, we will call these spanning trees \emph{0-forests} (or, more exactly, we use the more convenient definition of spanning trees of Bernardi and Klivans in~\cite{BK16}). Readers who are not familiar with simplicial homology theory could consult the details in Section~\ref{5:subsec:forest}.

The definition of a $0$-forest is really simple: it is a maximal simplicial subcomplex without nontrivial cycles of maximal dimension. To be more precise, let $\Delta$ be a finite simplicial complex of dimension $d$, and let $\bound_\Delta$ be the $d$-th boundary operator of $\Delta$. Let $\Gamma$ be a subcomplex of $\Delta$ with the same $(d-1)$-skeleton, and let $\bound_\Gamma$ be the $d$-th boundary operator of $\Gamma$. Then, for an integer $\kappa$, $\Gamma$ is a \emph{$\kappa$-forest} of $\Delta$ if $\ker(\bound_\Gamma)=\{0\}$ and $\rk(\bound_\Delta)-\rk(\bound_\Gamma)=\kappa$. Forests in higher dimensions mostly behave like forests in graphs. Moreover, the two notions coincide on connected graphs. In this case, $\kappa$-forests are the spanning forests with $\kappa+1$ connected components.

The boundary operator $\bound_\Delta$ is a linear map. Let $U$ be the associated matrix for some natural bases of simplicial $d$-chains and of simplicial $(d-1)$-chains. It is natural to define the Symanzik polynomial, \resp Kirchhoff polynomial, of $\Delta$ as the one of $u$. Then $\Kir_2(\Delta;x)$ is what we get applying the higher dimensional Kirchhoff weighted theorem to $\Delta$ (see Theorem~\ref{5:thm:Kirchhoff_s}):
\[ \Kir_2(\Delta; x)=\card{\Tor(\Homl_{d-1}(\Delta))}^2\sum_{\text{$\Gamma$ $0$-forest of $\Delta$}}\card{\rquot{\Bound_{d-1}(\Delta)}{\Bound_{d-1}(\Gamma)}}^2\prod_{\delta\text{ $d$-face of }\Gamma}x_\delta, \]
where $\Bound_{d-1}(\Delta)$, \resp $\Bound_{d-1}(\Gamma)$, is the group of $(d-1)$-boundaries of $\Delta$, \resp of $\Gamma$, and $\Tor(\Homl_{d-1}(\Delta))$ denotes the torsion part of the $(d-1)$-th homology group of $\Delta$.

Unlike the case of graphs, setting $x=(1,\dots,1)$ does not give the number of $0$-forests. Finding this number is difficult in general: there exist examples where all the coefficients of the polynomial are not equal (look at Example~\ref{5:ex:possible_Symanzik_polynomials}). Knowing the value of the Kirchhoff polynomial for all order $k$ might help us.

However, the weighted number of $0$-forests for the order $k=2$ seems to be the most natural object to consider. The first reason is in~\cite{BK16}. The authors explain that this number counts $0$-forests of $\Delta$ taking into account a \emph{fitted orientation}. The second reason is that Kalai computed in~\cite{Kal83} the weighted number of $0$-forests of the complete $d$-dimensional simplicial complex over $N$ vertices. He obtains a pretty nice formula: $N^{\binom{N-2}d}$.

\smallskip

Now, for each facet $\delta$ of $\Delta$, assign to the variable $x_\delta$ a positive number $y_\delta\in\R_+$ that represents the volume of $\delta$. This is natural. For example, let $\S$ be a compact oriented manifold of dimension $d$, and let $\omega$ be a volume form on it. Let $\Delta$ be a triangulation of $\S$ (see Section~\ref{5:subsec:triangulable_space} for a rigorous definition of triangulation). Then $\omega$ induces a volume $y_\delta\in\R_+$ for each facet $\delta$ of $\Delta$. It is not hard to verify that
\[ \Sym_2(\Delta; y)=\Vol(\S). \]
This is more general. Set $\S$ the topological space of any $d$-dimensional finite CW-complex. We endow $\S$ with a diffuse measure $\mu$ (\ie, a measure which is $0$ on any subset homeomorphic to $[0;1]^{d-1}$). Let $\Delta$ be a triangulation of $\S$, and $y=(y_\delta)$ be the induced volume for each facet $\delta$. Then
\[ \Sym_2(\Delta; y)\text{ does not depend on the chosen triangulation $\Delta$.} \]
The proof essentially consists on defining directly $\Sym_2(\S; \mu)$, which will be the value we get for any triangulation.

The case of a metric graph $\G$ is of particular interest. $\Sym_2(\G; \mu)$ (with $\mu$ induced by the metric) is the volume of the Jacobian torus of $\G$ introduced by Kotani and Sunada in~\cite{KS00} (see Example~\ref{5:ex:metric_graph}).

\subsection{Stability theorem}

Let $v$ be any free family of $n$ elements of $\R^p$. Let $a$ be an element of $\R^p$ independent from $v$ and let $w=v\concat(a)$. Let $y(t)=(y_1(t), \dots, y_n(t))$ be a family of $n$ elements of $\R_+$ depending on some parameter $t$. Let $Y(t)$ be the diagonal matrix $\diag(y_1(t), \dots, y_n(t))$. Let $\fF(t)$ be an element of $\M_n(\R)$ depending on $t$ whose entries are bounded. Theorem~\ref{5:thm:variation} states, with different notations, that there exists two constant $c, C\in\R_+$ such that, for any $t$ verifying $y_1(t), \dots, y_n(t)\geq C$,
\begin{equation} \label{5:eqn:intro:stability}
\left|\frac{\det(W^\tran Y(t)W)}{\det(V^\tran Y(t)V)}-\frac{\det(W^\tran(Y(t)+\fF(t))W)}{\det(V^\tran(Y(t)+\fF(t))V)}\right|\leqslant c.
\end{equation}

This result generalizes Theorem 1.1 of~\cite{Ami19}. Notice that both ratios are of degree $1$ in $y$. Moreover, even if the ratios were of degree zero, terms like $y_1/y_2$ are unbounded at infinity. Thus this property is quite rare among ratios of polynomials and it is somehow surprising that Symanzik rational functions verifies it.

Let us roughly explain what this theorem implies for Symanzik polynomials. For the sake of simplicity, we assume that $\fF$ is the diagonal matrix $\diag(\f_1(t), \dots, \f_n(t))$ for all $t$. One can think about the left-hand member of~\eqref{5:eqn:intro:stability} as the absolute value of
\[ \frac{\Sym_2(\Delta; (b); y)}{\Sym_2(\Delta; y)}-\frac{\Sym_2(\Delta; (b); y+\f)}{\Sym_2(\Delta; y+\f)}, \]
for some simplicial complex $\Delta$ and some parameter $b$ (which is a $(d-1)$-boundary of $\Delta$). Thus, the first ratio simply corresponds to the ratio of the two Symanzik polynomials. The second one corresponds to the same ratio after a bounded deformation of the geometry of $\Delta$. Hence, such a deformation only induces a variation of the ratio of the two Symanzik polynomials that is bounded independently of how great are the volumes of the facets.

Surprisingly, in order to prove this stability theorem, we crucially use a theorem describing the connected components of the exchange graph of some matroid. This theorem is quite interesting by itself, and for this reason, it seems important to us to give some details about it in this introduction. More can be found in the corresponding section.

\subsection{Exchange graph of a matroid}

Let $T$ and $T'$ be two spanning trees on some connected graph $G$. There is a well-known property: for any edge $e$ of $T$, there exists an edge $e'$ of $T'$ such that $T-e+e'$ and $T'-e'+e$ are still spanning trees of $G$. We will consider another kind of exchanges. Let $F$ and $F'$ be two spanning forests of $G$. We authorize the following operation: if there exists $e\in F$ such that $F'+e$ is still a forest, one can remove $e$ from $F$ and add it to $F'$. Following the same rule, one can also take an edge from $F'$ and add it to $F$. Then, one can wonder which pairs of forests can be obtained from $F$ and $F'$. Moreover, what happens if we add the constraint that the exchanges have to alternate (first from $F$ to $F'$, then from $F'$ to $F$, etc.)? We answer these questions in this article, generalizing the partial answer of~\cite{Ami19}*{Theorem 2.13}.

The natural objects to consider for this kind of questions are matroids. In fact, forests of a graph, but also of a simplicial complex, are encoded by the independents of the matroid associated to this graph. In this setting, the previous questions can be reformulated thanks to the \emph{exchange graph} of the matroid. Let $\Ma$ be a matroid and let $\Ind$ be its set of independents. The exchange graph $\G=(\V,\gE)$ associated to $\Ma$ is the graph with vertex set $\V:=\Ind\times\Ind$ and edge set $\gE$ such that two vertices $(I_1,I_2)$ and $(I'_1, I'_2)$ are adjacent if there exists $i\in\E$ such that either, $I'_1=I_1+i$ and $I'_2=I_2-i$, or, $I'_1=I_1-i$ and $I'_2=I_2+i$. The first question now corresponds to: \enquote{under which conditions two vertices $(I_1,I_2)$ and $(I'_1,I'_2)$ of $\V$ are in the same connected component of $\G$ ?} For the second question, one has to restrict $\G$ to some appropriate induced subgraph (see Section~\ref{5:sec:exchange_graph}).

A trivial condition in both cases is that $I_1\multicup I_2$ (the union with multiplicities) must equal $I'_1\multicup I'_2$. But this is not sufficient. For example, there cannot be any exchange between two bases (\ie, every pair of bases is isolated in $\G$) because bases have maximum number of elements among independent sets. More generally, one can find other isolated pairs of independents. If such a pair $(A,B)$ is included in some larger pair $(I,J)$ (in the sense $A\subset I$ and $B\subset J$), no exchange will happen between $A$ and $B$. This new kind of constraints is encoded by the \emph{maximal codependent pair (or MCP)} of $(I,J)$. One can define $\MCP(I,J)$ as the largest vertex $(A,B)$ isolated in $\G$ and included in $(I,J)$ (we will prove that such a vertex always exists). In fact, what we have introduced here is sufficient to classify the connected components of $\G$.

\begin{theorem*}[Corollary~\ref{5:cor:exchange_graph_2}]
Let $(I,J)$, $(I',J')$ be two arbitrary vertices of $\G$. Then $(I,J)$ and $(I',J')$ are in the same connected component of $\G$ if and only if
\[ I\multicup J = I' \multicup J'\qquad\text{and}\qquad\MCP(I, J) = \MCP(I', J'). \]
\end{theorem*}

For the second question, where we force the exchanges to alternate, the same result holds (up to some isolated vertices, cf. Theorem~\ref{5:thm:exchange_graph}).

\subsection{Structure of the paper}

Section~\ref{5:sec:duality} is devoted to the generalization in an abstract setting, and to the important duality theorem (Theorem~\ref{5:thm:duality}) linking Symanzik polynomials to the well-known Kirchhoff theorem. In the third section we develop applications of Symanzik polynomials to geometry. In Section~\ref{5:sec:deletion_contraction} we generalize some other properties of these polynomials thanks to matroids over hyperfields. The two last sections follow Sections 2 and 3 of~\cite{Ami19} and generalize the results of this paper to our setting. Section~\ref{5:sec:exchange_graph} describes a combinatorial result about what we call the exchange graph of a matroid, which is needed for the last section. In this last section, we state and prove the stability theorem.

\subsection*{Acknowledgements}

We whish to express our thanks to our Ph.D.\ supervisor Omid Amini for suggesting the problem, for his many advices and for careful and patient reading and correcting of the paper. We gratefully acknowledge the many references given by Adrien Kassel, Vincent Rivasseau and Alan Sokal, and the helpful suggestions of Spencer Bloch and Razvan Gurau.

%%%
\section{Kirchhoff and Symanzik polynomials and duality}
\label{5:sec:duality}

\subsection{Preliminaries}
\label{5:subsec:notations}

We begin this section setting up some more tools in order to deal with Symanzik polynomials of higher orders. Note that some important notations and conventions, which we will use all through the paper, have already been introduced in Section~\ref{5:subsec:intro:generalization}.

In the whole article, $k$ will always be any \emph{nonnegative even} integer.

Let $p$ be a positive integer. If $k$ is nonzero, let $a_1=(a_{1,1}, \dots, a_{1,p}), \dots, a_k$ be $k$ vectors of $\R^p$. We define the \emph{standard $k$-multilinear symmetric product} to be
\[ (a_1, \dots, a_k)_k:=\sum_{i=1}^p\prod_{j=1}^ka_{j,i}. \]
If $y$ is a family of $p$ variables, then we define the \emph{$k$-multilinear symmetric product associated to $y$} by
\[ (a_1, \dots, a_k)_{k,y}:=\sum_{i=1}^p\bigg(\prod_{j=1}^ka_{j,i}\bigg)y_i. \]
We will omit the index $k$ when the number of variables is clear.

Let us define what we call matrix of higher order, and some basic operations on it. Hyperdeterminants were first discovered by Arthur Cayley in 1843 (see~\cite{Cay1849}).

Let $\A$ be any commutative ring.
A \emph{matrix of order $k$ and size $(n_1,\dots, n_k)$ on $\A$}, where $k$ is a positive integer, and where the \emph{size} is a $k$-tuple of positive integers, is a family of elements of $\A$ indexed by $\zint1{n_1}\times\dots\times\zint1{n_k}$. We can naturally sum two matrices of the same size by summing corresponding entries.

Let $C=(c_{i_1,\dots,i_k})_{(i_1,\dots,i_k)\in\zint1{n_1}\times\dots\times\zint1{n_k}}$ be such a matrix. Let $l\in\zint1k$ and $P=(p_{i, j})\in\M_{n_l,m}(\A)$ be a usual matrix. Then the \emph{(right) multiplication of $C$ by $P$ along the $l$-th direction}, denoted by $C\mmult lP$, is the $k$-dimensional matrix $B$ of size
\[ \size(B)=(n_1, \dots, n_{l-1}, m, n_{l+1}, \dots, n_k), \]
verifying,
\[ b_{i_1,\dots,i_k}=\sum_{a=1}^{n_l} c_{i_1,\dots,a,\dots,i_k}\cdot p_{a,i_l}, \]
where $a$ is the $l$-th index of $c$.

Let $\mhyp_n^k(\A)$ denote the set of \emph{hypercubic matrices of order $k$ and of size $n$}, \ie, of matrices having size $(\underbrace{n, \dots, n}_{\text{$k$ times}})$. Let $C\in\mhyp_n^k(\A)$. We define the \emph{hyperdeterminant of $C$} by
\vspace{-1.5em}
\[ \det(C):=\frac1{n!}\sum_{\tau_1, \dots, \tau_k\in\perm n}\prod_{j=1}^k\sgn(\tau_j)\prod_{a=1}^n c_{\tau_1(a),\dots,\tau_k(a)}, \]
where $\perm n$ is the set of permutations on $\zint1n$, and $\sgn(\tau)$ denotes the signature of the permutation $\tau$.
Notice that this definition coincides with the usual one for $k=2$. Moreover, the determinant would be $0$ if $k$ was odd. Let $D$ be the diagonal hypermatrix $D:=\diag^k(a_1, \dots, a_n)$ defined by
\[ D_{i_1, \dots, i_k}=\begin{cases}
a_{i_1} & \text{if $i_1=\dots=i_k$}, \\
0 & \text{otherwise}.
\end{cases} \]
Then $\det(D)=a_1\cdots a_n$. The determinant is multiplicative: if $l\in\zint1k$ and if $P\in\M_n(\A)$, then
\[ \det(C\mmult lP)=\det(C)\det(P). \]

\medskip

Let $u^1, \dots, u^k$ be $k$ families of vectors in $\R^p$ of respective size $l_1, \dots, l_k$. If $(\wcdot,\dots,\wcdot)_*$ is any $k$-multilinear symmetric product from $(\R^p)^k$ to a commutative ring, we can extend it to $\Ext\R^p$ by
\[ (\mat u^1, \dots, \mat u^k)_*=\begin{cases}
  \det((u^1_{i_1},\dots,u^k_{i_k})_*)_{1\leq i_1,\dots,i_k\leq l_1} & \text{if $l_1=\dots=l_k$,} \\
  0 & \text{otherwise.}
\end{cases} \]
One can replace $\R^p$ by any real vector space. We allow ourselves to use standard terminologies and notations of inner products in this case. For instance, $\norm{\mat u}_*^k$ denoted $(\mat u, \dots, \mat u)_*$.

For instance, if $l_1=\dots=l_k=n$, then
\[ (\mat u^1, \mat u^2)_x=\det((U^1)^\tran X(U^2)), \]
where $X=\diag(x_1, \dots, x_n)$, and,
\[ (\mat u^1, \dots, \mat u^k)_x=\det(X\mmult 1U^1\cdots\mmult kU^k), \]
where $X$ is the diagonal hypercubic matrix $\diag^k(x_1, \dots, x_n)$ (and $U^i$ is the matrix associated to $u^i$).

\bigskip

We state a simple but very useful lemma.

\begin{lemma} \label{5:lem:u=fv}
Let $r, p, n$ be three positive integers and $U\in\M_{p,n}(\Z)$ be a matrix of rank $r$. Let $F\in\M_{p,r}(\Z)$. Assume that $\Zmod u\subset \Zmod f$. Then, there exists a unique $V\in\M_{r,n}(\Z)$ such that $U=FV$. Moreover, we have the following points.
\begin{enumerate}
\item \label{5:lem:u=fv_1} The following properties are equivalent.
  \begin{itemize}
  \item $f$ is a basis of $\Zmod{u}$.
  \item $v^\tran$ is a basis of $\Zvect{u^\tran}$.
  \end{itemize}
  Furthermore, if they hold, $v$ generates $\Z^r$ and $\norm{\mat u_I}=\norm{\Zmod u}\cdot\norm{\mat v_I}$ for any subset $I$ of $\zint1n$ of size $r$.
\item The following properties are equivalent.
  \begin{itemize}
  \item $f$ is a basis of $\Zvect u$.
  \item $v^\tran$ is a basis of $\Zmod{u^\tran}$.
  \end{itemize}
  Furthermore, if they hold, $\norm{\mat u_I}=\norm{\Zvect u}\cdot\norm{\mat v_I}$ for any subset $I$ of $\zint1n$ of size $r$.
\end{enumerate}
\end{lemma}

\begin{proof}
The first part of the lemma is clear. Notice that columns of $F$, \resp rows of $V$, have to be free. Let us prove Point~\eqref{5:lem:u=fv_1}.

If $f$ is a basis of $\Zmod{u}$, then there exists $V'\in\M_{n,r}(\Z)$ such that $F=UV'=FVV'$. Since columns of $F$ are free, $VV'=\Id_r$. Thus, $v$ generates $\Z^r$. Let $w$ be such that $w^\tran$ is a basis of $\Zvect{u^\tran}$ (which equals $\Zvect{v^\tran}$). This last space contains $v^\tran$. Therefore, there exists a square matrix $G$ such that $V=GW$. Then,
\[ 1=\det(VV')=\det(GWV')=\det(G)\det(WV'), \]
which implies that $G$ is invertible, and then that $v^\tran$ is in fact a basis of $\Zvect{v^\tran}$.

Let us prove the converse. If $v^\tran$ is a basis of $\Zvect{u^\tran}$, then $v^\tran$ can be completed into a basis ${\~v}^\tran$ of $\Z^n$. Set $V'=\big({\~V}^{-1}\big)_
{\zint1r}$. As $UV'=FVV'=F$, $\Zmod{f}\subset\Zmod{u}$. The other inclusion follows from the existence of $V$. Thus, $f$ is in fact a basis of $\Zmod u$.

Moreover,
\[ \norm{\mat u_I}^2=\det(U_I^\tran U_I)=\det(V_I^\tran F^\tran FV_I). \]
But $V_I$ and $F^\tran F$ are square matrices. Thus, since $f$ is a basis of $\Zmod u$,
\[ \det(V_I^\tran F^\tran FV_I)=\det(V_I)^2\det(F^\tran F)=\norm{\mat v_I}^2\norm{\Zmod u}^2. \]

The second point can be proven in a similar way.
\end{proof}

\smallskip

Some specific notations will be introduced later but we can already deal with the heart of the subject.

\subsection{Kirchhoff polynomials}

In this article, Kirchhoff polynomials are a generalization of polynomials appearing in the weighted Kirchhoff's matrix tree theorem, whereas Symanzik polynomials generalize first and second Symanzik polynomials better known in Physics (see Examples~\ref{5:ex:Kirchhoff_s_thm_1},~\ref{5:ex:first_Symanzik_polynomial} and~\ref{5:ex:second_Symanzik_polynomial}, Theorem~\ref{5:thm:Kirchhoff_s} and the introduction for more details). They are dual in a very precise way as we will show at the end of this subsection, \cf Theorem~\ref{5:thm:duality}.

We will fix some objects for the rest of the section. Let $p, n$ be two positive integers, $u$ be a family of $n$ elements in $\Z^p$, and $r:=\rk(u)$ be the rank of this family. Recall that $k$ is an even nonnegative integer.

\begin{definition} \label{5:defi:Kirchhoff_polynomials}
The \emph{Kirchhoff polynomial of order $k$ of $u$} is defined by
\[ \Kir_k(u;x):=\frac1{{\norm{\Zvect{u}}}^k}\sum_{\substack{ I\subset\zint1n \\ \card I=r}}\norm{\mat u_I}^k x^I. \qedhere \]
\end{definition}

\begin{remark}
\label{5:rem:another_definition}
Another definition would be obtained replacing $\norm{\Zvect{u}}$ by $\norm{\Zmod{u}}$. A priori, no definition is better than the other. We choose our definition in order to get determinantal formula without any factor (see Propositions~\ref{5:prop:determinantal_formula} and~\ref{5:prop:m_determinantal_formula}). As a counterpart, factors appear in the duality theorem~\ref{5:thm:duality}. This choice is also justified by Remark~\ref{5:rem:arithmetic_matroid}.
\end{remark}

Let $I$ be a subset of $\zint1n$ of size $r$, let $v$ be such that $v^\tran$ is a basis of $\Zmod{u^\tran}$, and let $F\in\M_{p,r}(\Z)$ be the unique matrix such that $U=FV$. By Lemma~\ref{5:lem:u=fv}, $f$ is a basis of $\Zvect{u}$. If $\mat u_I$ is nonzero, then $\Zmod{u_I}$ is a submodule of $\Zvect{u}$ of maximal rank, and we therefore have all the following equalities:
\begin{equation}
\label{5:eqn:norm_to_det}
\frac{\norm{\mat u_I}}{\norm{\Zvect{u}}}=\norm{\mat v_I}=\abs{\det(V_I)}=\card{\rquot{\Zvect{u}}{\Zmod{u_I}}}=\abs{(\mat v^\tran, \mat e_I)},
\end{equation}
where $e$ is the standard basis of $\Z^n$.
In particular, the coefficients of $\Kir_k$ are in $\Z$.

Moreover, $v^\tran$ could be any basis of $\Zmod{u^\tran}$. But Equation~\eqref{5:eqn:norm_to_det} gives a formula of the coefficients depending on $v$ alone. Thus,
\begin{equation}
\label{5:eqn:Kir_depends}
\text{$\Kir_k(u;x)$ only depends on $\Zmod{u^\tran}$.}
\end{equation}

\begin{remark}
\label{5:rem:PID}
Using $\det(V_I)$ instead of $\norm{\mat u_I}/\norm{\Zvect u}$ in Definition~\ref{5:defi:Kirchhoff_polynomials}, it is possible to define Kirchhoff polynomials for $k$ odd, or even for $u$ a family of vectors over $A^n$ where $A$ is any PID. However, one has to be careful because the $k$-th power of an element of $A^*$ could be different from $1$. The polynomial depends on the chosen $v$ (up to a $k$-th power of an invertible of $A^*$). Moreover, one has to add some artificial signs to the Symanzik polynomial (see Definition~\ref{5:defi:Symanzik_polynomials}) in order to make Theorem~\ref{5:thm:duality} about the duality true. For more details, we refer to~\cite{Piq19-master-thesis}.
\end{remark}

\begin{remark}
\label{5:rem:arithmetic_matroid}
One can also use $\card{\rquot{\Zvect{u}}{\Zmod{u_I}}}$ instead of $\norm{\mat u_I}/\norm{\Zvect u}$ in Definition~\ref{5:defi:Kirchhoff_polynomials}. Then the definition makes sense for any order $k$. As a counterpart, determinantal formul\ae\ as in Proposition~\ref{5:prop:determinantal_formula} are not true anymore. However, for $k=1$, we recover the main example of \emph{arithmetic matroid} for the group $G=\Z^n$: see~\cite{DAM11}*{Subsection 2.4}.
\end{remark}

Kirchhoff polynomials of order $2$ have a more computable definition under the condition $r=p$.

\begin{proposition} \label{5:prop:determinantal_formula}
If $r=p$ (or equivalently if $u^\tran$ is free), then
\[ \Kir_2(u;x)=\norm{\mat u^\tran}_x^2=\det(UXU^\tran), \]
where $X\in\M_n(\Z[x])$ is the diagonal matrix $\diag(x_1,\dots,x_n)$.
\end{proposition}

\begin{proof}
The second equality comes directly from the definitions. Let $e$ be the standard basis of $\R^n$. As the standard orthonormal basis of $\Ext^r\R^n$ (for the standard inner product) is an orthogonal basis for $(\wcdot, \wcdot)_x$, we obtain,
\begin{align*}
\norm{\mat u^\tran}_x^2
  &= \sum_{\substack{I\subset\zint1n \\ \card I=r}}(\mat u^\tran,\mat e_I)^2\norm{\mat e_I}_x^2 \\
  &= \sum_{\substack{I\subset\zint1n \\ \card I=r}}\norm{\mat u_I}^2x^I.
\end{align*}

It remains to check that $\norm{\Zvect{u}}=1$, but this is true because $\Zvect{u}=\Z^p$.
\end{proof}

For orders larger than 2, one can obtain a similar formula.

\begin{proposition} \label{5:prop:m_determinantal_formula}
If $r=p$ (or equivalently if $u^\tran$ is free), then
\[ \Kir_k(u;x)=\norm{\mat u^\tran}_{k,x}^k=\det\big(X\mmult1U^\tran\mmult2\cdots\mmult kU^\tran\big), \]
where $X$ is the diagonal hypercubic matrix $\diag^k(x_1,\dots,x_n)$.
\end{proposition}

\begin{proof}
The same proof works if one replaces the computation by
\begin{align*}
\norm{\mat u^\tran}_{k,x}^k
  &= \sum_{\substack{I\subset\zint1n \\ \card I=r}}(\mat u^\tran,\mat e_I)^k\norm{\mat e_I}_{k,x}^k \\
  &= \sum_{\substack{I\subset\zint1n \\ \card I=r}}\norm{\mat u_I}^kx^I. \qedhere
\end{align*}
\end{proof}

\begin{example} \label{5:ex:practical_computation}
How to compute the Kirchhoff polynomial of any family $u$ with the determinantal formula? Take $v$ a family such that $v^\tran$ is a basis of $\Zmod{u^\tran}$. Let $F$ be the unique matrix such that $U=FV$ (it is not necessary to compute $f$). By Lemma~\ref{5:lem:u=fv}, for any subset $I\subset\zint1n$ of size $r$,
\[ \norm{\mat u_I}=\norm{\mat v_I}\norm{\Zvect u}. \]
As $\Zvect v=\Z^r$, one can rewrite last equation as
\[ \frac{\norm{\mat u_I}}{\norm{\Zvect u}}=\frac{\norm{\mat v_I}}{\norm{\Zvect v}}. \]
Thus,
\[ \Kir_k(u; x)=\Kir_k(v; x)=\norm{\mat v^\tran}_{k,x}^k. \]
Note that, if instead of taking a basis of $\Zmod{u^\tran}$, one takes a basis of $\Zvect{u^\tran}$, you will divide all coefficients by a common divisor equal to
\[ \bigg(\frac{\norm{\Zmod{u}}}{\norm{\Zvect{u}}}\bigg)^k. \]
We will see a meaning of this factor in Remark~\ref{5:rem:thm:Kirchhoff_s}.
\end{example}

\begin{example} \label{5:ex:Kirchhoff_s_thm_1}
Let us now explain more precisely the link between Kirchhoff as mathematician and Kirchhoff polynomials. Let $G=(V,E)$ be a connected graph with vertex set $V$ of size $p$ and edge set $E$ of size $n$. Suppose that vertices are labelled from $1$ to $p$ and edges from $1$ to $n$. By abuse of notation, the same letter will denote a vertex, \resp an edge, and its label. Similarly, a set of edges could denote a set of numbers. Let $Q=(q_{v,e})\in\M_{p,n}(\Z)$ be an incidence matrix of $G$, \ie, put an orientation on edges of $G$ and set, for each $v\in V$ and each $e\in E$,
\[ q_{v,e}:=\left\{ \begin{array}{ll}
0\quad&\text{if $e$ is a loop,} \\
1\quad&\text{if $v$ is the head of $e$,} \\
-1\quad&\text{if $v$ is the tail of $e$,} \\
0\quad&\text{if $v$ and $e$ are not incident.}
\end{array} \right. \]
Let $i\in\zint1p$ be any number, and $U$ be the matrix $Q$ where we deleted the $i$-th row. The well-known Kirchhoff's matrix-tree theorem, in its weighted form, states that
\[ \det(UXU^\tran)=\sum_{I\in\trees}x^I, \]
where $X=\diag(x_1,\dots,x_n)$, and $\trees\subset\P(E)$ is the family of all subsets $I\subset E$ that verify that the spanning subgraph of $G$ with edge set $I$ is a spanning tree of $G$. Then Proposition~\ref{5:prop:determinantal_formula} implies that
\[ \Kir_2(u;x)=\sum_{I\in\trees}x^I. \]

In fact, in this very special case, \( \Kir_k(u;x) \) does not depend on $k$. Moreover one can see that $\Zmod{q^\tran}=\Zmod{u^\tran}$. Thus, by~\eqref{5:eqn:Kir_depends}, we even have
\begin{equation} \label{5:eqn:Kirchhoff_s_thm_1}
\Kir_k(q;x)=\sum_{I\in\trees}x^I.
\end{equation}
We will see later, in Theorem~\ref{5:thm:Kirchhoff_s}, a generalization of Kirchhoff's theorem to the case of finite simplicial complexes.
\end{example}

\subsection{Symanzik polynomials}

Symanzik polynomials have a similar definition, but note the complement in the last exponent.

\begin{definition} \label{5:defi:Symanzik_polynomials}
The \emph{Symanzik polynomial of order $k$ of $u$} is defined by
\[ \Sym_k(u;x):=\frac1{{\norm{\Zvect{u}}}^k}\sum_{\substack{ I\subset\zint1n \\ \card I=r}}\norm{\mat u_I}^k x^{\compl I}. \qedhere \]
\end{definition}

\begin{remark} \label{5:rem:Kirchhoff_and_Symanzik}
Kirchhoff and Symanzik polynomials are so similar that one can easily define ones from the others thanks to following formul\ae.
\begin{gather*}
\Sym_k(u;x)=x_1\cdots x_n\Kir_k(u;(x_1^{-1}, \dots,x_n^{-1})), \\
\Kir_k(u;x)=x_1\cdots x_n\Sym_k(u;(x_1^{-1}, \dots,x_n^{-1})). \qedhere
\end{gather*}
\end{remark}

\begin{example} \label{5:ex:first_Symanzik_polynomial}
What is the link between Symanzik polynomials of Definition~\ref{5:defi:Symanzik_polynomials} and the first Symanzik polynomial of the introduction (Equation~\eqref{5:eqn:first_Symanzik_polynomial})? Let $G$ be a graph as defined in Example~\ref{5:ex:Kirchhoff_s_thm_1}. Set $k=2$. We have seen in Example~\ref{5:ex:Kirchhoff_s_thm_1} that, using the same notations,
\[ \Kir_k(q;x)=\sum_{I\in\trees}x^I. \]
Using Remark~\ref{5:rem:Kirchhoff_and_Symanzik}, we obtain
\begin{align*}
\Sym_k(q;x) &= x_1\cdots x_n\Kir_k(q;(x_1^{-1}, \dots, x_n^{-1})) \\
 &= \sum_{I\in\trees}x^{\compl I},
\end{align*}
which is exactly the first Symanzik polynomial.
\end{example}

Symanzik polynomials also have a determinantal formula, \cf Proposition~\ref{5:prop:Symanzik_determinantal_formula}. As we will show, this formula will be a direct consequence of the duality theorem~\ref{5:thm:duality}. Therefore, we postpone the statement and the proof of the formula to the end of next subsection.

\subsection{Duality Theorem}

We can now state the duality theorem.

\begin{theorem}[Duality] \label{5:thm:duality}
Let $q$ be a positive integer and $v$ be a family of $n$ vectors in $\Z^q$ such $v^\tran$ spans $\ker(U)$. Then
\begin{equation} \label{5:eqn:duality_1}
\frac1{a^k}\Sym_k(u;x)=\frac1{b^k}\Kir_k(v;x),
\end{equation}
where
\begin{align*}
  a &:= \frac{\norm{\Zmod{u}}}{\norm{\Zvect{u}}} = \card{\rquot{\Zvect{u}}{\Zmod{u}}}\text{ and} \\
  b &:= \frac{\norm{\Zmod{v}}}{\norm{\Zvect{v}}} = \card{\rquot{\Zvect{v}}{\Zmod{v}}}.
\end{align*}
\end{theorem}

Note that $\ker(U)$ denotes a vector subspace of $\R^n$, and that $\ker(U)=\vect{u^\tran}^\perp$.

\begin{proof}
Set $\~u$ a family of vectors such that ${\~u}^\tran$ is a basis of $\Zvect{u^\tran}$. By Lemma~\ref{5:lem:u=fv}, one can find a matrix $U'\in\M_{r,n}(\Z)$ such that $\~UU^{\prime\tran}=\Id_r$. Similarly, set $\~v$ such that ${\~v}^\tran$ is a basis of $\Zvect{v^\tran}$, and $V'\in\M_{n-r,n}(\Z)$ such that $\~VV^{\prime\tran}=\Id_{n-r}$.

Let $I$ be a subset of $\zint1n$ of size $r$. The coefficient of $x^{\compl I}$ in the left-hand side of~\eqref{5:eqn:duality_1} is
\[ \Big(\frac{\norm{\Zvect{u}}}{\norm{\Zmod{u}}}\frac{\norm{\mat u_I}}{\norm{\Zvect{u}}}\Big)^k. \]
Therefore, by Lemma~\ref{5:lem:u=fv}, this coefficient is equal to
$\norm{\mat{\~u}_I}^k$. Applying the same argument on the right-hand side, it
remains to prove that
\[ \norm{\mat{\~u}_I}=\norm{\mat{\~v}_{\compl I}}. \]

We will deduce this last equality from
\begin{equation} \label{5:eqn:thm:duality_1}
\pbmatrix{\~U_I}{\~U_{\compl I}}{V'_I}{V'_{\compl I}}
\pbmatrix{\Id_r}{\~V_I^\tran}{0}{\~V_{\compl I}^\tran}=\pbmatrix{\~U_I}{0}{*}{\Id_s},
\end{equation}
where $s:=n-r$, and from
\[ \pbmatrix{U'_I}{U'_{\compl I}}{\~V_I}{\~V_{\compl I}}
\pbmatrix{\~U_I^\tran}{0}{\~U_{\compl I}^\tran}{\Id_s}=
\pbmatrix{\Id_r}{*}{0}{\~V_{\compl I}}. \]
Indeed, this implies
\begin{gather*}
\norm{\mat{\~u}^\tran\wedge\mat{v'}^\tran}\cdot\norm{\mat{\~v}_{\compl I}}=\norm{\mat{\~u}_I}, \\
\norm{\mat{u'}^\tran\wedge\mat{\~v}^\tran}\cdot\norm{\mat{\~u}_I}=\norm{\mat{\~v}_{\compl I}}.
\end{gather*}
Since all these factors are nonnegative integers. We must have
\[ \norm{\mat{\~u}_I}=\norm{\mat{\~v}_{\compl I}}, \]
which concludes the proof.
\end{proof}

\begin{remark}
In the proof above, the equality of coefficients implies that
\[ \frac{\norm{\mat u_I}}{\norm{\Zmod{u}}} = \frac{\norm{\mat v_{\compl I}}}{\norm{\Zmod{v}}}. \]
In fact, one can prove something stronger, namely,
\[ \rquot{\Zmod{u}}{\Zmod{u_I}} \simeq \rquot{\Zmod{v}}{\Zmod{v_{\compl I}}}. \qedhere \]
\end{remark}

We need a lemma for the proof of Theorem~\ref{5:thm:variation}. We state this lemma here because its proof is very similar to that of Theorem~\ref{5:thm:duality}. If $I$ is a subset of $\zint1n$ and if $\underline I$ is an ordering of the elements of $I$, then we write $\epsilon(\underline I):=(\mat e_{\underline I}, \mat e_I)$, \ie, $(-1)^\iota$ where $\iota$ is the number of inversions in $\underline I$.

\begin{lemma} \label{5:lem:duality_with_signs}
Let $p,m$ be two positive integers. Let $\~u$ and $\~v$ be such that $\~u^\tran$ and $\~v^\tran$ are two free families of vectors in $\R^m$ of respective size $p$ and $m-p$. Assume that $\~U{\~V}^\tran=0$. Let $I$ be a subset of $\zint1m$ of size $p$. Then
\[ \det(\~U_I)=\lambda\,\epsilon(I\concat\compl I)\det(\~V_{\compl I}), \]
where $\lambda\in\R^*$ does not depend on $I$.
\end{lemma}

\begin{proof}
We define $V'\in\M_{m-p,m}(\R)$ as in the proof of Theorem~\ref{5:thm:duality}, \ie, such that $\~V{V'}^\tran=\Id_{m-p}$. We still get Equation~\eqref{5:eqn:thm:duality_1}. By reordering columns and by transposing, this equation leads to
\[ \epsilon(I\concat\compl{I})\det(\~U^\tran\concat V'^\tran)\det(\~V_{\compl I})=\det(\~U_I). \]
Then set $\lambda=\det(\~U^\tran\concat V'^\tran)$ to conclude the proof ($\lambda$ cannot be $0$ since there exists an $I$ such that $\det(\~U_I)\neq0$).
\end{proof}

Let us now extend determinantal formul\ae{} to Symanzik polynomials. From the determinantal formula for Kirchhoff polynomials (Propositions~\ref{5:prop:determinantal_formula} and~\ref{5:prop:m_determinantal_formula}) and from the duality theorem (Theorem~\ref{5:thm:duality}), we immediately obtain the following formula.

\begin{proposition} \label{5:prop:Symanzik_determinantal_formula}
If $v$ is a family of vectors in $\Z^{n-r}$ such that $v^\tran$ is a basis of the vector subspace $\ker(U)$, then
\[ \Sym_k(u;x)=\Big(\frac{a}{b}\Big)^k\norm{\mat v^\tran}_{k,x}^k \]
where $a=\norm{\Zmod u}/\norm{\Zvect u}$ and $b=\norm{\Zmod v}/\norm{\Zvect v}$ are defined as in Theorem~\ref{5:thm:duality}.
\end{proposition}

\subsection{Symanzik polynomials with parameters}

We now wish to generalize the second Symanzik polynomials defined in the introduction (Equation~\eqref{5:eqn:second_Symanzik_polynomial}). Actually, one can naturally add more than one parameter.

\begin{definition} \label{5:defi:Symanzik_polynomials_with_parameters}
Let $l$ be a nonnegative integer and let $w=(w_1, \dots, w_l)$ be a family of elements in $\vect u$. The \emph{Symanzik polynomial of order $k$ of $u$ with parameters $w_1, \dots, w_l$} is defined by
\[ \Sym_k(u;w;x):=\frac1{{\norm{\Zvect{u}}}^k}\sum_{\substack{ I\subset\zint1n \\ \card I=r-l}}\norm{\mat u_I\wedge\mat w}^k x^{\compl I}. \qedhere \]
\end{definition}

\begin{remark}
In the case $w_1, \dots, w_l\in\Zmod u$, one simply has
\[ \Sym_k(u;w;x)=\Sym_k(u\concat w;x\concat(\underbrace{0,\dots,0}_{\text{$l$ times}})). \qedhere \]
\end{remark}

Once more, there is a determinantal formula.

\begin{proposition}
\label{5:prop:Symanzik_with_parameters_determinantal_formula}
Let $l$ be a nonnegative integer and $w$ be a family of $l$ elements in $\vect u$. For $i\in\zint1l$, let $\~w_i\in\R^n$ be such that $w_i=U\~w_i$. Let $v$ be a family of vectors in $\Z^{n-r}$ such that $v^\tran$ is a basis of the vector subspace $\ker(U)$. Then,
\[ \Sym_k(u; w;x)=\Big(\frac{a}{b}\Big)^k\norm{\mat v^\tran\wedge\mat{\~w}}_{k,x}^k \]
where $a=\norm{\Zmod u}/\norm{\Zvect u}$ and $b=\norm{\Zmod v}/\norm{\Zvect v}$ are defined as in Theorem~\ref{5:thm:duality}.
\end{proposition}

\begin{proof}
By the multilinearity of the determinant and by a density argument, we can restrict our study to the case $w_1, \dots, w_l\in\Zmod u$. Applying Proposition~\ref{5:prop:Symanzik_determinantal_formula} to $\Sym_k(u\concat w;x\concat(0,\dots,0))$, one obtains that
\[ \Sym_k(u; w;x)=\Big(\frac{a'}{b'}\Big)^k\norm{\mat v^{\prime\tran}}_{k,x'}^k\;, \]
where
\begin{gather*}
V':={\pbmatrixIV{V^\tran}{\~W}{0}{\scalebox{0.75}[1.0]{$-$}\!\Id_l}}^\tran, \\
x'=(x_1, \dots, x_n, 0, \dots, 0), \\
a'=\frac{\norm{\Zmod{(u\concat w)}}}{\norm{\Zvect{u\concat w}}} \text{ and} \\
b'=\frac{\norm{\Zmod{v'}}}{\norm{\Zvect{v'}}}.
\end{gather*}
It is not difficult to see that $a'=a$, $b'=b$ and $\norm{\mat v^{\prime\tran}}_{k,x'}^k=\norm{\mat v^\tran\wedge\~{\mat w}}_{k,x}^k$.
\end{proof}

We now give another form of Symanzik polynomials which will allow us to explain the link between Definition~\ref{5:defi:Symanzik_polynomials_with_parameters} and the second Symanzik polynomial of the introduction (Equation~\eqref{5:eqn:second_Symanzik_polynomial}). We need to define the following orientation. Choose a volume form $\omega$ on $\vect u$ and, for two subsets $I$ and $J$ of respective sizes $r-l$ and $l$, set $\epsilon_I(J):=\sign(\omega(\mat u_{I\concat J}))\in\{-1,0,1\}$. Depending on $\omega$, we will get an orientation or its opposite.

\begin{proposition}
\label{5:prop:Symanzik_and_orientations}
Let $l$ be a nonnegative integer and $w$ be a family of $l$ elements in $\vect u$.
Then,
\[ \Sym_k(u; w; x)=\frac1{{\norm{\Zvect{u}}}^k}\sum_{\substack{I\subset\zint1n \\ \card I=r-l}} \Big(\sum_{\substack{J\subset\compl I \\ \card J=l}}\epsilon_I(J)\norm{\mat u_{I\cup J}}(\mat{\~w}, \mat e_J)\Big)^kx^{\compl I}, \]
where $e$ is the standard basis of $\R^n$ and $\~w$ is defined as in Proposition~\ref{5:prop:Symanzik_with_parameters_determinantal_formula}.
\end{proposition}

Note that, since $k$ is even, the right-hand side does not depend on $\omega$.

\begin{proof}
Without loss of generality, assume that members of $w$ are in $\Zmod u$. Let $f$ be a basis of $\Zmod u=\Zmod(u\concat w)$ such that $\omega(\mat f)>0$. Let $V\in\M_{r,n}(\Z)$ be such that $U=FV$. Then, $\epsilon_I(J)=\sign(\det(V_{I\concat J}))$. Let $T=V\~W$. Lemma~\ref{5:lem:u=fv} shows that, for $I\subset\zint1n$ of size $r-l$, we have
\begin{equation} \label{5:eqn:prop:Symanzik_and_orientations}
\norm{\mat u_I\wedge\mat w}^k=\norm{\Zmod u}^k\norm{\mat v_I\wedge\mat t}^k=\norm{\Zmod u}^k(\mat v^\tran, \mat e_I\wedge\mat{\~w})^k.
\end{equation}
Then,
\begin{align*}
(\mat v^\tran, \mat e_I\wedge\mat{\~w})
 &= \bigg(\mat v^\tran, \mat e_I\wedge\Big(\sum_{\substack{J\subset\zint1n \\ \card J=l}} (\mat{\~w}, \mat e_J)\mat e_J\Big)\bigg) \\
 &= \sum_{\substack{J\subset\compl I \\ \card J=l}} (\mat v^\tran, \mat e_I\wedge\mat e_J)(\mat{\~w}, \mat e_J) \\
 &= \sum_{\substack{J\subset\compl I \\ \card J=l}} \epsilon_I(J)\norm{\mat v_{I\cup J}}(\mat{\~w}, \mat e_J).
\end{align*}
Hence, by~\eqref{5:eqn:prop:Symanzik_and_orientations},
\[ \norm{\mat u_I\wedge\mat w}^k=\norm{\Zmod u}^k\bigg(\sum_{\substack{J\subset\compl I \\ \card J=l}} \epsilon_I(J)\norm{\mat v_{I\cup J}}(\mat{\~w}, \mat e_J)\bigg)^k. \]
By Lemma~\ref{5:lem:u=fv}, $\norm{\Zmod u}\norm{\mat v_{I\cup J}}=\norm{\mat u_{I\cup J}}$. We conclude the proof using Definition~\ref{5:defi:Symanzik_polynomials_with_parameters}.
\end{proof}

\begin{example} \label{5:ex:second_Symanzik_polynomial}
In this example, we will explain the link between the second Symanzik polynomial from the introduction and Symanzik polynomials with one parameter introduced above. More precisely, we wish to show that
\begin{equation} \label{5:eqn:ex:second_Symanzik_polynomial}
\phi_G(\mom, x)=\Sym_2(q; (\mom); x),
\end{equation}
where we recall notations below.

Let $G=(V,E)$ be a connected graph as in Example~\ref{5:ex:Kirchhoff_s_thm_1}, and $Q\in\M_{p,n}(\Z)$ be an incidence matrix of $G$. Take $\mom$ a vector in $\R^p$ such that $\sum_{v\in V}\mom_v=0$ (this is a slightly less general case compared to the definition in the introduction since we take $D=1$). The second Symanzik polynomial, denoted by $\phi_G$, is defined by
\begin{equation*}
\phi_G(\mom, x):=\sum_{F\in\SF_2}-\mom_{F_1}\mom_{F_2}\prod_{e\not\in F}x_e,
\end{equation*}
with the following notations. The set $\SF_2$ denotes the set of spanning forests of $G$ that have two connected components. For $F\in\SF_2$, $F_1$ and $F_2$ denotes the two connected components of $F$. Then, $\mom_{F_i}:=\sum_{v\in F_i}\mom_v$ for $i\in\{1,2\}$. Since the sum of $\mom_v$ is zero, one can replace $-\mom_{F_1}\mom_{F_2}$ by $\mom_{F_1}^2$.

Since $G$ is a connected graph, $\sum_{v\in V}\mom_v=0$ is equivalent to $\mom\in\vect q$. Now, we compare $\phi_G(\mom, x)$ with the Symanzik polynomial $\Sym_2(q; (\mom); x)$. By Proposition~\ref{5:prop:Symanzik_and_orientations}, with similar notations,
\begin{equation*}
\Sym_2(q; (\mom); x)=\frac1{\norm{\Zvect q}^2}\sum_{\substack{I\subset\zint1n \\ \card I=r-1}} \Big(\sum_{\substack{j\in\compl I}}\epsilon_I(j)\norm{\mat q_{I+j}}\~\mom_j\Big)^2x^{\compl I},
\end{equation*}
where we write $\epsilon_I(j)$ instead of $\epsilon_I((j))$. By~\eqref{5:eqn:Kirchhoff_s_thm_1}, the ratio $\norm{\mat q_{I+j}}^2/\norm{\Zvect q}^2$ equals $1$ if $I+j$ forms a spanning tree of $G$, and $0$ otherwise. Thus, we can restrict the sum to the sets $I$ that correspond to $2$-forests of $G$. It remains to show that, if $I\subset\zint1n$ corresponds to $\F\in\SF_2$, then
\[ \Big\lvert\sum_{\substack{j\in\compl I}}\epsilon_I(j)\~\mom_j\Big\rvert=\abs{\mom_{F_1}}. \]
Let us try to better understand what happens here. On each edge $e$, we put a weight $\~\mom_e$ such that, for each vertex $v$, $\mom_v$ equals the algebraic sum of the weight of the edges incident to $v$:
\[ \mom_v = \sum_{e\in E}q_{v,e}\~\mom_e. \]
Thus, when we sum $\mom_v$ over all vertices $v$ of $F_1$, each edge in $F_1$ appears two times with opposite signs and each edge in $F_2$ does not appear in the sum. Thus, it only remains edges joining $F_1$ and $F_2$:
\[ \mom_{F_1}=\sum_{e\in E(F_1, F_2)}\epsilon_{F_2,F_1}(e)\~\mom_e, \]
where $E(F_1, F_2)$ is the set of edges joining $F_1$ and $F_2$, and $\epsilon_{F_2,F_1}(e)$ equals $+1$ if $e$ is oriented from $F_2$ to $F_1$, and equals $-1$ otherwise. Let us now see how $\epsilon_I$ computes this orientation.

More precisely, it remains to show that, if $e$ and $e'$ are two edges of $E(F_1, F_2)$, then $e$ and $e'$ have the same orientation (between $F_1$ and $F_2$) if and only if $\epsilon_I(e)=\epsilon_I(e')$. Suppose that they do not have the same orientation. In this case, there exists an oriented cycle in $F+e+e'$ respecting their orientation. This means that, for some nonnegative integer $l$, there exist $i_1, \dots, i_l\in I$ and $\eta_1, \dots, \eta_l\in\{-1, +1\}$ such that
\[ \eta_1q_{i_1}+\dots+\eta_lq_{i_l}+q_e+q_{e'}=0. \]
Thus,
\begin{align*}
\epsilon_I(e)
 &= \sign(\omega(q_I\concat q_e)) \\
 &= \sign(\omega(q_I\concat(-\eta_1q_{i_1}-\dots-\eta_lq_{i_l}-q_{e'}))) \\
 &= \sign(\omega(q_I\concat(-q_{e'}))) \\
 &= -\epsilon_I(e').
\end{align*}
The other case, with $e$ and $e'$ having the same orientation, leads to the opposite equality. Thus, $\epsilon_{F_2,F_1}(e)\epsilon_I(e)\in\{-1, +1\}$ does not depend on the chosen $e$. Finally, we have proved~\eqref{5:eqn:ex:second_Symanzik_polynomial}.
\end{example}

\subsection{Symanzik rational fractions}
\label{5:subsec:Symanzik_rational_fractions}

In this subsection, we introduce Symanzik rational fractions. The following definition is more canonical than Definition~\ref{5:defi:Symanzik_polynomials} on Symanzik polynomials. Indeed, the problems highlighted in Remarks~\ref{5:rem:another_definition} and~\ref{5:rem:PID} disappear for Symanzik rational fractions. For instance, one can easily define Symanzik rational fractions for families of vectors in $\R^p$. This is useful in order to compute the nonarchimedean height pairing, as we will see in Section~\ref{5:subsec:divisor}. Moreover, these rational fractions have a nice property of stability discussed in Section~\ref{5:sec:variation}.

\begin{definition} \label{5:defi:Symanzik_rational_fractions}
Let $u$ be a family of vectors in $\R^p$, let $l$ be a nonnegative integer and $w$ be a family of $l$ vectors in $\vect u$. The \emph{(normalized) Symanzik rational fraction of order $k$ of $u$ with parameters $w_1, \dots, w_l$} is defined by
\[ \ratSym_k(u; w;x):=\frac{\displaystyle\sum_{\substack{I\subset\zint1n \\ \card I=r-l}}\norm{\mat u_I\wedge\mat w}^k x^{\compl I}}{\displaystyle\sum_{\substack{I\subset\zint1n \\ \card I=r}}\norm{\mat u_I}^k x^{\compl I}}. \qedhere \]
\end{definition}

If elements of $u$ are in $\Z^p$, then we simply have
\[ \ratSym_k(u; w;x)=\frac{\Sym_k(u;w;x)}{\Sym_k(u;x)}. \]

Later in Remark~\ref{5:rem:ratSym_determinantel_formula}, we will give a determinantal formula for Symanzik rational fractions in the case $k=2$.

We have seen in Example~\ref{5:ex:Kirchhoff_s_thm_1} that Kirchhoff and Symanzik polynomials are linked to graphs. In the next section, we will use our generalization to define Symanzik polynomials over simplicial complexes.

%%%
\section{Symanzik polynomials on simplicial complexes}
\label{5:sec:simplicial_complexes}

Every result in this section can be extended to CW-complexes.

\medskip
Before extending Symanzik polynomials to simplicial complexes, we will generalize the notion of forests in graph theory to the case of simplicial complexes. Generalized forests will reveal interesting properties of these polynomials.

\subsection{Simplicial complexes and forests}
\label{5:subsec:forest}

Let $V$ be a finite set of \emph{vertices}. An \emph{abstract simplicial complex on $V$} is a nonempty set $\Delta$ of subsets of $V$ called \emph{faces} such that $\Delta$ is stable by inclusion: if $\delta$ is a face and if $\gamma\subset\delta$, then $\gamma$ is also a face. A simplicial complex $\Gamma$ is a \emph{subcomplex} of $\Delta$ if $\Gamma\subset\Delta$. If $\delta$ is a face, its \emph{dimension} is $\dim(\delta):=\card\delta-1$. Notice that $\Delta$ has always a unique face of dimension $-1$: the empty set. The dimension of $\Delta$ is the maximal dimension of its faces. We call it $d$. The $d$-dimensional faces are called \emph{facets}. If $l$ is an integer, $l\geq-1$, then $\Delta_l$ is the set of faces of $\Delta$ of dimension $l$. The \emph{$l$-skeleton $\Delta_{\leq l}$ of $\Delta$} is the subcomplex of all faces of dimension at most $l$ of $\Delta$:
\[ \Delta_{\leq l}:=\bigcup_{i=-1}^l\Delta_i. \]

In this article, we will suppose that a complex $\Delta$ is always endowed with an enumeration on each set of faces $\Delta_l$, $l\in\zint{-1}d$, by numbers from $1$ to $\card{\Delta_l}$.

Let $\Chain_l(\Delta):=\freeZmod{\Delta_l}$ be the set of \emph{$l$-chains} of $\Delta$, where each simplex has the standard orientation given by the enumeration of $\Delta_0$. Let $\d_\Delta$ be the \emph{$d$-th boundary operator} of the augmented simplicial chain complex:
\[ \begin{array}{rrcl}
\d_\Delta: & \Chain_d(\Delta) & \longrightarrow & \Chain_{d-1}(\Delta), \\
& \{i_0, \dots, i_d\}\text{ with $i_0<\dots<i_d$ } & \longmapsto & \sum_{j=0}^d(-1)^j\{i_0,\dots, i_{j-1}, i_{j+1}, \dots, i_d\}.
\end{array} \]
If $\delta_i$ denotes the $i$-th facet for the chosen enumeration, then $(\delta_1, \dots, \delta_{\card{\Delta_d}})$ is a basis of $\Chain_d(\Delta)$. We identify $\Chain_d(\Delta)$ with $\Z^{\card{\Delta_d}}$ mapping this basis to the standard basis. We obtain in the same way a basis of $\Chain_{d-1}(\Delta)$ and an identification $\Chain_{d-1}(\Delta)\simeq\Z^{\card{\Delta_{d-1}}}$. We can now represent $\d_\Delta$ by a matrix in $\M_{\card{\Delta_{d-1}}, \card{\Delta_d}}(\Z)$ which we will call the \emph{$d$-th incidence matrix of $\Delta$}. The kernel of $\d_\Delta$ in $\Chain_d(\Delta)$ is denoted by $\Cyc_d(\Delta)$ and its elements are called \emph{$d$-cycles}. The image of $\d_\Delta$ in $C_{d-1}(\Delta)$ is denoted by $\Bound_{d-1}(\Delta)$ and its elements are called \emph{$(d-1)$-boundaries}.

We can also define $\Cyc_l(\Delta)$ and $\Bound_{l-1}(\Delta)$ for $l<d$ as the kernel and the image of $\bound_{\Delta, l}:=\bound_{\Delta_{\leq l}}$.

Abusing notation, in this section we will denote in the same way an $l$-chain (or a family of $l$-chains) and its image by the identification $\Chain_l(\Delta)\simeq\Z^{\card{\Delta_l}}$.

\begin{example} \label{5:ex:bipyramid}
Figure~\ref{5:fig:bipyramid} is an example of a $2$-dimensional simplicial complex $\Delta$ called the bipyramid. Its $2^\text{nd}$ incidence matrix is shown on the right. We represented a $2$-cycle in blue and a $1$-boundary, $\d_\Delta\big(\{2,3,5\}+\{3,4,5\})$, in red. \qedhere

\begin{figure}
\begin{tikzpicture}[x={(1cm,0)},y={(3.54mm, 3.54mm)},z={(0,.85cm)}, thick]

\def\face[#1] (#2)--(#3)--(#4)--cycle;{
  \begin{scope}[shift=(#3), y={($(#2)-(#3)$)}, x={($(#4)-(#3)$)}]
  \fill[#1] (0,0)--(0,1)--(1,0)--cycle;
  \draw[#1] (.33,.33) [y={(${1/.866}*(-1/2,1)$)},->] ++(30:.2) arc (30:330:.2);
\end{scope} }

\coordinate (1) at (0,0,2.828);
\coordinate (2) at (-2,0,0);
\coordinate (3) at (-60:2);
\coordinate (4) at (60:2);
\coordinate (5) at (0,0,-2.828);

\foreach \i/\j/\k/\l/\m in {
    1/2/3/blue/,
    1/2/4/blue/dashed,
    1/3/4/blue/,
    2/3/4/blue/dashed,
    2/3/5/gray/,
    2/4/5/gray/dashed,
    3/4/5/gray/} {
  \face[color=\l, fill opacity=.1, \m] (\i)--(\j)--(\k)--cycle; }

\foreach \i/\j in {1/2,1/3,1/4,3/5} {
  \draw (\i)--(\j) \midarrow; }
\draw[dashed] (2)--(4) \midarrow;
\foreach \i/\j in {2/3,2/5,3/4,4/5} {
  \draw[very thin] (\i)--(\j) \midarrow; }
\foreach \i/\j in {2/3,3/4,4/5,5/2} {
  \draw[red, -latex] (\i)--(\j); }

\node[above] at (1) {$1$};
\node[left] at (2) {$2$};
\node[right] (node4) at (4) {$4$};
\node[below] at (5) {$5$};

\draw[<-, shorten <=3pt, thin] (3) to[bend right] ++(1cm,-.5cm) node[right] {$3$};

\path (node4.west)++(2cm,0) node[right] {$
\begin{matrix}
&\hspace{-4mm}
\begin{array}{ccccccc}
\phantom{-1} & \phantom{-1} & \phantom{-1} & \phantom{-1} & \phantom{-1} & \phantom{-1} & \phantom{-1} \\[-1em]
  \color{blue}\rotatebox{-90}{$\{1,2,3\}\rightarrow$} &
  \color{blue}\rotatebox{-90}{$\{1,2,4\}\rightarrow$} &
  \color{blue}\rotatebox{-90}{$\{1,3,4\}\rightarrow$} &
  \color{blue}\rotatebox{-90}{$\{2,3,4\}\rightarrow$} &
  \color{gray}\rotatebox{-90}{$\{2,3,5\}\rightarrow$} &
  \color{gray}\rotatebox{-90}{$\{2,4,5\}\rightarrow$} &
  \color{gray} \rotatebox{-90}{$\{3,4,5\}\rightarrow$}
\end{array}
\\[1cm]
\begin{matrix}
  \{1,2\}\rightarrow \\
  \{1,3\}\rightarrow \\
  \{1,4\}\rightarrow \\
  \color{red}\{2,3\}\rightarrow \\
  \{2,4\}\rightarrow \\
  \color{red}\{2,5\}\rightarrow \\
  \color{red}\{3,4\}\rightarrow \\
  \{3,5\}\rightarrow \\
  \color{red}\{4,5\}\rightarrow
\end{matrix}
&\hspace{-4mm}
\begin{pmatrix}
  1 & 1 & 0 & 0 & 0 & 0 & 0 \\
 -1 & 0 & 1 & 0 & 0 & 0 & 0 \\
  0 &-1 &-1 & 0 & 0 & 0 & 0 \\
  1 & 0 & 0 & 1 & 1 & 0 & 0 \\
  0 & 1 & 0 &-1 & 0 & 1 & 0 \\
  0 & 0 & 0 & 0 &-1 &-1 & 0 \\
  0 & 0 & 1 & 1 & 0 & 0 & 1 \\
  0 & 0 & 0 & 0 & 1 & 0 &-1 \\
  0 & 0 & 0 & 0 & 0 & 1 & 1
\end{pmatrix}
\end{matrix} $};

\end{tikzpicture}
\caption{\label{5:fig:bipyramid} A bipyramid and its $d$-th incidence matrix.}
\end{figure}
\end{example}

For the rest of this section, we fix $V$ a finite set, $\Delta$ an abstract simplicial complex on $V$, $d:=\dim(\Delta)$, $n:=\card{\Delta_d}$ and $p:=\card{\Delta_{d-1}}$. We set $U\in\M_{p,n}(\Z)$ the $d$-th incidence matrix of $\Delta$.
As in the previous section, we set $r:=\rk(U)$.

\medskip

Now we define (simplicial) $\kappa$-forests of $\Delta$ following~\cite{BK16} which is based on~\cite{DKM09} and~\cite{Kal83}. This definition is one possibility to generalize the notion of forests in graphs to higher dimensions (see~\cite{DKM09} for a slightly different definition). Indeed, in dimension one, our definition will coincide with the usual one if one sees graphs as $1$-dimensional simplicial complexes.

\begin{definition} \label{5:defi:forest}
Let $\Gamma\subset\Delta$ be a simplicial subcomplex of $\Delta$ such that $\Gamma_{\leq d-1}=\Delta_{\leq d-1}$. Let $\kappa$ be a nonnegative integer.
Then $\Gamma$ is a \emph{$\kappa$-forest of the simplicial complex $\Delta$} if it verifies the following three properties.
\begin{enumerate}
\item \emph{acyclicity}: $\Gamma$ has no nonzero $d$-cycle,
\item $\rk(\d_\Delta)-\rk(\d_\Gamma)=\kappa$,
\item $\card{\Gamma_d}=\card{\Delta_d}-\rk\big(\Cyc_d(\Delta)\big)-\kappa$.
\end{enumerate}
The set of $\kappa$-forests of $\Delta$ is denoted by $\For_\kappa(\Delta)$. Moreover, we will call $\Gamma$ a \emph{subforest} of $\Delta$ if it is acyclic. By Proposition~\ref{5:prop:two_out_of_three} below, this is equivalent to being a $\kappa$-forest for some $\kappa\geq0$.
\end{definition}

\begin{example} \label{5:ex:bipyramid_forest}
Let $\Gamma$ be the bipyramid where we removed facets $\{1, 3, 4\}$ and $\{3, 4, 5\}$. To check that it is a $0$-forest, we look at the $d$-th incidence matrix $U_\Gamma$ of $\d_\Gamma$ (which is a submatrix of $U$). Clearly $\Gamma_{\leq 1}=\Delta_{\leq 1}$. Moreover, the three conditions are verified:
\begin{enumerate}
\item the kernel of $U_\Gamma$ is trivial. Thus, $\Gamma$ is acyclic.
\item $\Im(U_\Gamma)=\Im(U)$. Thus, $\rk(\d_\Delta)-\rk(\d_\Gamma)=0$.
\item $\card{\Gamma_2}=5$, $\card{\Delta_2}=7$ and $\rk(\Cyc_2(\Delta))=\rk(\ker(\d_\Delta))=2$. Thus,
\[ \card{\Gamma_2}=\card{\Delta_2}-\rk{\Cyc_2(\Delta)}-0. \qedhere \]
\end{enumerate}

\end{example}

We will see easier ways to check whether a subcomplex is a $\kappa$-forest in this section. But before, let us motivate the name of forest.

\begin{example} \label{5:ex:forests_on_graphs}

If $G=(V,E)$ is a simple graph, then it can be seen as a $1$-dimensional complex. The three conditions of Definition~\ref{5:defi:forest} for a subgraph $F=(V_F, E_F)$ to be a $\kappa$-forest can be rewritten as
\begin{enumerate}
\item $F$ is acyclic,
\item $F$ has $\kappa$ more connected components than $G$,
\item $\card{E_F}=\card{E}-g_G-\kappa$ (which equals $\card{V}-c_G-\kappa$) where $c_G$ is the number of connected components of $G$ and $g_G$ is the genus of $G$ which equals $\card{E}-\card{V}+c_G$.
\end{enumerate}
Thus, for instance, if $G$ is connected, a $0$-forest is exactly a spanning tree, and a $\kappa$-forest in our definition is what is usually called a $(\kappa+1)$-forest (because in graphs it is more natural to count the number of connected components).
\end{example}

It is well-known that, in graphs, only two out of the three conditions enumerated in the previous example are needed to be a $\kappa$-forest. This is still true in simplicial complexes.

\begin{proposition} \label{5:prop:two_out_of_three}
A subcomplex $\Gamma$ of $\Delta$ such that $\Gamma_{\leq d-1}=\Delta_{\leq d-1}$ is a $\kappa$-forest for a nonnegative integer $\kappa$ if and only if it verifies two out of the three conditions of the Definition~\ref{5:defi:forest}.
\end{proposition}

\begin{proof}
By the rank-nullity theorem, acyclicity (\ie, triviality of $\ker(\d_\Gamma)$) is equivalent to $\rk(\d_\Gamma)=\card{\Gamma_d}$. The same theorem implies that $\rk(\Cyc_d(\Delta))=\card{\Delta_d}-\rk(\d_\Delta)$. Thus, the three conditions can be rewritten as:
\begin{enumerate}
\item $\rk(\d_\Gamma)=\card{\Gamma_d}$,
\item $\rk(\d_\Delta)-\rk(\d_\Gamma)=\kappa$,
\item $\card{\Gamma_d}=\card{\Delta_d}-(\card{\Delta_d}-\rk(\bound_\Delta))-\kappa$.
\end{enumerate}
Now, the proposition is clear.
\end{proof}

Two more remarks before we end this subsection. First, there is a natural bijection between the set of $\kappa$-forests $\Delta$ and the set of independent subfamilies of $u$ of size $r-\kappa$. Thus, the subforests naturally form the family of independent sets of a matroid over $\zint1n$ (we refer to Section~\ref{5:subsec:matroids} for a definition of matroids).

Finally, it is useful to have in mind a practical algorithm to create $\kappa$-forests of $\Delta$. As long as there are still $d$-cycles, choose a cycle, pick a facet in this cycle and remove it. At the end, we will obtain a $0$-forest. One can then remove any $\kappa$ more facets to obtain a $\kappa$-forest. Every $\kappa$-forest can be obtained in this way.

\subsection{Kirchhoff and Symanzik polynomials for simplicial complexes}

Here is some more notations. Set
\[ \Homl_k(\Delta):=\rquot{\Cyc_k(\Delta)}{\Bound_k(\Delta)}, \]
the \emph{$k$-th reduced homology group of $\Delta$}. If $A$ is a $\Z$-module, $\Tor(A)$ will denote its torsion part. Abusing notations, if a face $\delta$ is numbered by $i$, we will often write $\delta$ instead of $i$. In the same way, a set of faces will sometimes denote the set of numbers associated to these faces. Finally, if $I\subset\zint1n$, $\Delta\rest I$ will denote the subcomplex of $\Delta$ with facets labeled by an element of $I$.

\begin{definition} \label{5:defi:Kirchhoff_polynomials_simplicial_complexes}
We define the \emph{Kirchhoff polynomial of order $k$ of $\Delta$} to be the Kirchhoff polynomial of the boundary operator:
\[ \Kir_k(\Delta; x):=\Kir_k(u; x), \]
where we recall that $U$ is the matrix associated to $\bound_\Delta$
\end{definition}

This definition can be rewritten in terms of forests.

\begin{theorem}[Kirchhoff's theorem for simplicial complexes] \label{5:thm:Kirchhoff_s}
One has
\[ \Kir_k(\Delta; x)=\card{\Tor(\Homl_{d-1}(\Delta))}^k\sum_{\text{$\Gamma$ $0$-forest of $\Delta$}}\card{\rquot{\Bound_{d-1}(\Delta)}{\Bound_{d-1}(\Gamma)}}^kx^{\Gamma_d}. \]
\end{theorem}

\begin{proof}
We recall the Definition~\ref{5:defi:Kirchhoff_polynomials} about Kirchhoff polynomials:
\[ \Kir_k(u; x):=\frac1{\norm{\Zvect u}^k}\sum_{\substack{I\subset\zint1n \\ \card I=r}}\norm{\mat u_I}^k x^I. \]
Notice that $\mat u_I$ is nonzero if and only if $U_I$ has a trivial kernel, \ie, if and only if $\Delta\rest I$ is acyclic. Hence, in the above sum, we can restrict to sets $I$ such that $\Delta\rest I$ is a $0$-forest of $\Delta$. We have
\begin{align*}
\Tor(H_{d-1}(\Delta))
 &= \rquot{\Zvect{\Bound_{d-1}(\Delta)}}{\Bound_{d-1}(\Delta)} \\
 &\simeq \frac{\rquot{\Zvect{\Bound_{d-1}(\Delta)}}{\Bound_{d-1}(\Delta\rest I)}}{\rquot{\Bound_{d-1}(\Delta)}{\Bound_{d-1}(\Delta\rest I)}}, \\
 &\simeq \frac{\rquot{\Zvect u}{\Zmod u_I}}{\rquot{\Bound_{d-1}(\Delta)}{\Bound_{d-1}(\Delta\rest I)}},
\end{align*}
Since $\bound_\Delta$ and $\bound_{\Delta\rest I}$ have same rank, we can write
\[ \card{\Tor(H_{d-1}(\Delta))}\cdot\card{\rquot{\Bound_{d-1}(\Delta)}{\Bound_{d-1}(\Delta\rest I)}} = \card{\rquot{\Zvect u}{\Zmod u_I}}. \]
We can now use~\eqref{5:eqn:norm_to_det} to conclude the proof.
\end{proof}

\begin{remark} \label{5:rem:thm:Kirchhoff_s}
Actually, the equality of the theorem takes a simpler form:
\[ \Kir_k(\Delta; x)=\sum_{\text{$\Gamma$ $0$-forest of $\Delta$}}\card{\Tor(\Homl_{d-1}(\Gamma))}^k x^{\Gamma_d}. \]
Indeed,
\[ \Tor(\Homl_{d-1}(\Gamma))=\rquot{\Zvect{\Bound_{d-1}(\Gamma)}}{\Bound_{d-1}(\Gamma)}\simeq\rquot{\Zvect{u_\Gamma}}{\Zmod{u_\Gamma}}=\rquot{\Zvect{u}}{\Zmod{u_\Gamma}}. \]
Though, the form of the theorem enlights that $\Kir_k(\Delta; x)$ is divisible by $\card{\Tor(\Homl_{d-1}(\Delta))}^k$ in $\Z[x]$. This is the factor we have seen above in Example~\ref{5:ex:practical_computation}, and its the $a^k$ of Theorem~\ref{5:thm:duality} and Propositions~\ref{5:prop:Symanzik_determinantal_formula} and~\ref{5:prop:Symanzik_with_parameters_determinantal_formula}.

We can compare the above formula to Corollary 21 of~\cite{BK16} (called Simplicial matrix tree theorem). Let $R\subset\zint1p$ corresponds to the chosen root. We get the corollary setting $k=2$, $x=(1,\dots,1)$ and multiplying both sides by $\norm{(\mat u^\tran)_R}^2/\norm{\Zmod{u^\tran}}^2$.
\end{remark}

\begin{example} \label{5:ex:real_projective_plane}
The real projective plane is a very instructive example. In order to simplify calculations, we study it as a $\Delta$-complex, called $\Delta$, instead of a simplicial complex: see Figure~\ref{5:fig:real_projective_plane}. We refer to~\cite{Hat02}*{p. 103}, for a precise definition of $\Delta$-complexes.

\begin{figure}
\begin{tikzpicture}[thick]

\fill[blue!10] (0,0) rectangle (3,3);
\draw (0,0)
  --++(0,3)  \midarrow node[midway,left] {$a$} node[above left] {$1$}
  --++(3,0)  \midarrow node[midway,above] {$b$} node[above right] {$2$}
  --++(0,-3) \midarrow node[midway,right] (a) {$a$} node[below right] {$1$}
  --++(-3,0) \midarrow node[midway,below] {$b$} node[below left] {$2$}
  --++(3,3)  \midarrow node[midway, left] {$c$};
\draw[->,blue] (2.1,.9) node {$\delta$} ++(30:.4) arc(30:330:.4);
\draw[->,blue] (.9,2.1) node {$\gamma$} ++(30:.4) arc(30:330:.4);

\path (a.west)++(2cm,0) node[right] {$
\begin{matrix}
&
\begin{array}{cc}
  \phantom{-1} & \phantom{-1} \\[-1em]
  \color{blue}\gamma & \color{blue}\delta
\end{array}
\\[-.2em]
&
\begin{array}{cc}
  \phantom{-1} & \phantom{-1} \\[-1em]
  \downarrow & \downarrow
\end{array}
\\
U=\!\!\!\!\!\!
&
\left(
\begin{array}{cc}
 -1 & -1 \\
 -1 & -1 \\
  1 & -1
\end{array}
\right)
&
\!\!\!\!\!\!
\begin{array}{c}
  \leftarrow a \\
  \leftarrow b \\
  \leftarrow c
\end{array}
\end{matrix} $};
\end{tikzpicture}
\caption{\label{5:fig:real_projective_plane} A decomposition of $\RPII$.}
\end{figure}

A basis of $\Zmod{u^\tran}$ is $(\begin{psmallmatrix}1\\1\end{psmallmatrix},\begin{psmallmatrix}1\\-1\end{psmallmatrix})$. Thus, following Example~\ref{5:ex:practical_computation},
\[ \Kir_2(\Delta;(x_1,x_2))=\det(\begin{psmallmatrix}1&1\\1&-1\end{psmallmatrix}\begin{psmallmatrix}x_1&0\\0&x_2\end{psmallmatrix}\begin{psmallmatrix}1&1\\1&-1\end{psmallmatrix})=2^2 x_1x_2. \]
We see that $\Delta$ has only one $0$-forest: itself. Indeed, it does not have nontrivial cycles. Then, looking at Remark~\ref{5:rem:thm:Kirchhoff_s}, $\card{\Tor(\Homl_{d-1}(\Delta))}=2$. Indeed, the class of $a+b$ is an element of order $2$ in $\Tor(\Homl_{d-1}(\Delta))$. Unlike the case of graphs, setting $x_1=x_2=1$ does not give the number of $0$-forests. Finding this number is difficult in general: there exist examples where all the coefficients of the polynomial are not equal (see Example~\ref{5:ex:possible_Symanzik_polynomials}). Knowing the value of the Kirchhoff polynomial for all order $k$ might help us.
\end{example}

\medskip

Now we only focus on Symanzik polynomials.

\begin{definition} \label{5:defi:Symanzik_polynomials_simplicial_complexes}
We define the \emph{Symanzik polynomial of order $k$ of $\Delta$} to be the Symanzik polynomial of the boundary operator:
\[ \Sym_k(\Delta; x):=\Sym_k(u; x). \qedhere \]
\end{definition}

A direct corollary of Theorem~\ref{5:thm:Kirchhoff_s} follows.

\begin{corollary} \label{5:cor:Kirchhoff_s}
One has
\[ \Sym_k(\Delta; x)=\card{\Tor(\Homl_{d-1}(\Delta))}^k\sum_{\text{$\Gamma$ $0$-forest of $\Delta$}}\card{\rquot{\Bound_{d-1}(\Delta)}{\Bound_{d-1}(\Gamma)}}^kx^{\compl{\Gamma_d}}. \]
\end{corollary}

\begin{example} \label{5:ex:metric_graph}
Symanzik polynomials compute some interesting data on metric graphs. For more details about this example, we refer to~\cite{KS00}. Let $G$ be a simple graph with $n$ edges $e_1, \dots, e_n$. For each $i\in\zint1n$, set a positive length $l_i$ to the edge $e_i$. Let $\G$ be the geometric representation of $G$, \ie, $\G$ is a one dimensional $CW$-complex with the metric induced by the lengths. We set $y=(l_1, \dots, l_n)$. Let $g$ be the genus of $G$.

Let $v$ be a basis of $\Cyc_1(G)$. We define the map
\[ \begin{array}{rrcl}
\phi: & \Chain_1(G) & \to & \R^g, \\
 & a & \mapsto & ((a,v_1)_y, \dots, (a,v_g)_y).
\end{array} \]
Actually, $\phi$ could be naturally extended to the singular $1$-chains on $\G$. We now define the Jacobian torus of $\G$ to be
\[ \Jac(\G):=\R^g/\phi(\Cyc_1(G)). \]
The standard inner product in $\R^g$ induces an inner product on $\Chain_1(G)$. The latter is associated to the symmetric matrix $\~Y:=YVV^\tran Y$. Thus,
\begin{equation} \label{5:eqn:ex:metric_graph}
\Vol(\Jac(\G))=\sqrt{\det(V^\tran\~YV)}=\sqrt{\det(V^\tran YV)^2}=\Sym_2(G, y).
\end{equation}

In fact, this equality holds for a deeper combinatorial reason, (\cf \cite{KS00}). Choosing an arbitrary point $p$ of $\G$ leads to a bijection between the set of so-called \emph{break divisors} of $G$ and its Jacobian torus. This bijection induced a quasi-partition of the Jacobian torus into parallelotopes labelled by spanning trees of $G$ (see Figure 1 of~\cite{ABKS14}). The parallelotope associated to a spanning tree $T$ has volume
\[ \prod_{i, e_i\not\in E(T)}l_i. \]
This gives another proof of the equality~\eqref{5:eqn:ex:metric_graph}. \qedhere
\end{example}

In the previous example, the value of the Symanzik polynomial obtained at the end will not change if one adds or deletes some vertices inside an edge of the metric graph. We will see in the next subsection that this result is more general. Before that, let us define Symanzik polynomials with parameters for simplicial complexes.

\begin{definition}
Let $w$ be a family of size $l$ of $(d-1)$-boundaries of $\Delta$. We define the \emph{Symanzik polynomial of order $k$ of $\Delta$ with parameters $w$} to be the Symanzik polynomial of $u$ with parameters $w$:
\[ \Sym_k(\Delta; w; x):=\Sym_k(u; w; x). \qedhere \]
\end{definition}

The following result is a direct corollary of both Theorem~\ref{5:thm:Kirchhoff_s} and Proposition~\ref{5:prop:Symanzik_and_orientations}.

\begin{corollary}
Let $w$ be a family of size $l$ of $(d-1)$-boundaries. Then
\[ \Sym_k(\Delta; w; x)=\card{\Tor(H_{d-1}(\Delta))}^k\sum_{\Gamma\in\For_l(\Delta)}\sum_{\substack{F\subset\compl{\Gamma_d} \\ \Gamma\cup F\in\For_0(\Delta)}}
\Big(\epsilon_{\Gamma_d}(F)\left\vert\frac{\Bound_{d-1}(\Delta)}{\Bound_{d-1}(\Gamma\cup F)}\right\vert(\mat{\~w},\mat e_F)\Big)^kx^{\compl{\Gamma_d}}, \]
where $\epsilon$ is defined before Proposition~\ref{5:prop:Symanzik_and_orientations}, and $\~w$ is a family of $d$-chains of $\Delta$ such that $\bound\~w_j=w_j$ for every $j\in\zint1l$.
\end{corollary}

\begin{remark}
As in Example~\ref{5:ex:second_Symanzik_polynomial}, $\epsilon_{\Gamma_d}$ has a geometrical meaning. For instance, if $l=1$ and if $\Gamma$ is a $1$-forest of $\Delta$, $\epsilon_{\Gamma_d}$ is one of the two opposite canonical orientations on facets $\delta$ of $\compl{\Gamma_d}$ such that $\Gamma+\delta$ is a $0$-forest of $\Delta$. These orientations are characterized by the fact that, for any two different such facets $\delta, \delta'$, any cycle of $\Gamma+\delta+\delta'$ will follow the orientation of exactly one of them.

In matroid theory, such orientations are encoded by oriented matroids, or by generalizations of them like $\R$-matroids. For more details, one can consult~\cite{Piq19-master-thesis} and~\cite{BB16}.
\end{remark}

\subsection{Nonarchimedean height pairing}
\label{5:subsec:divisor}

Let us introduce an interesting inner product which is linked to the Archimedean height pairing. Before going through this construction, let us roughly sketch this link and motivate the results of this section. For more details, the reader can compare Proposition~\ref{5:prop:divprod} with Theorem 1.2 of~\cite{ABBF}. On a smooth projective complex curve $X$ of dimension $n$, if $b$ and $b'$ are two divisors of degree $0$ with disjoint support, one can define the Archimedean height pairing
\[ \divprod{b}{b'}_X:=\mathrm{Re}\Bigl(\int_{\gamma_b}\omega_{b'}\Bigr), \]
where $\gamma_b$ is a $1$-chain disjoint from the support of $b'$ with boundary $b$, and $\omega_{b'}$ is a canonical logarithmic differential form with residue $b'$. If one considers, instead of a single curve $X$, a degenerating family of curves $\X=(X_s)_s$, parametrized by $s$ and degenerating at $s=0$, and two families of divisors $(b_s)_{s\neq0}$ and $(b'_s)_{s\neq0}$ on it, one can wonder what is the asymptotic of $\divprod{b_s}{b'_s}_{X_s}$ as $s$ approaches $0$. Authors of~\cite{ABBF} proved that, under some good conditions, the asymptotic can be described thanks to a Symanzik rational fraction as defined in Section~\ref{5:subsec:Symanzik_rational_fractions}. More precisely, one can associate to the singular fiber $X_0$ a dual graph $G$. Considering $G$ as a $1$-dimensional simplicial complex, the divisors will induce some boundaries $b_0,b'_0\in\Bound_0(G)$ on $G$. Then one can write
\[ \divprod{b_s}{b_s'}_{X_s}=\divprod{b_0}{b_0'}_{y(s)}+h(s), \]
where $h$ is a bounded function and $\divprod{b_0}{b_0'}_{y(s)}$ is defined as below (Definition~\ref{5:defi:divprod}) where $y(s)$ is a vector depending on $s$ and whose coordinates diverge to $+\infty$, as $s$ approaches $0$. This asymptotic is strongly linked to Theorem~\ref{5:thm:variation} of Section~\ref{5:sec:variation}. The dual graph appears in the nonarchimedean analytification of $X_s$, for any $s\neq0$. Thus, one can interpret $\divprod{\wcdot}{\wcdot}_y$ as a nonarchimedean height pairing. Moreover, as we mentioned above, this nonarchimedean height pairing, or more precisely the associated quadratic form, can be expressed as a Symanzik rational fraction. This is stated by Proposition~\ref{5:prop:divprod}.

If one replaces curves by higher dimensional projective complex manifolds, one can still define an analogous of the dual graph, which will be a higher dimensional complex. One can also define an analogous of the Archimedean height pairing. Thus, one can wonder if the asymptotic of the height pairing can be described thanks to a nonarchimedean height pairing on the dual complex as the one we define in Definition~\ref{5:defi:divprod} below.

Let $y$ be a family of $n$ positive real numbers. We denote by $\Chain_d(\Delta; \R)$ the vector subspace $\vect{\Chain_d(\Delta)}$. We endow $\Chain_d(\Delta; \R)$ with the inner product $(\wcdot, \wcdot)_y$ and $\Chain_{d-1}(\Delta; \R)$ with the standard inner product $(\wcdot, \wcdot)$. This allows us to define the adjoint $\bound^*:\Chain_{d-1}(\Delta; \R)\to\Chain_d(\Delta; \R)$ of $\bound$ according to these products.

Notice that $\Im(\bound^*)\subset\ker(\bound)^\perp$ and $\ker(\bound^*)=\Im(\bound)^\perp$. Thus, $\rk(\bound^*)+\dim(\ker(\bound))\leq\dim(C_d(\Delta; \R))$ and $\dim(\ker(\bound^*))+\rk(\bound)=\dim(C_{d-1}(\Delta; \R))$. Using the rank-nullity theorem twice, we deduce that the inequality must be an equality. Hence,
\[ \Im(\bound^*)=\Cyc_d(\Delta;\R)^\perp. \]

Let $b$ and $b'$ be two $(d-1)$-boundaries. One can find a unique preimage $a\in\Chain_d(\Delta; \R)$ of $b$ which is orthogonal to $\Cyc_d(\Delta; \R)$. The orthogonal of $\Cyc_d(\Delta; \R)$ is the image of $\bound^*$. Let $f$ be any preimage of $a$ by $\bound^*$. Then, we have the following definition.

\begin{definition} \label{5:defi:divprod}
With the above notation, we set
\[ \divprod{b}{b'}_y:=(f,b'). \qedhere \]
\end{definition}

Choosing $f'$ such that $b'=\bound\bound^*f'$, one has
\[ \divprod{b}{b'}_y:=(f,\bound\bound^*f')=(\bound^*f,\bound^*f')_y. \]
Thus, the product is indeed well-defined, symmetric and positive-definite.

\begin{proposition} \label{5:prop:divprod}
If $y$ is a family of $n$ positive real numbers and if $b$ is a $(d-1)$-boundary of $\Delta$, then
\[ \divprod{b}{b}_y=\ratSym_2(\Delta;(b);y):=\frac{\Sym_2(\Delta;(b);y)}{\Sym_2(\Delta;y)}. \]
\end{proposition}

\begin{proof}
Let $a$ be the preimage of $b$ orthogonal to $\Cyc_d(\Delta; \R)$, and let $v$ be a basis of $\Cyc_d(\Delta; \R)$. By Propositions~\ref{5:prop:Symanzik_determinantal_formula} and~\ref{5:prop:Symanzik_with_parameters_determinantal_formula}, after simplification of the factors,
\[ \frac{\Sym_2(\Delta;(b);y)}{\Sym_2(\Delta;y)}=\frac{\norm{\mat v\wedge a}_y^2}{\norm{\mat v}_y^2}=\norm{a}_y^2=\divprod{b}{b}_y.  \qedhere \]
\end{proof}

\begin{remark} \label{5:rem:ratSym_determinantel_formula}
If rows of $U$ are free, then the matrix corresponding to $\bound^*$ is $Y^{-1}U^\tran$. Therefore,
\[ \ratSym_2(\Delta;(b);y)=\divprod{b}{b}_y=(b,(\bound\bound^*)^{-1}b)=b^\tran(UY^{-1}U^\tran)^{-1}b. \]
Following the proof of Proposition~\ref{5:prop:divprod}, one can generalize the previous formula. Let $w$ be a family of vectors in $\vect{u}$. Then
\[ \ratSym_2(\Delta;w;y)=\det(W^\tran(UY^{-1}U^\tran)^{-1}W). \qedhere \]
\end{remark}

\subsection{Symanzik polynomials on triangulable topological spaces}
\label{5:subsec:triangulable_space}

So far, we only discussed the case of discrete objects. Actually, we can define the Symanzik polynomial of any sufficiently nice topological space endowed with a diffuse measure (see details below). This definition is consistent with above definitions when we discretize the space by taking a triangulation.

Let $\S$ be the topological space underlying a finite CW-complex of dimension $d$. Assume there is some finite diffuse measure $\mu$ on $\S$, \ie, a measure which is zero on every subset homeomorphic to a $(d-1)$-dimensional ball. The main idea of this subsection is to introduce a characteristic function of some sufficiently nice $d$-dimensional singular chains. For instance, if $a$ is a function from a $d$-simplex to $\S$ which is a homeomorphism onto its image, one can introduce a function $\omega_a:\S\to\R$ which equals $1$ on the image of $a$ (or $-1$ depending on the orientation), and $0$ elsewhere. One can then introduce the canonical $k$-multilinear symmetric product for this kind of functions with respect to $\mu$. This product will be the continuous analogous of $(\wcdot,\dots,\wcdot)_{k,y}$ and will be sufficient to define a natural Symanzik polynomial $\Sym_k(\S;\mu)$. In the previous description $a$ is a particular $d$-chain. In order to take into account enough singular $d$-chains, we need to use the language of sheaves and of relative singular homology.

\medskip

Let $\smS$ be the subset of points $z$ of $\S$ such that $z$ has a neighborhood homeomorphic to a $d$-dimensional ball (thus, $\smS$ is a topological $d$-manifold). Let $(\Chainsing_*(\S), \boundsing)$ be the complex of singular homology of $\S$ with coefficient in $\Z$, as defined for example in~\cite{Hat02}. Let $\sR$ be the \emph{presheaf of relative homology} defined by $\sR(U):=\Hsing_d(\S,\S\setminus U)$ for any open subset $U$ (we refer to~\cite{Hat02} for a definition of relative homology). We believe that $\sR$ may be a sheaf, but we have not found any proof of this in the literature. Anyway, this is not crucial for our purpose, so we leave it as an open question for the moment. One can restrict $\sR$ to $\smS$ to get the \emph{orientation sheaf} of $\smS$ denoted by $\smsR$. Note that any stalk of $\smsR$ is isomorphic to $\Z$. Though $\smsR$ is not always isomorphic to the constant sheaf $\sZ_{\smS}$, (for instance $\smS$ could be the projective plane), we have a canonical isomorphism of sheaves (see, \eg,~\cite{KS90}*{Proposition 3.3.4})
\[ \smsR\otimes\smsR\cong\sZ_{\smS}. \]

Let $a\in\Chainsing_d(\S)$ be a singular $d$-chain of $\S$ with coefficients in $\Z$. To $a$ we associate the section $\omega_a\in\sR(\S\setminus\supp{\boundsing a})$ where $\supp{\boundsing a}$ is the support of the boundary of $a$ (\ie, the union of the images of terms in $\boundsing a$). If $z\in\S$, providing it is well-defined, we denote by $\omega_a(z)$ the germ of $\omega_a$ at $z$. By an abuse of notations, we also denote by $\omega_a$ the restriction of $\omega_a$ to $\smsR$. We denote $\Chainsingreg_d(\S)$ the vector subspace of all $a\in\Chainsing_d(\S)$ such that $\omega_a$ is defined $\mu$-almost everywhere. In particular, if $a\in\Cyc_d(\S)$, then $\omega_a$ is defined on $\S$, thus $a\in\Chainsingreg_d(\S)$.

One can do the same remarks replacing $\Z$ by $\R$. In particular, if $a\in C_{d,sing}(\S;\R)$ is a singular $d$-chain with coefficients in $\R$, one can associate to $a$ the section
\[ \omega_a\in(\sR\otimes\scR_\S)(\S\setminus\supp{\boundsing a}). \]

We put a (positive semidefinite) multilinear inner product on $\Chainsingreg_d(\S; \R)$: for $k$ chains $a_1, \dots, a_k\in\Chainsingreg_d(\S; \R)$,
\[ (a_1,\dots,a_k)_{k,\mu}:=(\omega_{a_1},\dots,\omega_{a_k})_{L^k(\mu)}:=\int_\S\omega_{a_1}\otimes\cdots\otimes\omega_{a_k}d\mu, \]
where the tensor product is considered as a section of $\scR_U$ on some open set $U$ of $\smS$ of full measure. This inner product is also well-defined on $\Hsing_d(\S;\R)\cong(\sR\otimes\scR_\S)(\S)$. Thus, we have the following definition.

\begin{definition} \label{5:defi:Symanzik_polynomials_topological_space}
With the above notations, the \emph{Symanzik polynomial of $\S$ of order $k$ with respect to the measure $\mu$} is defined by
\[ \Sym_k(\S;\mu):=\card{\Tor(\Hsing_{d-1}(\S))}^k\norm{\mat v}_{k,\mu}^k, \]
where $v$ is any basis of $\Hsing_d(\S)$.

Moreover, if $w$ is a family of $l$ elements of $\Boundsing_{d-1}(\S; \R)$ whose supports have measure zero, then the \emph{Symanzik polynomial of $\S$ of order $k$ with respect to the measure $\mu$ and with parameters $w$} is defined by
\[ \Sym_k(\S; w;\mu):=\card{\Tor(\Hsing_{d-1}(\S))}^k\norm{\mat v\wedge\mat{\~w}}_{k,\mu}^k, \]
where $\~w$ is a family of elements of $\Chainsing_d(\S; \R)$ such that $\boundsing\~w_j=w_j$ for any $j\in\zint1l$.
\end{definition}

\begin{example} \label{5:ex:volume_manifold}
For example, the Symanzik polynomial of any order of a compact orientable manifold endowed with a volume form equals its entire volume (the torsion is trivial by the universal coefficient theorem and Poincaré duality).
\end{example}

Let us now study triangulations. We recall that $\Delta$ is a simplicial complex on some vertex set $V$. Let $m:=\card{\Delta_0}$ be the number of vertices of $\Delta$ and assume that $V=\zint1m$. Set
\[ \begin{array}{rrcl}
\nu_\Delta: & \Delta & \to & \P(\R^m), \\
 & \{i_0, \dots, i_l\} & \mapsto & \conv(e_{i_0}, \dots, e_{i_l}),
\end{array} \]
where $(e_1, \dots, e_m)$ is the standard basis of $\R^m$.
We denote by $\geomreal\Delta$ \emph{the standard geometrical realization of $\Delta$} defined by
\[ \geomreal\Delta:=\bigcup_{\delta\in\Delta}\nu_\Delta(\delta). \]

A \emph{triangulation of $\S$} consists of a simplicial complex $\Delta$ and an application $\Psi$ of the form $\Psi=\Phi\circ\nu_\Delta$ where $\Phi:\geomreal\Delta\to\S$ is a homeomorphism.

If $\S$ comes from a regular finite CW-complex, such a triangulation exists. Fix one of them: $(\Delta, \Psi)$. The dimension of $\Delta$ is also $d$.

\begin{theorem} \label{5:thm:Symanzik_polynomials_triangulation}
With above notations, one has
\[ \Sym_k(\S; \mu)=\Sym_k(\Delta; y), \]
where $y_\delta:=\mu(\Psi(\delta))$ for every $\delta\in\Delta_d$.
\end{theorem}

\begin{proof}
The homeomorphism $\Psi$ induces a map $\Psi':\Chain_d(\Delta)\to\Chain_{d,sing}(\S)$. If $\delta$ is a facet of $\Delta$, for each $z\in\interior(\Psi(\delta))$, $\omega_{\Psi'(\delta)}(z)$ is a generator of the stalk of $\smsR$ over $z$, and for each $z\in\smS\setminus\Psi(\delta)$, $\omega_{\Psi'(\delta)}(z)=0$. Thus, for $k$ facets $\delta_1\dots,\delta_k$ of $\Delta$,
\begin{align*}
(\Psi'(\delta_1),\dots,\Psi'(\delta_k))_\mu
  &= \begin{cases}
    \mu(\Psi(\delta_1)) & \quad\text{if $\delta_1=\dots=\delta_k$,} \\
    0 & \quad\text{otherwise}
    \end{cases} \\
  &= (\delta_1, \dots, \delta_k)_y.
\end{align*}
Moreover, it is well-known that $\Psi'$ induces an isomorphism $\Psi_*:\Homl_d(\Delta)\simto\Hsing_d(\Delta)$~\cite{Hat02}*{Theorem 2.27}. Thus, $\Psi_*$ maps bases of $\Cyc_d(\Delta)$ to bases of $\Hsing_d(\S)$. If $v$ is a basis of $\Cyc_d(\Delta)$, and if $w$ is its image by $\Psi'$, the last equality implies that $\norm{\mat w}_{k,\mu}^k=\norm{\mat v}_{k,y}^k$. Moreover, Theorem 2.27 of~\cite{Hat02} also implies
\[ \card{\Tor(\Hsing_{d-1}(\S))}^k=\card{\Tor(\Homl_{d-1}(\Delta))}^k. \]
We conclude by Definition~\ref{5:defi:Symanzik_polynomials_topological_space}, Proposition~\ref{5:prop:Symanzik_determinantal_formula} and Remark~\ref{5:rem:thm:Kirchhoff_s}.
\end{proof}

Thus, the specialization of the polynomial does not depend on the triangulation. Let us study the abstract meaning of this fact.

\subsection{Geometrical factorization of Symanzik polynomials}
\label{5:subsec:factorization}

In this subsection, we use Theorem~\ref{5:thm:Symanzik_polynomials_triangulation} to prove that Symanzik polynomials of simplicial complexes are stable by subtriangulation in the sense of Corollary~\ref{5:cor:Symanzik_polynomials_triangulation}. Moreover, we deduce a canonical factorization of the Symanzik polynomial of a simplicial complex.

Here is a corollary of Theorem~\ref{5:thm:Symanzik_polynomials_triangulation}. Let $\Delta$ be any $d$-dimensional simplicial complex with $n$ facets. Let $(\Gamma, \Psi)$ be a triangulation of $\geomreal{\Delta}$. Assume that there exists a map $\varphi:\Gamma_d\to\Delta_d$ such that, for any facet $\gamma$ of $\Gamma$, $\Psi(\gamma)\subset\nu_\Delta(\varphi(\gamma))$. In this case, we call $(\Gamma, \Psi)$ a \emph{subtriangulation of $\Delta$}. Set $n':=\card{\Gamma_d}$.

\begin{corollary} \label{5:cor:Symanzik_polynomials_triangulation}
Let $x'$ be a family of $n'$ variables. For $\delta\in\Delta_d$, let
\[ x_\delta:=\sum_{\gamma\in\varphi^{-1}(\delta)}x'_\gamma. \]
Then
\[ \Sym_k(\Gamma; x')=\Sym_k(\Delta; x). \qedhere \]
\end{corollary}

\begin{proof}
Let $y'$ be a family of $n'$ positive real numbers. It is not hard to create a finite diffuse measure $\mu$ on $\geomreal{\Delta}$ such that $\mu(\Psi(\gamma))=y'_\gamma$ for every $\gamma\in\Gamma_d$. Thus, considering $(\Delta, \nu_\Delta)$ as a second triangulation of $\geomreal{\Delta}$, Theorem~\ref{5:thm:Symanzik_polynomials_triangulation} implies
\[ \Sym_k(\Gamma; y')=\Sym_k(\Delta; y), \]
where, for $\delta\in\Delta_d$,
\[ y_\delta=\sum_{\gamma\in\varphi^{-1}(\delta)}y'_\gamma. \]
As this is true for any family $y'$ of positive numbers, the equality holds in $\Z[x']$.
\end{proof}

One can go further in the factorization. Let us use notations of the last subsection.
We say that two points $z,z'\in\smS$ are equivalent if, for every $c\in\Hsing_d(\S)$, $\omega_c(z)$ is zero if and only if $\omega_c(z')$ is zero. In particular, all points of a connected component of $\smS$ are equivalent. Equivalent classes form a partition $\smS_1, \dots, \smS_l$ of $\smS$ (notice that some $\smS_j$ could contain several connected components: see Example~\ref{5:ex:S_j_with_several_connected_component}). In what follows, we do not consider the possible class of points $z$ belonging to no cycle, \ie, such that $\omega_c(z)=0$ for every cycle $c$. Let $j\in\zint1l$, $z_j\in\smS_j$ and $c_j\in\Hsing_d(\S)$ be such that $|\omega_{c_j}(z)|\in\Z_+$ is nonzero and minimal. Finally, let $\tau_j\in\smsR(\smS)$ be zero on $\compl{(\smS_j)}$ and equal to $\omega_{c_j}\rest{\smS_j}$ on $\smS_j$.

\begin{proposition} \label{5:prop:factorization_of_Symanzik}
With above notations, $\tau_j$ does not depend on $z_j$ nor $c_j$. Moreover, there exists a polynomial $P_\S\in\Z[T_1, \dots, T_l]$ such that, for every finite diffuse measure $\mu$ on $\S$,
\[ \Sym_k(\S;\mu)=P_\S(\norm{\tau_1}_{L^k(\mu)}^k, \dots, \norm{\tau_l}_{L^k(\mu)}^k). \]
\end{proposition}

\begin{proof}
Let $v$ be a basis of $\Hsing_d(\S)$. Let $n$ be the size of $v$. The Symanzik polynomial equals $P_1((v_{i_1},\dots,v_{i_k})_{k,\mu})_{i_1,\dots,i_k\in\zint1n}$, where $P_1$ is the integer polynomial on $n^k$ variables corresponding to $\Tor(\Hsing_{d-1}(\S))^k\cdot\det$.

Let $c\in\Hsing_d(\S)$. Then
\begin{equation} \label{5:eqn:factorization_of_Symanzik}
\omega_c=\sum_{j=1}^l\frac{\omega_c(z_j)}{\tau_j(z_j)}\tau_j.
\end{equation}
Indeed, for the sake of a contradiction, assume there exists $z\in\smS$ such that $\omega_c(z)\neq\omega'_c(z)$, where $\omega'_c$ denote the right-hand side. Let $j\in\zint1l$ be such that $z\in\smS_j$. Then the cycle
\[ c':=\frac{\omega_c(z_j)}{\tau_j(z_j)}c_j-c \]
verifies $\omega_{c'}(z_j)=0$, but $\omega_{c'}(z)=\omega'_c(z)-\omega_c(z)\neq0$. This is absurd since $z$ and $z_j$ are equivalent.

In particular, if $z\in\smS_j$, then
\[ \frac{\omega_c(z)}{\tau_j(z)}=\frac{\omega_c(z_j)}{\tau_j(z_j)}. \]
This implies that $\tau_j$ does not depend on $z_j$ and $c_j$.

Notice that the ratios in~\eqref{5:eqn:factorization_of_Symanzik} are integers by the minimality condition on $c_j$. Rewriting the elements of $v$ in terms of the $\tau_j$, we get that the Symanzik polynomial equals
\[ P_2((\tau_{j_1},\dots,\tau_{j_k})_{k,\mu})_{j_1,\dots,j_k\in\zint1l} \]
for some integer polynomial $P_2$. But clearly, if $j_1, \dots, j_k\in\zint1l$, then
\[ (\tau_{j_1},\dots,\tau_{j_k})_{k,\mu}=
\begin{cases}
\norm{\tau_{j_1}}_{L^k(\mu)}^k &\text{if $j_1=\dots=j_k$}, \\
0 &\text{otherwise.}
\end{cases} \]
Finally, there exists $P_\S\in\Z[T_1, \dots, T_l]$ such that
\[ \Sym_k(\S;\mu)=P_\S(\norm{\tau_1}_{L^k(\mu)}^k, \dots, \norm{\tau_l}_{L^k(\mu)}^k). \qedhere \]
\end{proof}

Now suppose that $\S=\geomreal{\Delta}$ for some simplicial complex $\Delta$. For every $\delta\in\Delta_d$, $\interior(\nu_\Delta(\delta))$ is included in a connected component of $\smS$, thus, in a $\smS_j$ for some $j\in\zint1l$. We pick a point $z_\delta$ in the interior of $\nu_\Delta(\delta)$ for every facet $\delta\in\Delta_d$. Then we set, for $j\in\zint1l$,
\[ Q_j(x)=\sum_{\delta\in\Delta_d}(\tau_j(z_\delta))^{\otimes k}x_\delta. \]
Using the argument we used to deduce Corollary~\ref{5:cor:Symanzik_polynomials_triangulation} from Theorem~\ref{5:thm:Symanzik_polynomials_triangulation}, we infer the following corollary from Proposition~\ref{5:prop:factorization_of_Symanzik}.

\begin{corollary} \label{5:cor:factorization_of_Symanzik}
With above notations,
\[ \Sym_2(\Delta; x)=P_{\geomreal\Delta}(Q_1(x), \dots, Q_l(x)). \qedhere \]
\end{corollary}

\begin{example}
For example, if $\Delta$ is the bipyramid (Figure~\ref{5:fig:bipyramid}), $\geomreal\Delta$ has three equivalence classes: the equatorial plane and both pyramids. The corresponding factorized polynomial is
\[ \Sym_k(\Delta; x)=(x_1+x_2+x_3)x_4+(x_1+x_2+x_3)(x_5+x_6+x_7)+x_4(x_5+x_6+x_7). \qedhere \]
\end{example}

\begin{example} \label{5:ex:S_j_with_several_connected_component}
Let $\S$ be the following $\Delta$-complex. We take two copies of Figure~\ref{5:fig:real_projective_plane} and we identify the four edges labeled by $a$, and the four edges labeled by $b$. Then $\smS$ has two connected components: the interiors of each copy. But it has only one equivalent class because $\rk(\Cyc_2(\S))=1$.
\end{example}

\begin{example} \label{5:ex:possible_Symanzik_polynomials}
For any $U\in\M_{p,l}(\Z)$, it is not hard to create a regular CW-complex $\S$ (of dimension $2$) such that $U$ is the matrix associated to the cellular boundary operator of $\S$. Thus, triangulating $\S$, there is a simplicial complex $\Delta$ and a factorization of $\Sym_2(\Delta; x)$ of the form $\det(UX'U^\tran)$. Hence, one can somehow obtain any possible Symanzik polynomial from simplicial complexes.
\end{example}

\subsection{Symanzik polynomials with parameters and geometry}
\label{5:subsec:geometry_of_parameters}

We have seen a geometrical interpretation of Symanzik polynomials. In this subsection, we tackle the case of Symanzik polynomials with parameters. We will see in Section~\ref{5:sec:deletion_contraction} that adding a parameter corresponds to doing a \emph{matroidal contraction}. However, this matroidal contraction has, a priori, no concrete geometrical interpretation in general. Nevertheless, in this subsection, we show that, in some simple cases, adding a parameter is equivalent to doing a topological contraction, and we give insight on how to deal with the parameters geometrically for more complex cases.

Recall that $\Delta$ is a simplicial complex of dimension $d$. Let $a$ be a $(d-1)$-chain of $\Delta$. If $\delta\in\Delta_{d-1}$, then $[\delta]a\in\Z$ will denote the coefficient of $\delta$ in $a$. The \emph{support of $a$}, denoted by $\supp a$, is the set of $(d-1)$-faces associated to a nonzero coefficient in $a$. One can see $\Delta$ as a CW-complex. Then, $\rquot{\Delta}{\supp{a}}$ denotes the CW-complex obtained as the result of contracting elements in $\supp{a}$ (see the following example).

\begin{example}
The cellular homology of $\rquot\Delta{\supp a}$ is naturally isomorphic to the relative simplicial homology $\Homl_*(\Delta,\Delta\rest{\supp a})$ (see~\cite{Hat02}). For instance, $\Chain_{d-1}(\rquot\Delta{\supp a})=\rquot{\Chain_{d-1}(\Delta)}{\freeZmod{\supp{a}}}$. If $U$ is the $d$-th incidence matrix of $\Delta$, then the $d$-th incidence matrix of $\rquot{\Delta}{\supp{a}}$ is obtained by removing rows of $U$ corresponding to elements of $\supp{a}$. Hence we can still define the Symanzik polynomial of $\rquot{\Delta}{\supp{a}}$.

In fact, if $(\Delta, \Psi)$ is the triangulation of some CW-complex $\S$, and if $\mu$ is some diffuse finite measure on $\S$, then
\[ \Sym_2(\rquot{\Delta}{\supp{a}}; y)=\Sym_2(\rquot{\S}{\bigcup_{\alpha\in\supp{a}}{\Psi(\alpha)}}; \~\mu), \]
where, for $\delta\in\Delta_d$, $y_\delta=\mu(\Psi(\delta))$, the quotient on the right-hand side is the topological contraction, and $\~\mu$ is the induced measure on the contracted space.
\end{example}

\begin{definition} \label{5:defi:simple}
Let $b\in\Bound_{d-1}(\Delta)$ be a nonzero boundary. Then $b$ is a \emph{simple boundary} of $\Delta$ if it is a generator of $\freeZmod{\supp b}\cap\Bound_{d-1}(\Delta)$.

Let $b, b'\in\Bound_{d-1}(\Delta)$ be two boundaries. Then they are called \emph{cosimple boundaries} if they are simple and if they form an independent family of maximal rank of $\freeZmod{\supp b\cup\supp{b'}}\cap\Bound_{d-1}(\Delta)$. In this case, we define:
\[ \theta_{b,b'}:=\frac{\norm{\Zmod{(b,b')}}}{\norm{\freeZmod{\supp b\cup\supp{b'}}\cap\Bound_{d-1}(\Delta)}}. \]
\end{definition}

\begin{proposition} \label{5:prop:simple_boundaries}
For any $(d-1)$-boundary $b\in\Bound_{d-1}(\Delta;\R)$ with real coefficients, there exist simple boundaries $b_1, \dots, b_l$, pairwise cosimple, and real numbers $\lambda_1, \dots, \lambda_l$ such that
\[ b=\lambda_1b_1+\dots+\lambda_lb_l. \]
\end{proposition}

\begin{proof}
Let $D$ be a maximal subset of $\supp{b}$ such that $\rk(\Cyc_d(\rquot\Delta D))=\rk(\Cyc_d(\Delta))$. Let $\delta\in\supp{b}\setminus D$. Then $\rquot{\Cyc_d(\Delta/(D+\delta))}{\Cyc_d(\Delta)}$ has rank $1$. Let $\~a_\delta$ be a generator of this quotient. Let $a_\delta\in\Chain_d(\Delta)$ be a representative of $\~a_\delta$. Let $b_\delta$ be its boundary. One can easily check that $(b_\delta)_{\delta\in\supp{b}\setminus D}$ forms a family of simple and pairwise cosimple boundaries, and that
\[ b=\sum_{\delta\in\supp{b}\setminus D}\frac{[\delta]b}{[\delta]b_\delta}b_\delta. \qedhere \]
\end{proof}

\begin{remark} \label{5:rem:Symanzik_and_contraction}
Let $b, b'$ be cosimple boundaries of $\Delta$ and let $a, a'$ be $d$-chains such that $\bound a=b$ and $\bound a'=b'$. Let $v$ be a basis of $\Cyc_d(\Delta)$. Then, $v\concat (a)$ is a basis of $\Cyc_d(\rquot\Delta{\supp b})$, and that $v\concat (a, a')$ is a basis of the vector subspace $\vect{\Cyc_d(\rquot\Delta{(\supp b\cup\supp{b'})})}$. Thus, by Propositions~\ref{5:prop:Symanzik_determinantal_formula},~\ref{5:prop:Symanzik_with_parameters_determinantal_formula} and Remark~\ref{5:rem:thm:Kirchhoff_s},
\begin{gather*}
\Sym_k(\Delta; x)=\Tor(H_{d-1}(\Delta))^k\norm{\mat v}_{k,x}^k, \\
\Sym_k(\Delta; (b); x)=\Tor(H_{d-1}(\Delta))^k\norm{\mat v\wedge a}_{k,x}^k, \\
\Sym_k(\rquot{\Delta}{\supp b}; x)=\Tor(H_{d-1}(\rquot{\Delta}{\supp b}))^k\norm{\mat v\wedge a}_{k,x}^k, \\
\Sym_k(\Delta; (b, b'); x)=\Tor(H_{d-1}(\Delta))^k\norm{\mat v\wedge a\wedge a'}_{k,x}^k, \\
\Sym_k(\rquot{\Delta}{(\supp b\cup\supp{b'})}; x)=\frac{\Tor\big(H_{d-1}(\rquot{\Delta}{(\supp b\cup\supp{b'})})\big)^k}{\theta_{b,b'}^k}\norm{\mat v\wedge a\wedge a'}_{k,x}^k.
\end{gather*}
The last denominator comes from the following diagram:
\[ \begin{tikzcd}
0 \arrow[r] & \Zvect{v} \arrow[no head, d, "\sim"{sloped, above, inner sep=1pt}] \arrow[r] & \Zvect{v\concat(a,a')} \arrow[r, "\bound"] & \freeZmod{\supp{b}\cup\supp{b'}}\cap\Bound_{d-1}(\Delta) \arrow[r] & 0 \\[-1.3em]
0 \arrow[r] & \Zmod{v} \arrow[r] & \Zmod{(v\concat(a,a'))} \arrow[u, draw=none, "\subset"{sloped, above}] \arrow[r, "\bound"] & \Zmod{(b,b')} \arrow[u, draw=none, "\subset"{sloped, above}] \arrow[r] & 0
\end{tikzcd} \]

Thus, adding a parameter which is a simple boundary is equivalent to contracting topologically the support of this boundary, up to a factor. This is still true for two parameters which are cosimple.
\end{remark}

Let $y$ be a family of $n$ variables and $b, b'\in\Bound_{d-1}(\Delta;\R)$. We recall from Section~\ref{5:subsec:divisor} that
\[ \divprod{b}{b'}_y=(h,h')_y, \]
where $h$, \resp $h'$, is the unique preimage of $b$, \resp $b'$, by $\bound$ that is orthogonal to $\Cyc_d(\Delta;\R)$ for the scalar product $(\wcdot,\wcdot)_y$.

Thanks to Proposition~\ref{5:prop:simple_boundaries} and to Remark~\ref{5:rem:Symanzik_and_contraction}, the following proposition allows us to deal with more complex cases for order 2.

\begin{proposition} \label{5:prop:geometry_of_parameters}
Let $y$ be a family of $n$ positive real numbers. If $v$ is a basis of $\Cyc_d(\Delta)$ and if $a_1, \dots, a_l$ are some $d$-chains of $\Delta$ with real coefficients, then
\begin{gather*}
\norm{\mat v\wedge(a_1+\dots+a_l)}_y^2=\sum_{i=1}^l\norm{\mat v\wedge a_i}_y^2+2\norm{\mat v}_y^2\sum_{1\leq i<j\leq l}\divprod{\bound a_i}{\bound a_j}_y\quad\text{and} \\
\norm{\mat v}_y^2\divprod{\bound a_i}{\bound a_j}_y=\sign(\divprod{\bound a_i}{\bound a_j}_y)\sqrt{\norm{\mat v\wedge a_i}_y^2\norm{\mat v\wedge a_j}_y^2-\norm{\mat v}_y^2\norm{\mat v\wedge a_i \wedge a_j}_y^2}.
\end{gather*}
\end{proposition}

\begin{proof}
Let $\pi$ be the orthogonal projection onto $\vect v=\Cyc_d(\Delta,\R)$ for the scalar product $(\wcdot,\wcdot)_y$. For each $i\in\zint1n$, let $h_i:=a_i-\pi(a_i)$. Then,
\[ \norm{\mat v\wedge(a_1+\dots+a_l)}_y^2=\norm{\mat v}_y^2\norm{h_1+\dots+h_l}_y^2=\norm{\mat v}_y^2\sum_{i=1}^l\norm{h_i}_y^2+2\norm{\mat v}_y^2\sum_{1\leq i<j\leq l}(h_i,h_j)_y. \]
We have $(h_i,h_j)_y=\divprod{\bound a_i}{\bound a_j}_y$. Moreover one can factorize $\norm{\mat v}_y\norm{h_i}_y$ into $\norm{\mat v\wedge a_i}_y$. The following calculation gives us the second part of the proposition:
\[ \norm{\mat v\wedge a_i\wedge a_j}_y^2=\norm{\mat v}_y^2\norm{h_i\wedge h_j}_y^2=\norm{\mat v}_y^2\big(\norm{h_i}_y^2\norm{h_j}_y^2-(h_i,h_j)_y^2\big). \qedhere \]
\end{proof}

\begin{remark}
Notice that
\[ P(y):=\norm{\mat v}_y^2\divprod{\bound a_i}{\bound a_j}_y=(\mat v\wedge a_i,\mat v\wedge a_j)_y \]
is a polynomial in $y$. More precisely, it is one of the two square roots of the polynomial $\norm{\mat v\wedge a_i}_y^2\norm{\mat v\wedge a_j}_y^2-\norm{\mat v}_y^2\norm{\mat v\wedge a_i \wedge a_j}_y^2$. Let $J\subset\zint1n$ and $\alpha_J\in\R^*$ be such that $\alpha_Jy^{\compl J}$ is a nonzero monomial of $P(y)$. It is enough to find the sign of $\alpha_J$ to infer which square root is $P(y)$. The set $J$ corresponds to a $1$-forest of $\Delta$. Let $\delta\in\compl J$ be a facet such that $\Gamma:=\Delta\rest{J+\delta}$ is a $0$-forest of $\Delta$. Then let $a_i', a_j'\in\Chain_d(\Gamma; \R)$ be such that $\bound a_i'=\bound a_i$ and $\bound a_j'=\bound a_j$. Then,
\begin{gather*}
P(y)=(\mat v\wedge a_i',\mat v\wedge a_j')_y, \text{ and} \\
\alpha_J=(\mat v\wedge a_i',\mat e_{\compl J})(\mat v\wedge a_j', \mat e_{\compl J})=(\pm\det(V_{\compl I})[\delta]a_i')(\pm\det(V_{\compl I})[\delta]a_j'),
\end{gather*}
where the last equality comes from the Laplace cofactor expansion along the last row, and where both signs are equal. Finally, one only has to look to the orientations of $\delta$ in $a_i'$ and $a_j'$:
\[ \sign(\alpha_J)=\sign([\delta]a_i'[\delta]a_j'). \qedhere \]
\end{remark}

\begin{example}
\newcommand{\drawG}{\raisebox{-1em}{\tikz[scale=0.6]{
  \scriptsize
  \draw (0,0) node {$\bullet$};
  \draw (0,1) node {$\bullet$};
  \draw (0,0) to[bend left=90] node[midway] {$\bullet$} (0,1);
  \draw (0,0) to node[midway] {$\bullet$} (0,1);
  \draw (0,0) to[bend right=90] node[midway] {$\bullet$} (0,1);}}}
\newcommand{\drawGl}{\raisebox{-1em}{\tikz[scale=0.6]{
  \scriptsize
  \draw (0,0) node {$\bullet$};
  \draw (0,1) node {$\bullet$};
  \draw (0,0) to[bend left=80] (0,0.5) to[bend left=80] (0,1);
  \draw (0,0) to node[midway] {$\bullet$} (0,1);
  \draw (0,0) to[bend right=80] node[midway] {$\bullet$} (0,1);}}}
\newcommand{\drawGr}{\raisebox{-1em}{\tikz[scale=0.6]{
  \scriptsize
  \draw (0,0) node {$\bullet$};
  \draw (0,1) node {$\bullet$};
  \draw (0,0) to[bend left=80] node[midway] {$\bullet$} (0,1);
  \draw (0,0) to node[midway] {$\bullet$} (0,1);
  \draw (0,0) to[bend right=80] (0,0.5) to[bend right=80] (0,1);}}}
\newcommand{\drawGlr}{\raisebox{-1em}{\tikz[scale=0.6]{
  \scriptsize
  \draw (0,0) node {$\bullet$};
  \draw (0,1) node {$\bullet$};
  \draw (0,0) to[bend left=80] (0,0.5) to[bend left=80] (0,1);
  \draw (0,0) to node[midway] {$\bullet$} (0,1);
  \draw (0,0) to[bend right=80] (0,0.5) to[bend right=80] (0,1);}}}

The results of this subsection let us make computation like the following one. Consider the following graph $G$.

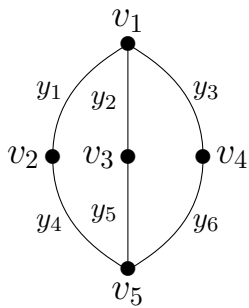
\begin{center}
\begin{tikzpicture}
\coordinate (v1) at (0,1.5);
\coordinate (v2) at (-1,0);
\coordinate (v3) at (0,0);
\coordinate (v4) at (1,0);
\coordinate (v5) at (0,-1.5);

\Large
\draw (v1) node {$\bullet$} node[above] {$v_1$};
\draw (v2) node {$\bullet$} node[left] {$v_2$};
\draw (v3) node {$\bullet$} node[left] {$v_3$};
\draw (v4) node {$\bullet$} node[right] {$v_4$};
\draw (v5) node {$\bullet$} node[below] {$v_5$};

\normalsize
\draw (v1)to[out=-150, in=90] node[midway, left] {$y_1$} (v2) to[out=-90, in=150] node[midway, left] {$y_4$} (v5);
\draw (v1)--(v3) node[midway, left] {$y_2$} --(v5) node[midway, left] {$y_5$};
\draw (v1)to[out=-30, in=90] node[midway, right] {$y_3$} (v4) to[out=-90, in=30] node[midway, right] {$y_6$} (v5);
\end{tikzpicture}
\end{center}

For example, we have the following decomposition into simple boundaries
$2v_2-v_3-v_4=b+b'$ where $b=v_2-v_3$ and $b'=v_2-v_4$. Since we are in a graph, all factors from Remark~\ref{5:rem:Symanzik_and_contraction} are trivial. Moreover, we have $\sign(\divprod{b}{b'}_y)$ is constant equal to $+1$. Hence,
\begin{eqnarray*}
\lefteqn{\Sym_2(G;(2v_2-v_3-v_4);y)} \hspace{2.5cm} & \\
  & = & \Sym_2(G;(b);y)+\Sym_2(G;(b');y) \\
  && +2\sqrt{\Sym_2(G;(b);y)\Sym_2(G;(b');y)-\Sym_2(G;y)\Sym_2(G;(b,b');y)} \\
  & = & \Sym_2(\drawGl;y)+\Sym_2(\drawGr;y) \\
  && +2\sqrt{\Sym_2(\drawGl;y)\Sym_2(\drawGr;y)-\Sym_2(\drawG;y)\Sym_2(\drawGlr;y)}.
\end{eqnarray*}
For example, if we set $y_1=\dots=y_6=1$, then it remains to count the number of spanning trees:
\[ \Sym_2(G;(2v_2-v_3-v_4);y)=12+12+2\sqrt{12\cdot12-12\cdot9}=36. \qedhere \]
\end{example}

%%%
\section{Generalization to matroids over hyperfields} \label{5:sec:deletion_contraction}

We have seen that many properties of the classical Kirchhoff and Symanzik polynomials on graphs are still true for our generalization. In this section, we wish to generalize two very important properties of the classical case: the deletion contraction formula~\eqref{5:eqn:deletion_contraction_graph} and the partial factorization~\eqref{5:eqn:partial_factorization}.

To generalize these formul\ae, we need to define deletion and contraction for family of vectors. The natural objects to consider for contractions and deletions are matroids. We will recall the definitions and the main properties in a first subsection.

But classical matroids do not contain enough data. We have to assign some coefficients to the bases. Recently, in~\cite{BB16}, Baker and Bowler define matroids over hyperfields. To do so, they use Grassman-Plücker functions. These will be the main tool to achieve our generalization of these formul\ae. Though we will work in the general setting of hyperfields, to avoid technicalities one might only consider the case of fields, which will be sufficient for the case of family of vectors.

\subsection{Reminder on matroids}
\label{5:subsec:matroids}

We begin with recalling basic definitions and properties about matroids without any proof. We refer to~\cite{Oxl11} for more information.

A matroid can have many equivalent, or \emph{cryptomorphic}, definitions. Here is a first one.

\begin{definition} \label{5:defi:matroid}
A \emph{matroid} $\Ma$ is a pair which consists of a \emph{ground set} $\E$, which can be any finite set, and a family of \emph{independent sets} $\Ind$, which is a subset of $\P(\E)$. We write $\Ma=(\E, \Ind)$. A matroid has to verify three axioms:
\begin{enumerate}
\item $\emptyset\in\Ind$,
\item (hereditary property) $\Ind$ is stable by inclusion ($J\subset I\in\Ind\Rightarrow J\in\Ind$),
\item (augmentation property) if $I, J\in\Ind$ and if $\card J<\card I$, then there exists $i\in I\setminus J$ such that $J+i\in\Ind$. \label{5:defi:matroid:augmentation_property}  \qedhere
\end{enumerate}
\end{definition}

One could think of $E$ as a set of vectors generating some vector space, and $\Ind$ as the set of free subfamilies of $E$ (though, there are matroids which have not such a representation).

Let $\Ma=(E,\Ind)$ be a matroid. If $I\subset\E$, we define the \emph{rank of $I$} by
\[ \rk(I):=\max_{J\in\P(I)\cap\Ind}\card J. \]
The \emph{rank of $\Ma$} is the rank of $E$. We call the \emph{closure of $I$} the set
\[ \cl(I):=\{i\in\E \st \rk(I\cup\{i\})=\rk(I)\}\subset\E. \]
This is a closure operator: it is extensive, increasing and idempotent.

A \emph{basis} is an independent set maximal for the inclusion. The set of all bases is denoted by $\Bas(\Ma)$. If $l$ is a nonnegative integer, the set of all independents of rank $l$ is denoted by $\Ind_l$. Note that we have, $\Bas(\Ma)=\Ind_{\rk(\Ma)}$ and that, for a given ground set $E$, a matroid is characterized by its bases. The bases verify the following \emph{exchange property}:
if $B_1, B_2$ are two different bases, then for any $i\in B_1\setminus B_2$, there exists $j\in B_2\setminus B_1$ such that $B_1-i+j$ (\ie, $B_1\setminus\{i\}\cup\{j\}$) is a basis.

An element $e\in E$ is said to be:
\begin{itemize}
\item a \emph{loop} if $\{e\}\not\in\Ind$,
\item \emph{free} if it is in every basis, and
\item \emph{proper} otherwise.
\end{itemize}

\medskip

Let $I$ be a subset of $E$. One can delete $I$ from $\Ma$ to get a new matroid denoted by $\Ma\del I$ and defined by
\[ \Ma\del I:=(E\setminus I, \Ind\cap\P(E\setminus I)). \]
We also define the restriction of $\Ma$ to $I$ by $\Ma\mrest I:=\Ma\del\compl I$.

One can also contract $I$ in $\Ma$ to get a new matroid denoted by $\Ma\contr I$ and defined by
\[ \Ma\contr I:=(E\setminus I, \Ind'), \]
where
\begin{align*}
\Ind' &= \{J\subset E\setminus I \st \exists B\in\Bas(\Ma\mrest I), J\cup B\in\Ind \} \\
  &= \{J\subset E\setminus I \st \forall B\in\Bas(\Ma\mrest I), J\cup B\in\Ind \}.
\end{align*}

Concerning the bases, we have the following characterizations:
\begin{equation} \label{5:eqn:contraction_deletion_bases}
\begin{split}
\Bas(\Ma\mrest I)=\{B\cap I \st B\in\Bas(\Ma)\text{ and $\card{B \cap I}$ is maximal} \},\text{ and} \\
\Bas(\Ma\contr I)=\{B\cap \compl I \st B\in\Bas(\Ma)\text{ and $\card{B \cap \compl I}$ is minimal} \}.
\end{split}
\end{equation}
Moreover, if $I\subset E$ and $B\subset E$, any two of the following statements imply the third one.
\begin{flalign} \label{5:eqn:two_out_of_three_bases}
\begin{cases}
\bullet\ B\cap I\in\Bas(\Ma\mrest I),\\
\bullet\ B\cap\compl I\in\Bas(\Ma\contr I), \\
\bullet\ B\in\Bas(\Ma).
\end{cases} &&
\end{flalign}
From~\eqref{5:eqn:contraction_deletion_bases} and~\eqref{5:eqn:two_out_of_three_bases}, we deduce the following equivalence.
\begin{equation} \label{5:eqn:equivalence_bases}
B\in\Bas(\Ma)\text{ and }\card{B\cap I}=\rk(I)\quad\iff\quad
B\cap I\in\Bas(\Ma\mrest I)\text{ and }B\cap\compl I\in\Bas(\Ma\contr I).
\end{equation}

The dual of the matroid $\Ma$ will be denoted by $\hat\Ma$ and is defined by
\[ \Bas(\hat\Ma)=\{\compl B\st B\in\Bas(\Ma)\}. \]
This is a duality in the sense that we have $\hat{\hat\Ma}=\Ma$. Moreover, $\hat{\Ma\del I}=\hat\Ma\contr I$.

\subsection{Kirchhoff and Symanzik polynomials of a matroid}

Let $n$ be a positive integer, $E:=\zint1n$ be a finite set and $\Ma=(\E, \Ind)$ be a matroid of rank $r$.

If $u=(u_1, \dots, u_n)$ is a family of elements of a $\Z$-module, we denote by $\Ma_{u}=(\E_{u}, \Ind_{u})$, with $\E_u=\zint1n$ and
\[ \Ind_u=\{I\subset E\st \text{$u_I$ is free}\}, \]
the \emph{matroid representing the family $u$}.

Here is the definition of Kirchhoff and Symanzik polynomials of a matroid.
\begin{definition} \label{5:defi:polynomials_for_matroids}
The \emph{Kirchhoff polynomial (of order $0$) of the matroid $\Ma$ with variables $x_1, \dots, x_n$} is defined by
\[ \Kir_0(\Ma;x)=\sum_{I\in\Bas(\Ma)}x^I. \]
The \emph{Symanzik polynomial (of order $0$) of the matroid $\Ma$ with variables $x_1, \dots, x_n$} is defined by
\[ \Sym_0(\Ma;x)=\sum_{I\in\Bas(\hat\Ma)}x^I. \qedhere \]
\end{definition}

This definition is natural because of the following claim, whose proof is straightforward.

\begin{claim} \label{5:claim:matrices_to_matroids}
Let $u$ be a family of $n$ elements in $\Z^p$. Then
\begin{gather*}
\Kir_0(u;x)=\Kir_0(\Ma_{u};x), \\
\Sym_0(u;x)=\Sym_0(\Ma_{u};x).
\end{gather*}
\end{claim}

The analogous for matroids of the duality Theorem~\ref{5:thm:duality} is given by the following claim, which is a direct corollary of Claim~\ref{5:claim:matrices_to_matroids} and of Theorem~\ref{5:thm:duality} for $k=0$.

\begin{claim} \label{5:claim:duality_matroids}
Let $q$ be a positive integer and $v$ be a family of $n$ vectors in $\Z^q$ such that $v^\tran$ spans $\ker(U)$, then $\Ma_{u}=\hat\Ma_{v}$.
\end{claim}

Next proposition shows that parameters of Symanzik polynomials correspond to a so-called \emph{quotient} of the matroid.

If $w$ is a family of elements of $\vect{u}$, then $\Ma_{u\concat w}$ is an extension of $\Ma_{u}$, \ie, we have $\Ma_{u}$ equals $\Ma_{u\concat w}\mrest{\zint1n}$.

\begin{proposition}
Let $w$ consist of independent elements of $\vect{u}$. Then
\[ \Sym_0(u; w; x)=\Sym_0(\Ma_{u\concat w}\contr{\zint{n+1}{n+l}};x). \]
\end{proposition}

\begin{proof}
Looking at Definition~\ref{5:defi:Symanzik_polynomials_with_parameters}, the left-hand side is the sum of $x^{\compl I}$ over all $I\subset\zint1n$ of size $r-l$ such that $u_I\concat w$ is an independent family. But such sets $I$ exactly correspond to the bases of the matroid $\Ma_{u\concat w}\contr\zint{n+1}{n+l}$. Thus, Definition~\ref{5:defi:polynomials_for_matroids} ends the proof.
\end{proof}

\begin{remark} \label{5:rem:non_zero_determinants}
In particular, if $I\subset\zint1n$ is of size $r-l$, then $x^{\compl I}$ appears with a nonzero coefficient in $\Sym_k(u;w;x)$ only if $I$ is in $\Ind_u$. Therefore, if $\Sym_k(u;w;x)=\norm{\~{\mat v}^\tran}_{k,x}^k$ for some family $\~v$, then $\det(\~V_{\compl I})\neq0$ implies that $I\in(\Ind_u)_{r-l}$. This implication is an equivalence in the case $l=0$.
\end{remark}

In the next subsection, we generalize the definitions to higher orders.

\subsection{Grassman-Plücker functions} \label{5:subsec:hyperfield}

In this subsection we introduce the main tool to work with orders higher than $0$. We will work with hyperfields, which are similar to fields except that the operation of addition is multivalued. We will quickly reproduce the axioms from~\cite{BB16} below and refer for more information and examples to the literature (see, \eg,~\cites{BB16, BL21}). However, if the reader prefers, it is possible to think of $\FF$ as a field (in this case, \enquote{$\ni 0$} have to be replaced by \enquote{$= 0$}).

In a hyperfield $\FF$, the addition is a \emph{hyperoperation}, \ie, a map $\hp: \FF\times \FF\to\P(\FF)\setminus\{\emptyset\}$. If $A,B\subset\FF$, we define
\[ A\hp B:=\bigcup_{\substack{a\in A\\b\in B}}a\hp b.\]
We require $\hp$ to be associative: $(a\hp b)\hp c=a\hp (b\hp c)$ for every $a,b,c\in\FF$.

A \emph{commutative hypergroup} is a tuple $(G,\hp,0)$, where $\hp$ is a commutative and associative hyperoperation on $G$ such that:
\begin{itemize}
\item $0\hp x=\{x\}$ for every $x\in G$,
\item for every $x\in G$, there is a unique element $-x\in G$ such that $0\in x\hp-x$,
\item $x\in y\hp z$ if and only if $z\in x\hp-y$.
\end{itemize}

A \emph{hyperfield} is a tuple $(\FF,\wcdot,\hp,1,0)$ with $0\neq1$ such that:
\begin{itemize}
\item $(\FF\setminus\{0\},\wcdot,1)$ is a commutative group,
\item $(\FF,\hp,0)$ is a commutative hypergroup,
\item $0\cdot x=x \cdot0=0$ for every $x\in\FF$,
\item $a\cdot(x\hp y)=(a\cdot x)\hp(a\cdot y)$ for all $a, x, y\in\FF$.
\end{itemize}

\medskip

Let $n$ be a positive integer and $E:=\zint1n$ be a finite set. A Grassman-Plücker function over $\FF$ of rank $r$ is a function $\varphi:E^r\to\FF$ such that
\begin{itemize}
\item $\varphi$ is not identically zero,
\item (alternating) for any $I=(i_1,\dots,i_r)\in E^r$ and for any two different indices $\ell, m$,
\[ \varphi(i_1,\dots,i_\ell,\dots,i_m,\dots,i_r)=-\varphi(i_1,\dots,i_m,\dots,i_\ell,\dots,i_r), \]
and
$\varphi(i_1,\dots,i_\ell,\dots,i_m,\dots,i_r)=0$
if $i_\ell=i_m$,
\item (Grassman-Plücker relations) for any $I=(i_1,\dots,i_r), J=(j_1,\dots,j_r)\in E^r$ and any $\ell\in\zint1r$,
\[ \bighp_{m=1}^r \varphi(i_1,\dots,i_{m-1},j_\ell,i_{m+1},\dots,i_r)\varphi(j_1,\dots,j_{\ell-1},i_m,j_{\ell+1},\dots,j_r) \hp -\varphi(I)\varphi(J)\ni 0. \]
\end{itemize}

Recall that, if $I=\{i_1,\dots,i_l\}\subset E$ with $i_1<\dots<i_l$, then we also denote by $I$ the $l$-tuple $(i_1,\dots,i_l)\in E^l$. We this convention, a Grassman-Plücker function $\varphi$ of rank $r$ induces a function on the subsets of $E$ of size $r$. Moreover, since $\varphi$ is alternating, $\varphi$ is characterized by its value on the subsets of $E$ of size $r$.

Let $\varphi$ be a Grassman-Plücker function over $\FF$ of rank $r$ . Then one can verify that the set
\[ \{I\subset E \st \card{I}=r\text{ and }\varphi(I)\neq 0\} \]
is the set of bases of some matroid over the ground set $E$. We call this matroid the \emph{underlying matroid of $\varphi$} and denote it by $\Ma_\varphi$.

For example, if $u=(u_1,\dots,u_n)$ is a family of elements of $\Z^p$ of rank $r$, we define
\[ \begin{array}{rrcl}
\varphi_u: & E^r & \to & \R, \\
 & I & \mapsto & \det(V_I),
\end{array}\]
where $V$ is defined as in Equation~\eqref{5:eqn:norm_to_det}. Then $\varphi_u$ is a Grassman-Plücker function and its underlying matroid is $\Ma_u$, as defined in the previous subsection.

Here is the main definition of this section.
\begin{definition}
Let $\FF$ be a hyperfield and let $\varphi$ be a Grassman-Plücker function over $\FF$. Let $\Ma$ be its underlying matroid. Then we define
\begin{align*}
\Kir_k(\varphi; x) &:= \sum_{I\in\Bas(\Ma)}\varphi(I)^kx^I, \\
\Sym_k(\varphi; x) &:= \sum_{I\in\Bas(\Ma)}\varphi(I)^kx^{\compl I}. \qedhere
\end{align*}
\end{definition}

Note that this definition is compatible with Section~\ref{5:sec:duality}, \ie,
\[ \Kir_k(\varphi_u; x)=\Kir_k(u;x)\text{\quad and\quad}\Sym_k(\varphi_u; x)=\Sym_k(u;x). \]

\begin{remark}
In order to take into account complex matroids (\ie, matroids over the phase hyperfield, \cf \cites{AD12, BB16}), one may want to consider an involution $\varphi\mapsto\overline\varphi$. In this case, $\varphi$ must be replaced by $\overline\varphi$ in the definition of the Symanzik polynomials.
\end{remark}

\begin{remark}
One has to be careful with polynomials with coefficients in a hyperfield. Several definitions are possible (\cf \cite{BL18}*{Appendix A}). However, in this article, every polynomials over hyperfields will be of the form
\[ P(x)=\sum_{I\subset E} p_Ix^I \]
with coefficients $p_I$ in $\FF$ (and not a subset of $\FF$ nor a sum of elements in $\FF$). In particular, if $P(x)=\sum_I p_Ix^I$ and $Q(x)=\sum_I q_Ix^I$, we will only add $P$ and $Q$ if, for any $I$, $p_I=0$ or $q_I=0$. In the same way, we will only multiply $P$ and $Q$ if for any $I$ and $J$ such that $p_I\neq0$ and $q_J\neq0$, their intersection $I\cap J$ is empty.
\end{remark}

\medskip

We now explain how to extend deletion, contraction and duality. Let $\varphi$ be a Grassman-Plücker function of rank $r$ over $\FF$ and let $\Ma$ be its underlying matroid. For any subset $I\subset E$, we choose a basis $B_{\varphi\contr I}\in\Bas(\Ma\mrest I)$, and a basis $B_{\varphi\del I}\in\Bas(\Ma\contr \compl{I})$.

Following~\cite{BB16}, we define $\varphi\del I:(E\setminus I)^{\rk(\Ma\del I)}\to\FF$ by
\[ (\varphi\del I)(J)=\varphi(J\concat B_{\varphi\del I}). \]
We also define the restriction: $\varphi\mrest I:=\varphi\del\compl I$.
Similarly, we define $\varphi\contr I:(E\setminus I)^{\rk(\Ma\contr I)}\to\FF$ by
\[ (\varphi\contr I)(J)=\varphi(J\concat B_{\varphi\contr I}). \]
If $I'\subset E$ is disjoint from $I$, we set
\[ (\varphi\contr I'\del I)(J)=\varphi(J\concat B_{\varphi\contr I'}\concat B_{\varphi\del I}). \]
We define $\hat\varphi:E^{n-\rk(\Ma)}\to\FF$ to be the unique Grassman-Plücker function such that, for any subset $J$ of $E$ of size $n-\rk(\Ma)$,
\[ \hat\varphi(J)=\sgn(J\concat\compl J)\varphi(\compl J), \]
where $\sgn(J\concat\compl J)$ denotes the signature of the permutation $J\concat\compl J$. Moreover, we choose
\[ B_{\hat\varphi\del I}:=I\setminus B_{\varphi\contr I}\quad\text{and}\quad B_{\hat\varphi\contr I}:=I\setminus B_{\varphi\del I}. \]

As stated in~\cite{BB16}, the four new functions are Grassman-Plücker functions, and the underlying matroids are respectively $\Ma\del I$, $\Ma\contr I$, $\Ma\contr I'\del I$ and $\hat\Ma$. Moreover, the functions are independent of the chosen bases up to a nonzero factor (up to a sign for $\hat\varphi$). We also keep the usual rules $\hat{\hat{\varphi\phantom{\,}}}=(-1)^{r(n-r)}\varphi$ and $\hat{\varphi\contr I}=\pm\hat\varphi\del I$. Article~\cite{BB16} refers to another article,~\cite{AD12}, for the proofs. The latter proves the results in the special case of complex matroids. For the sake of completeness, we provide a proof in full generality.

\begin{proof}
Many points are easy to check. We will only check the Grassman-Plücker relations, the independency up to a factor, and the formula $\hat{\varphi\contr I}=\pm\hat\varphi\del I$.

First, we prove the Grassman-Plücker relations for the contraction or the deletion of one element. For some cases, for example the case of the contraction of a loop, there is nothing to prove. Otherwise, for any $e\in E$, $I=(i_1,\dots,i_{r-1})\in E^{r-1}, J=(j_1,\dots,j_{r-1})\in E^{r-1}$ and $\ell\in\zint1{r-1}$,
\begin{multline*}
\bighp_{m=1}^{r-1} \varphi(i_1,\dots,i_{m-1},j_\ell,i_{m+1},\dots,i_{r-1},e)\varphi(j_1,\dots,j_{\ell-1},i_m,j_{\ell+1},\dots,j_{r-1},e) \\
\hp \varphi(i_1,\dots,i_{r-1},j_\ell)\varphi(j_1,\dots,j_{\ell-1},e,j_{\ell+1},\dots,j_{r-1},e) \hp - \varphi(I\concat(e))\varphi(J\concat(e))\ni0.
\end{multline*}
Since
\[ \varphi(j_1,\dots,j_{\ell-1},e,j_{\ell+1},\dots,j_{r-1},e)=0, \]
we get the Grassman-Plücker relation for $\varphi\del\{e\}$ and for $\varphi\contr\{e\}$. The Grassman-Plücker relations for a general $\varphi\del I$ follow easily by induction using that
\begin{align*}
\varphi\del I &= \varphi\del B_{\varphi\del I} \\
  &= (\dots((\varphi\del \{b_s\})\del \{b_{s-1}\})\dots)\del \{b_1\},
\end{align*}
where $B_{\varphi\del I}=(b_1, \dots, b_s)$. A similar argument works for $\varphi\contr I$ and $\varphi\contr I'\del I$.

\medskip

For the independency, let us first do a simple case. Let $e, e'\in E$ be two parallel elements and let $J, J'\subset E$ of size $r-1$. Applying the Grassman-Plücker function to $J\concat (e)$ and $J'\concat (e')$ for $\ell=r$, most of the terms will be zero because
\[ \varphi(\dots,e',\dots,e)=0. \]
Then, we will get
\begin{equation} \label{5:eqn:up_to_a_factor_one_element}
\varphi(J\concat (e'))\varphi(J'\concat (e))=\varphi(J\concat (e))\varphi(J'\concat (e')).
\end{equation}

Let $I\subset E$. Let $J,J'\subset I$ of size $\rk(I)$ and $K,K'\subset\compl I$ of size $r-\rk(I)$. Let us prove that
\begin{equation} \label{5:eqn:up_to_a_factor}
\varphi(J\concat K)\varphi(J'\concat K')=\varphi(J'\concat K)\varphi(J\concat K').
\end{equation}
First, by~\eqref{5:eqn:equivalence_bases}, $\varphi(J\concat K)$ is nonzero if and only if $J\in\Bas(\Ma\mrest I)$ and $K\in\Bas(\Ma\contr I)$. Thus, either $J,J'\in\Bas(\Ma\mrest I)$ and $K,K'\in\Bas(\Ma\contr I)$, and then both members of~\eqref{5:eqn:up_to_a_factor} are nonzero, or both members are zero, and then~\eqref{5:eqn:up_to_a_factor} is verified. In the first case, assume that $J\neq J'$. By the basis exchange property, one can find $e\in J$ and $e'\in J'$ such that $e\neq e'$ and $J_1:=J-e+e'$ is a basis of $\Ma\mrest I$. In $\Ma\mrest I\contr (J-e)=\Ma\contr (J-e)\mrest I$, elements $e$ and $e'$ are parallel. Thus, they are parallel in $\Ma\contr (J-e)$. Hence, using Formula~\eqref{5:eqn:up_to_a_factor_one_element} for $\varphi\contr(J-e)$, we get
\[
(\varphi\contr(J\!-\!e))(K\concat (e))\;(\varphi\contr(J\!-\!e))(K'\concat (e'))=(\varphi\contr(J\!-\!e))(K'\concat (e))\;(\varphi\contr(J\!-\!e))(K\concat (e')). \]
Being careful with signs, we deduce
\begin{gather*}
\varphi(K\concat J)\varphi(K'\concat J_1)=\varphi(K'\concat J)\varphi(K\concat J_1),\text{ thus} \\
\frac{\varphi(K\concat J)}{\varphi(K'\concat J)}=\frac{\varphi(K\concat J_1)}{\varphi(K'\concat J_1)}.
\end{gather*}
Using $\card{J'\setminus J}$ times the basis exchange property, we get
\[ \frac{\varphi(K\concat J)}{\varphi(K'\concat J)}=\frac{\varphi(K\concat J')}{\varphi(K'\concat J')}, \]
which completes the proof of~\eqref{5:eqn:up_to_a_factor}. As a direct consequence, we get the independency up to a factor of $\varphi\contr I$ and $\varphi\mrest I$.

\medskip

It remains to prove that, if $J\subset\compl I$ is of size $\rk(\hat\varphi\del I)$, then
\[ (\hat{\varphi\contr I})(J)=\pm(\hat\varphi\del I)(J). \]
Then, recalling that we chose $B_{\hat\varphi\del I}=I\setminus B_{\varphi\contr I}$, we get
\begin{align*}
(\hat\varphi\del I)(J)
  &=\hat\varphi(J\concat B_{\hat\varphi\del I}) \\
  &=\pm\varphi((\compl I\setminus J)\concat B_{\varphi\contr I}) \\
  &=\pm(\varphi\contr I)(\compl I\setminus J) \\
  &=\pm(\hat{\varphi\contr I})(J). \qedhere
\end{align*}

\end{proof}

\medskip

We now explain how to generalize Symanzik polynomials with parameters. Let $\varphi'$ be a Grassman-Plücker function on some set $E\cup W$ with $W$ a finite set disjoint from $E$. Assume that $W$ is independent in $\Ma_{\varphi'}$. Suppose that $\varphi=\varphi'\del W$ and notice that there is only one way to define $\varphi'\del W$. We call $\varphi'\!\contr\,W$ the \emph{quotient of $\varphi$ by W}. It is uniquely defined. Then we define
\[ \Sym_k(\varphi; W; x):=\Sym_k(\varphi'\!\contr\,W; x). \]
Once more, this extends Definition~\ref{5:defi:Symanzik_polynomials_with_parameters}.

\subsection{Deletion-contraction formula}

Thanks to the definitions of last subsections, we can state the generalization of the deletion-contraction formula (\cf Equation~\eqref{5:eqn:deletion_contraction_graph} in the introduction). From those definitions, the following result is immediate.

\begin{proposition}[Deletion-contraction formula] \label{5:prop:deletion_contraction}
Let $e\in E$. We set
\[ N_1(e):=\begin{cases}
0 & \text{if $e$ is a loop in $\Ma$,} \\
x_e & \text{otherwise,}
\end{cases}
\quad\text{and}\quad
N_2(e):=\begin{cases}
0 & \text{if $e$ is free in $\Ma$,} \\
1 & \text{otherwise.}
\end{cases} \]
Then,
\[ \Kir_k(\varphi; x)=N_1(e)\Kir_k(\varphi\contr \{e\}; x_{E-e})+N_2(e)\Kir_k(\varphi\del \{e\}; x_{E-e}). \]
Moreover, setting
\[ \hat N_1(e):=\begin{cases}
0 & \text{if $e$ is a loop in $\Ma$,} \\
1 & \text{otherwise,}
\end{cases}
\quad\text{and}\quad
\hat N_2(e):=\begin{cases}
0 & \text{if $e$ is free in $\Ma$,} \\
x_e & \text{otherwise,}
\end{cases} \]
we have
\[ \Sym_k(\varphi; x)=\hat N_1(e)\Sym_k(\varphi\contr \{e\}; x_{E-e})+\hat N_2(e)\Sym_k(\varphi\del \{e\}; x_{E-e}). \]
\end{proposition}

The notations of the previous proposition follows those of~\cite{DFM18}. This article describes Universal Tutte characters which unify objects behaving like Tutte polynomials. By the following proposition, we show that the polynomials $\Kir_k$ and $\Sym_k$ belong to this big family. In particular, we have the following formul\ae.

\begin{proposition}
\label{5:prop:Tutte_character}
Let $I\subset E$. We extend $N_1$ and $N_2$ by
\[ N_1(\varphi\mrest I)=\begin{cases}
0 & \text{if $\Ma\mrest I$ is not the free matroid,} \\
\prod_{i\in I}x_i & \text{otherwise,}
\end{cases}
\qquad
N_2(\varphi\contr I)=\begin{cases}
0 & \text{if $\rk(I)>0$,} \\
1 & \text{otherwise.}
\end{cases} \]
Then,
\[ \Kir_k(\varphi; x)=\sum_{I\subset E}N_1(\varphi\mrest I)\tau(\varphi\mrest I\contr I)N_2(\varphi\contr I), \]
where, denoting by $\varnothing$ the $0$-tuple,
\[ \tau(\varphi\mrest I\contr I):=(\varphi\mrest I\contr I)(\varnothing)^k. \]
Moreover,
\[ \hat N_1(\varphi\mrest I)=\begin{cases}
0 & \text{if $\Ma\mrest I$ is not the free matroid,} \\
1 & \text{otherwise,}
\end{cases}
\qquad
\hat N_2(\varphi\contr I)=\begin{cases}
0 & \text{if $\rk(I)>0$,} \\
\prod_{i\in \compl I}x_i & \text{otherwise,}
\end{cases} \]
we have
\[ \Sym_k(\varphi; x)=\sum_{I\subset E}\hat N_1(\varphi\mrest I)\tau(\varphi\mrest I\contr I)\hat N_2(\varphi\contr I). \]
\end{proposition}

One might have noticed that $\hat N_2$ in Proposition~\ref{5:prop:Tutte_character} is not an extension of $\hat N_2$ in Proposition~\ref{5:prop:deletion_contraction}. Actually, $\hat N_2(e)$ must be understood as $\hat N_2(\varphi\contr\compl{\{e\}})$ (\cf \cite{DFM18}).

To understand why this last proposition holds, notice that a term of the right-hand side sum is zero if $I$ is not a basis of $\Ma$. If $I$ is a basis, then, for the Kirchhoff polynomial for instance, we have
\[ N_1(\varphi\mrest I)\tau(\varphi\mrest I\contr I)N_2(\varphi\contr I)=\prod_{i\in I}x_i\cdot(\varphi\mrest I\contr I)(\varnothing)^k. \]
Since $I$ is a basis, $B_{\varphi\mrest I}=\varnothing$ and $B_{\varphi\contr I}=I$. Thus,
\[ \varphi\mrest I\contr I(\varnothing)=\varphi(I). \]

\begin{remark}
Following this idea, one might define the Tutte polynomial of $\varphi$ by
\[ \sum_{I\subset E}(x-1)^{\rk(E)-\rk(I)}(y-1)^{\card I-\rk(I)}\varphi\mrest I\contr I(\varnothing)^k. \]
Unfortunately, this formula strongly depends on the bases of $B_{\varphi\mrest I}$ and $B_{\varphi\contr I}$ (unless for terms where $I$ is a basis).
\end{remark}

\subsection{Partial factorization}

Let us conclude this section by a generalization of the useful formula of partial factorization (\cf Equation~\eqref{5:eqn:partial_factorization} in the introduction).

Let
\[ P(x)=\sum_J p_Jx^J \]
be a multivariate polynomial. Let $I$ be a subset of $\zint1n$. We define the set of the degrees of the monomials of $P$ with respect to the variables $x_I$ as
\begin{gather*}
\degs_{x_I}(P(x)):=\{\card{J\cap I}\st p_J\neq0\}.
\end{gather*}

Here is the generalization of the formula.

\begin{proposition} \label{5:prop:partial_factorization}
Let $I\subset E$. Then
\[ \Kir_k(\varphi; x)=\frac1{\varphi\mrest I\contr I(\varnothing)^k}\Kir_k(\varphi\mrest I; x_I)\Kir_k(\varphi\contr I; x_{\compl I})+R_{\varphi,I}(x) \]
for some polynomial $R_{\varphi,I}(x)$ verifying
\begin{gather*}
\max\degs_{x_I}(R_{\varphi,I}(x))<\deg(\Kir_k(\varphi\mrest I; x_I))=\rk(I), \\
\min\degs_{x_{\compl I}}(R_{\varphi,I}(x))>\deg(\Kir_k(\varphi\contr I; x_{\compl I}))=\rk(\Ma\contr I).
\end{gather*}

Similarly,
\[ \Sym_k(\varphi; x)=\frac1{\varphi\mrest I\contr I(\varnothing)^k}\Sym_k(\varphi\mrest I; x_I)\Sym_k(\varphi\contr I; x_{\compl I})+\hat R_{\varphi,I}(x), \]
for some polynomial $\hat R_{\varphi,I}(x)$ verifying
\begin{gather*}
\min\degs_{x_I}(\hat R_{\varphi,I}(x))>\deg(\Sym_k(\varphi\mrest I; x_I)), \\
\max\degs_{x_{\compl I}}(\hat R_{\varphi,I}(x))<\deg(\Sym_k(\varphi\contr I; x_{\compl I})).
\end{gather*}
\end{proposition}

\begin{proof}
We have
\[ \Kir_k(\varphi; x)=\sum_{B\in\Bas(\Ma)}\varphi(B)^kx_I^{B\cap I}x_{\compl I}^{B\cap\compl I}. \]
If $B\in\Bas(\Ma)$, then, by~\eqref{5:eqn:contraction_deletion_bases}, $\card{B\cap\compl I}$ is minimal if and only if $\card{B\cap I}$ is maximal if and only if $B\cap I$ is a basis of $\Ma\mrest I$. By~\eqref{5:eqn:two_out_of_three_bases}, we can cut the sum in two parts in the following way:
\[ \Kir_k(\varphi; x)=\sum_{B'\in\Bas(\Ma\mrest I)}\sum_{B''\in\Bas(\Ma\contr I)}\varphi(B'\concat B'')^kx_I^{B'}x_{\compl I}^{B''} + R_{\varphi,I}(x), \]
where
\[ R_{\varphi,I}(x):=\sum_{\substack{B\in\Bas(\Ma) \\ \card{B\cap I}<\rk(I)}}\varphi(B)^kx^B. \]
The sum is of degree $\rk(I)$ in $x_I$ and $\rk(\Ma\contr I)$ in $x_{\compl I}$. Moreover, the degree of each term of $R_{\varphi, I}(x)$ is less than $\rk(I)$ in $x_I$ and greater than $\rk(\Ma\contr I)$ in $x_{\compl I}$.

We also have
\[ \Kir_k(\varphi\mrest I; x_I)\Kir_k(\varphi\contr I; x_{\compl I})=\sum_{B'\in\Bas(\Ma\mrest I)}(\varphi\mrest I)(B')^kx_I^{B'}\sum_{B''\in\Bas(\Ma\contr I)}(\varphi\contr I)(B'')^kx_{\compl I}^{B''}. \]
Hence, it remains to prove that if $B'\in\Bas(\Ma\mrest I)$ and if $B''\in\Bas(\Ma\contr I)$, then
\[ (\varphi\mrest I)(B')\,(\varphi\contr I)(B'')=\pm(\varphi\mrest I\contr I)(\varnothing)\,\varphi(B'\concat B''). \]
This is equivalent to
\[ \varphi(B'\concat B_{\varphi\mrest I})\varphi(B''\concat B_{\varphi\contr I})=\pm\varphi(B_{\varphi\contr I}\concat B_{\varphi\mrest I})\varphi(B'\concat B''). \]
But this is just a consequence of~\eqref{5:eqn:up_to_a_factor}. The proof for Symanzik polynomials is similar.
\end{proof}

%%%
\section{Exchange graph for matroids}
\label{5:sec:exchange_graph}

This section could seem out of context: we will not talk about Symanzik polynomials. However, we need Corollary~\ref{5:cor:exchange_graph} below in the next section. Theorem~\ref{5:thm:exchange_graph} and its corollaries are interesting combinatorial results about connected components of what we call the exchange graph of a matroid. These results generalize Theorem 2.13 of~\cite{Ami19} to the matroids, and they go further in the study of the exchange graph.

Notations and reminders about matroids can be found in Section~\ref{5:subsec:matroids}. In this section, we fix a matroid $\Ma=(\E, \Ind)$. We set $r:=\rk(\Ma)$. If $I\subset\E$, $\Fr(I)$ will denote the set of elements independent from $I$, \ie, the complement of $\cl(I)$ in $\E$.

We are interested in finding the different connected components of some interesting subgraphs of the exchange graph of a matroid we define right below.

\begin{definition} \label{5:defi:exchange_graph}
The \emph{exchange graph} $\G=(\V,\gE)$ associated to $\Ma$ is the graph with vertex set $\V:=\Ind\times\Ind$ and edge set $\gE$ such that two vertices $(I_1,I_2)$ and $(I'_1, I'_2)$ are adjacent if there exists $i\in\E$ such that either, $I'_1=I_1+i$ and $I'_2=I_2-i$, or, $I'_1=I_1-i$ and $I'_2=I_2+i$.
\end{definition}

We fix $\G=(\V, \gE)$ the exchange graph of $\Ma$. If $\U$ is a subset of $\V$, then $\G[\U]$ denotes the induced subgraph of $\G$ with vertex set $\U$, \ie, the subgraph of $\G$ of vertex set $\U$ and edge set all the edges connecting pairs of vertices in $\U$.
If $p,q\in\zint0r$, then we set $\V_{p,q}:=\Ind_p\times\Ind_q$.
Moreover, if $p\not=0$ and $q\not=r$, we define the bipartite graph $\G_{p,q}:=\G[\V_{p,q}\sqcup\V_{p-1,q+1}]$.

\begin{remark} \label{5:rem:exchange_graph_isomorphism}
If $p\in\zint{1 }r$ and $q\in\zint0{r-1}$, then we have a natural graph isomorphism:
\[ \begin{array}{rrl}
\Phi_{p,q}: & \G_{p,q} \xrightarrow{\ \sim\ } & \G_{q+1,p-1}, \\
 & (I_1,I_2)\in\V_{p,q}\sqcup\V_{p-1,q+1} \xmapsto{\phantom{\ \sim\ }} & (I_2,I_1)\in\V_{q,p}\sqcup\V_{q+1,p-1}.
\end{array} \qedhere \]
\end{remark}

There are two important invariants in connected components of $\G$. They correspond to Definitions~\ref{5:defi:multicup} and~\ref{5:defi:MCP} below.

\begin{definition} \label{5:defi:multicup}
If $I, J$ are two non necessarily disjoint sets, we write $I\multicup J$ for the multiset containing elements of $I$ and, disjointly, elements of $J$ (such that elements in $I\cap J$ appear in $I\multicup J$ with multiplicity $2$).
\end{definition}

Abusing notation, if $(U,V)$ and $(I,J)$ are two ordered pairs of sets, then we write $(U,V)\subset(I,J)$ if $U\subset I$ and $V\subset J$. That defines a partial order on ordered pairs of sets.

\begin{definition} \label{5:defi:CDP}
If $(I, J)\in\Ind\times\Ind$ is an ordered pair of independent sets and if $(U,V)$ is another one, we say that $(U,V)$ is a \emph{codependent pair of $(I,J)$} if $(U,V)\subset(I,J)$ and $\cl(U)=\cl(V)$.
\end{definition}

\begin{claim}
Let $(I,J)\in\Ind\times\Ind$ be an ordered pair of independent sets. Let $(U,V), (U', V')\in\Ind\times\Ind$ be two codependent pairs of $(I,J)$. Then $(U\cup U',V\cup V')$ is a codependent pair of $(I,J)$.
\end{claim}

\begin{proof}
Clearly $(U\cup U',V\cup V')\subset(I,J)$. And,
\[ \cl(U\cup U')=\cl(\cl(U)\cup\cl(U'))=\cl(\cl(V)\cup\cl(V'))=\cl(V\cup V'). \qedhere \]
\end{proof}

Thanks to the previous claim, and noticing that $(\emptyset, \emptyset)$ is always a codependent pair, one can state the following definition.

\begin{definition} \label{5:defi:MCP}
If $(I, J)\in\Ind\times\Ind$ is an ordered pair of independent sets, we call \emph{maximal codependent pair (or MCP)} of $I$ and $J$, denoted by $\MCP(I,J)$, the unique maximal codependent pair of $(I,J)$ for the inclusion.
\end{definition}

\begin{proposition} \label{5:prop:invariant}
If $(I,J)\in\V$, then $I\multicup J$ and $\MCP(I,J)$ are invariant in the connected component of $(I,J)$ in $\G$, \ie, if $(I',J')$ is any element of the connected component of $(I,J)$, then
\[ I\multicup J=I'\multicup J'\myand \MCP(I,J)=\MCP(I',J'). \]
\end{proposition}

\begin{proof}
Let $(I,J)$ and $(J',I')$ be two neighbors in $\V$ such that there exists $i\in\E$ such that $I'=J+i$, and $I=J'+i$. One has
\[ I\multicup J=(J'+i)\multicup J=J'\multicup(J+i)=J'\multicup I'. \]
Moreover, if $(U,V):=\MCP(I,J)$, since $i\in\Fr(J)$, $i\not\in\cl(V)$, and so $i\not\in U$. Thus, $\MCP(I,J)=\MCP(I-i,J)=\MCP(J',J)$. Using the same argument on $(J',I')$, one obtains the equality $\MCP(I,J)=\MCP(J',I')$. The end of the proof is now straightforward.
\end{proof}

Let us fix $p\in\zint{ 1}r$ and $q\in\zint0{r-1}$. The nice result is that the two invariants of Proposition~\ref{5:prop:invariant} form a complete set of invariants (see Corollary~\ref{5:cor:exchange_graph_2}). This is also true for $\G_{p,q}$ if we ignore isolated vertices. That is why we first study which vertices are isolated.

\begin{proposition} \label{5:prop:isolated_vertices}
Let $(I,J)$ be a vertex of $\G$. Then:
\begin{enumerate}
\item $(I,J)$ is an isolated vertex of $\G$ if and only if $\cl(I)=\cl(J)$, \ie, $\MCP(I,J)=(I,J)$;
  \label{5:prop:isolated_vertices:isolated_G}
\item if $(I,J)\in\G_{p,q}$, then $(I,J)$ is an isolated vertex of $\G_{p,q}$ if and only if one of the two following cases happens.
\begin{itemize}
  \item either $(I,J)\in\V_{p,q}$ and $I\subset\cl(J)$,
  \item or $(I,J)\in\V_{p-1,q+1}$ and $J\subset\cl(I)$;
\end{itemize} \label{5:prop:isolated_vertices:isolated_Gpq}
\item if $p=q+1$, $\G_{p,q}$ has no isolated vertex.
  \label{5:prop:isolated_vertices:not_isolated}
\end{enumerate}
\end{proposition}

\begin{proof}
Let $(I,J)\in\V_{p,q}$. The neighborhood of $(I,J)$ in $\G_{p,q}$ consists of all vertices of the form $(I-i,J+i)$ for $i\in I\cap\Fr(J)$. Thus, Point~\eqref{5:prop:isolated_vertices:isolated_Gpq} is clear. Then, $(I,J)$ is isolated in $\G$ if and only if it is isolated in $\G_{p,q}$ and in $\G_{p+1,q-1}$ (if $p\neq r$ and $q\neq0$). By Point~\eqref{5:prop:isolated_vertices:isolated_Gpq}, this is equivalent to $I\subset\cl(J)$ and $J\subset\cl(I)$, \ie, to $\cl(I)=\cl(J)$ (notice that the second inclusion is trivial if $p=r$ or $q=0$). Finally, if $p>q$, $\rk(I)>\rk(J)$, thus $I$ cannot be a subset of $\cl(J)$. Thus, if $p=q+1$, we can get Point~\eqref{5:prop:isolated_vertices:not_isolated} from Point~\eqref{5:prop:isolated_vertices:isolated_Gpq}.
\end{proof}

Let us introduce another definition below. Actually, it is equivalent to Definition~\ref{5:defi:MCP}.

\begin{definition} \label{5:defi:CFE}
Let $(I,J)\in\V_{p,q}$ and let $\W$ be the connected component of $(I,J)$ in $\G_{p,q}$. Then we define the \emph{ordered pair of fixed elements of $(I,J)$} by
\[ \PFE(I,J):=\bigg(\bigcap_{(U,V)\in\W} U,\ \bigcap_{(U,V)\in\W} V\bigg). \qedhere \]
\end{definition}

\begin{proposition} \label{5:prop:MCP_CFE}
If $(I,J)$ is a non-isolated vertex of $\G_{p,q}$, then
\[ \MCP(I,J)=\PFE(I,J). \]
\end{proposition}

We need the following lemma which is a direct consequence of Proposition 1.1.6 of~\cite{Oxl11}.

\begin{lemma} \label{5:lem:matroid_cycles}
If $U\in\Ind$ and $i\in\cl(U)$, then $\{C\subset U \st i\in\cl(C)\}$ admits a least element for the inclusion.
\end{lemma}

\begin{proof}[Proof of the proposition]
Let $(P_1,P_2):=\PFE(I,J)$. Denote by $\W$ the set of vertices in the connected component of $(I,J)$ in $\G_{p,q}$ and by $\W_{p,q}$, \resp by $\W_{p-1,q+1}$, the intersection $\W\cap\V_{p,q}$, \resp $\W\cap\V_{p-1,q+1}$. By Proposition~\ref{5:prop:invariant}, it is clear that $\MCP(I,J)\subset\PFE(I,J)$.

In order to prove the other inclusion, we introduce the two following sets
\[ Q_1:=\bigcap_{(J_1,I_1)\in \W_{p-1,q+1}}\cl(J_1),\qquad Q_2:=\bigcap_{(I_2,J_2)\in \W_{p,q}}\cl(J_2). \]
Clearly,
\begin{equation} \label{5:eqn:exchange_graph:P_subset_Q}
P_1\subset Q_1\myand P_2\subset Q_2.
\end{equation}

The interesting property of these sets is that
\begin{equation} \label{5:eqn:exchange_graph:same_components}
Q_1=Q_2.
\end{equation}
In order to prove this, it is enough to show that any element of $\E$ which is not in $Q_2$, is not in $Q_1$ either. Let $i$ be an element of $\E$ which is not in $Q_2$. There exists $(I_2,J_2)\in\W_{p,q}$ such that $i\not\in\cl(J_2)$. If $i\not\in\cl(I_2)$, then $i\not\in\cl(J_1)$ for any neighbor $(J_1,I_1)$ of $(I_2,J_2)$, and so $i\not\in Q_1$. Otherwise, assume $i\in\cl(I_2)$. By Lemma~\ref{5:lem:matroid_cycles}, there exists a least element for the inclusion $C\subset I_2$ such that $i\in\cl(C)$. Since $i\not\in\cl(J_2)$, we have $C\not\subset\cl(J_2)$. Let $j\in C\cap\Fr(J_2)$. One obtains that $(I_2-j,J_2+j)\in\W_{p-1,q+1}$, and that $i\not\in\cl(I_2-j)$ since $C\not\subset I_2-j$. Thus, $i\not\in Q_1$.

The third result is
\begin{equation} \label{5:eqn:exchange_graph:Q_subset_cl_P}
Q_2\subset\cl(P_2).
\end{equation}
To see this, let $i\in Q_2$. Let $(J_1, I_1)$ be a neighbor of $(I,J)$. One has $i\in\cl(J)$, and so $i\in\cl(I_1)$ (since $J\subset I_1$). Thus, by Lemma~\ref{5:lem:matroid_cycles}, one can choose $C$, \resp $C_1$, minimal for the inclusion in $J$, \resp $I_1$, such that $i$ is in the closure of $C$, \resp $C_1$. By minimality, $C=C_1$. Hence, by connectivity, for all $(U,V)\in\W$, $C\subset V$. Thus, we have $C\subset P_2$, and so $i\in\cl(P_2)$.

We now have all the needed intermediate results. Equations~\eqref{5:eqn:exchange_graph:P_subset_Q},~\eqref{5:eqn:exchange_graph:same_components} and~\eqref{5:eqn:exchange_graph:Q_subset_cl_P} imply that $P_1\subset Q_1=Q_2\subset\cl(P_2)$, and so $\cl(P_1)\subset\cl(P_2)$. Using a symmetric argument, we obtain that $\cl(P_2)\subset\cl(P_1)$ and so $\cl(P_1)=\cl(P_2)$. Thus, $(P_1,P_2)$ is a dependent ordered pair of $(I,J)$. In particular, $(P_1,P_2)=\PFE(I,J)\subset\MCP(I,J)$.
\end{proof}

In what follow, we will use the following lemma, which is easy to prove.

\begin{lemma} \label{5:lem:same_rank_independents}
Let $U, U'\in\Ind$ be such that $\card{U}=\card{U'}$ and $\Fr(U)\not=\Fr(U')$. Then, $U\cap\Fr(U')$ and $U'\cap\Fr(U)$ are nonempty.
\end{lemma}

We can now state the main theorem of this section.

\begin{theorem} \label{5:thm:exchange_graph}
Let $(I,J)$ and $(I',J')$ be two non-isolated vertices of $\G_{p,q}$. Then $(I,J)$ and $(I',J')$ are in the same connected component of $\G_{p,q}$ if and only if
\[ I\multicup J = I' \multicup J'\qquad\text{and}\qquad\MCP(I, J) = \MCP(I', J'). \]
\end{theorem}

\begin{proof}
The forward direction is given by Proposition~\ref{5:prop:invariant}.

Let $(I_0,J_0)$ and $(I_0',J_0')$ be two non-isolated vertices of $\G_{p,q}$ such that $I_0\multicup J_0=I_0'\multicup J_0'$ and $\MCP(I_0, J_0)=\MCP(I'_0, J'_0)$. Denote by $\W$, \resp $\W'$, the set of vertices in the connected component of $(I_0,J_0)$, \resp of $(I_0',J_0')$, in $\G_{p,q}$. Define $\W_{p,q}:=\W\cap\V_{p,q}$, and define similarly $\W_{p-1,q+1}, \W'_{p,q}, \W'_{p-1,q+1}$. Since both vertices are assumed to be non-isolated, the four previous sets are nonempty. If $(I,J)$ and $(I',J')$ are two elements of $\V_{p,q}$, we set
\[ d((I,J), (I',J')):=\card{I\setminus I'}=p-\card{I\cap I'}. \]
Now, set $(I,J)\in\W_{p,q}$ and $(I',J')\in\W'_{p,q}$ such that $d((I,J), (I',J'))$ is minimal. Set $d:=d((I,J), (I',J'))$. We must show that $d=0$.

For the sake of a contradiction, assume that $d\not=0$. Set $(P_1,P_2):=\PFE(I,J)$ and $(P'_1,P'_2):=\PFE(I',J')$. Assume for now the following equation.
\begin{equation} \label{5:eqn:exchange_graph:remain_to_show}
\forall (\~I,\~J)\in\W_{p,q}, J\setminus J'\subset \~J.
\end{equation}
First, if $(\~J,\~I)\in\W_{p-1,q+1}$, one can find a neighbor $(\~I_1,\~J_1)\in\W_{p,q}$ of $(\~J,\~I)$. We have $J\setminus J'\subset\~J_1\subset\~I$. Thus, one can replace $\W_{p,q}$ by $\W$ in~\eqref{5:eqn:exchange_graph:remain_to_show}, \ie,
\[ J\setminus J'\subset P_2. \]
Secondly, since $d\neq0$, $J\neq J'$. Hence
\[ J\setminus J'\neq\emptyset. \]
Thirdly, $P'_2\subset J'$. Thus
\[ (J\setminus J')\cap P'_2=\emptyset. \]
Finally, the three last equations implies that $P_2\neq P_2'$. Therefore $\PFE(I,J)\neq\PFE(I',J')$. By Proposition~\ref{5:prop:MCP_CFE}, $\MCP(I,J)\neq\MCP(I',J')$. Applying Proposition~\ref{5:prop:invariant}, $\MCP(I_0,J_0)\neq\MCP(I_0',J_0')$. This contradicts our assumption.

\bigskip

It remains to prove Equation~\eqref{5:eqn:exchange_graph:remain_to_show}. We note first that
\begin{equation} \label{5:eqn:exchange_graph:same_intersection}
I\cap J=I'\cap J'.
\end{equation}
This can be easily deduced from $I\multicup J=I'\multicup J'$.

Next, we have
\begin{equation} \label{5:eqn:exchange_graph:common_hyperplan}
\Fr(J)=\Fr(J').
\end{equation}
To see this, suppose for the sake of a contradiction that this equation is false. By Lemma~\ref{5:lem:same_rank_independents}, there exists an $i\in J\,\cap\,\Fr(J')$. One can see that $i\in I'$, because $i\in (I\cup J)\setminus J'$. Moreover, let $j\in \Fr(I'-i)\,\cap\, I$, which exists because $\rk(I'-i)<\rk(I)$. Similarly, $j\in J'+i$. One obtains that $(I'-i+j,J'+i-j)\in\W_{p,q}'$ and $(I'-i+j)\,\cap\, I=(I'\,\cap\, I)+j$, because $j\in I$ and $i\not\in I$ (otherwise one would have $i\in I\,\cap\, J$, and so $i\in I'\,\cap\, J'$ by~\eqref{5:eqn:exchange_graph:same_intersection}; but $i\in\Fr(J')$). The last equality contradicts minimality of $d$.

Now set $i\in\Fr(J)\,\cap\, I$, which exists by Proposition~\ref{5:prop:isolated_vertices},~\eqref{5:prop:isolated_vertices:isolated_Gpq}. One has directly by~\eqref{5:eqn:exchange_graph:common_hyperplan} that $i\in\Fr(J')$, and so $i\in I\,\cap\, I'$. Then we prove that
\begin{align} \label{5:eqn:exchange_graph:common_hyperplan_2}
\Fr(I-i)=\Fr(I'-i).
\end{align}
Otherwise, by Lemma~\ref{5:lem:same_rank_independents}, one could set $j\in(I'-i)\,\cap\,\Fr(I-i)$ and $j'\in(I-i)\,\cap\,\Fr(I'-i)$. But $(I-i+j,J+i-j)\in\W_{p,q}$ and $(I'-i+j',J'+i-j')\in\W_{p,q}'$. Moreover $(I-i+j)\,\cap\,(I'-i+j')=(I\,\cap\, I')-i+j+j'$. Once more, that contradicts minimality of $d$.

If $(J_1, I_1)$ is a neighbor of $(I,J)$, and if $(I_2,J_2)$ is a neighbor of $(J_1,I_1)$, then there exists an $i\in\Fr(J)$ such that $I_1=J+i$, and there exists a $j\in\Fr(J_1)=\Fr(I-i)$ such that $I_2=I-i+j$. By~\eqref{5:eqn:exchange_graph:common_hyperplan}, $i\in\Fr(J')$, and, by~\eqref{5:eqn:exchange_graph:common_hyperplan_2}, $j\in\Fr(I'-i)$. Setting $(I_2',J_2'):=(I'-i+j,J'+i-j)$, one obtains that $(I_2',J_2')\in \W_{p,q}'$. First observe that
\begin{gather}
\label{5:eqn:exchange_graph:propagation_1} J_2\setminus J_2'=(J+i-j)\setminus(J'+i-j)=J\setminus J', \\
\label{5:eqn:exchange_graph:propagation_2} I_2\setminus I_2'=(I-i+j)\setminus(I'-i+j)=I\setminus I'.
\end{gather}
Hence, $d((I_2,J_2),(I'_2,J'_2))=d$ is minimal. We infer that this two vertices also verifies Equations~\eqref{5:eqn:exchange_graph:same_intersection},~\eqref{5:eqn:exchange_graph:common_hyperplan} and~\eqref{5:eqn:exchange_graph:common_hyperplan_2}.

Now let $(\~I,\~J)\in\W_{p,q}$. Since $(I_2,J_2)$ has been arbitrarily chosen among neighbors of neighbors of $(I,J)$, one can in the same way choose a path in $\W$:
\[ ((I,J),(J_1,I_1),(I_2,J_2),(J_3,I_3),\dots,(I_n,J_n)=(\~I,\~J)). \]
To this path corresponds a path in $\W'$:
\[ ((I',J'),(J'_1,I'_1),(I'_2,J'_2),(J'_3,I'_3),\dots,(I'_n,J'_n)). \]
Set $(\~I',\~J'):=(I'_n,J'_n)$. Equations~\eqref{5:eqn:exchange_graph:same_intersection},~\eqref{5:eqn:exchange_graph:common_hyperplan},~\eqref{5:eqn:exchange_graph:common_hyperplan_2},~\eqref{5:eqn:exchange_graph:propagation_1} and~\eqref{5:eqn:exchange_graph:propagation_2} propagate along the paths. In particular, $\~J\setminus\~J'=J\setminus J'$. Hence,
\[ J\setminus J'\subset\~J. \]
We conclude that~\eqref{5:eqn:exchange_graph:remain_to_show} is true, which prove the theorem.
\end{proof}

This first corollary will be useful in Section~\ref{5:sec:variation}.

\begin{corollary} \label{5:cor:exchange_graph}
Let $(I,J), (I',J')$ be two arbitrary vertices of $\G_{r,r-1}$. Then $(I,J)$ and $(I',J')$ are in the same connected component of $\G_{r,r-1}$ if and only if
\[ I\multicup J = I' \multicup J'\qquad\text{and}\qquad\MCP(I, J) = \MCP(I', J'). \]
\end{corollary}

\begin{proof}
It suffices to combine Theorem~\ref{5:thm:exchange_graph} with Point~\eqref{5:prop:isolated_vertices:not_isolated} of Proposition~\ref{5:prop:isolated_vertices}.
\end{proof}

The second corollary extends Theorem~\ref{5:thm:exchange_graph} to the case of the whole exchange graph.

\begin{corollary} \label{5:cor:exchange_graph_2}
Let $(I,J)$, $(I',J')$ be two arbitrary vertices of $\G$. Then $(I,J)$ and $(I',J')$ are in the same connected component of $\G$ if and only if
\[ I\multicup J = I' \multicup J'\qquad\text{and}\qquad\MCP(I, J) = \MCP(I', J'). \]
\end{corollary}

\begin{proof}
The forward direction is given by Proposition~\ref{5:prop:invariant}. For the other direction, let $(I,J)$ and $(I',J')$ be two vertices of $\G$ such that $I\multicup J = I' \multicup J'$ and $\MCP(I, J) = \MCP(I', J')$. Let
\begin{gather*}
p:=\left\lceil\frac{\card I+\card J}2\right\rceil=\left\lceil\frac{\card{I'}+\card{J'}}2\right\rceil, \\
q:=\left\lfloor\frac{\card I+\card J}2\right\rfloor=\left\lfloor\frac{\card{I'}+\card{J'}}2\right\rfloor.
\end{gather*}
Note that $p-q$ is either zero or one. If $\card J>\card I$, then $J\cap\Fr(I)$ is nonempty. Let $j$ be an element of this set. Then $(I+j, J-j)$ is adjacent to $(I,J)$. Iterating this process, it is clear that, if $\card J>\card I$, the connected component of $(I,J)$ contains a vertex of $\G_{p,q}$. Actually, this is still true if $\card J\leq\card I$ putting elements of $I$ in $J$ and stopping at the right time. Let $(\~I,\~J)$, \resp $(\~I',\~J')$, be an element of $\G_{p,q}$ in the connected component of $(I,J)$, \resp of $(I',J')$. We have $\~I\multicup\~J = \~I'\multicup\~J'$ and $\MCP(\~I,\~J)=\MCP(\~I',\~J')$. Thus, $(\~I,\~J)$ and $(\~I',\~J')$ are connected in $\G_{p,q}$, thus in $\G$, provided they are not isolated in $\G_{p,q}$. There are two possibilities.

If $p=q+1$, then they cannot be isolated by Proposition~\ref{5:prop:isolated_vertices}~\eqref{5:prop:isolated_vertices:not_isolated}. Otherwise, $p=q$. In this case, suppose, for example, that $(\~I,\~J)$ is isolated in $\G_{p,q}$. Then, by Proposition~\ref{5:prop:isolated_vertices}~\eqref{5:prop:isolated_vertices:isolated_Gpq}, $\~I\subset\cl(\~J)$. But, since $\card{\~J}=\card{\~I}$, $\cl(\~J)=\cl(\~I)$, thus $\MCP(\~I,\~J)=(\~I,\~J)$, and so $\MCP(\~I',\~J')=(\~I,\~J)$. Looking at cardinalities, this last equality implies that $(\~I,\~J)=(\~I',\~J')$.

In every case, $(\~I,\~J)$ and $(\~I',\~J')$ are in the same connected component, so are $(I,J)$ and $(I',J')$.
\end{proof}

\begin{remark} \label{5:rem:White_s_conjecture}
Let $I\subset\E$ be a subset of size $l\geq r$. We call $I$ a \emph{coindependent of rank $l$} if $\compl I$ is an independent of the dual of $\Ma$. Let $\J$ be the set containing all the independents and all the coindependents. One can extend $\G$ setting $\V:=\J\times\J$ and defining the edges in a similar way to Definition~\ref{5:defi:exchange_graph}. Then Conjecture 1.8. of~\cite{Bla08} is equivalent to the conjectural extension of Corollary~\ref{5:cor:exchange_graph} to $\G_{r,r}$. This conjecture is a part of the White's conjecture about toric ideals of matroids.
\end{remark}

Now we arrive to the last section of this paper which generalizes Theorem 1.1 of~\cite{Ami19}.

%%%
\section{Variation of Symanzik rational fractions}
\label{5:sec:variation}

In this section, if $Y, \fF\in\mhyp^k_n(\R)$ are any two hypercubic matrices of order $k$ and of size $n$, we set, for $k+1$ families $u, u_{(1)}, \dots, u_{(k)}$ of size $p$ in $\R^n$,
\[ (\mat u_{(1)}, \dots, \mat u_{(k)})_{k,y}:=\det(Y\mmult 1U_{(1)}\cdots\mmult kU_{(k)}), \qquad\norm{\mat u}_{k,y}:=\sqrt[k]{(\mat u, \dots, \mat u)_{k,y}}, \]
where, in this section, $(u,\dots,u)_{k,y}$ will always be nonnegative.
Moreover, we use the same definition replacing $Y$ by $Y+\fF$ and $y$ by $y+\f$. In the rest of the section, we will omit $k$ from this notation. In particular, $\norm{\wcdot}_y:=\norm{\wcdot}_{k,y}$.

Let $n, p$ be two positive integers, let $u$ be a family of size $n$ of elements of $\R^p$, and let $r$ be its rank. We set $\Ma:=\Ma_u$ the matroid associated to $u$. Let $v$ be such that $v^\tran$ is a basis of $\ker(U)$. Let $\beta\in\vect{u}$ be any nonzero vector, $\alpha\in\R^n$ be such that $U\alpha=\beta$ and $w:=(v^\tran\concat(\alpha))^\tran$. Moreover, we set $\Delta$ a simplicial complex of dimension $d>0$ with $n$ facets and $p$ $(d-1)$-dimensional faces.

In this last section, we state a nice property of Symanzik rational fractions defined in Definition~\ref{5:defi:Symanzik_rational_fractions}. One can roughly states it as \enquote{a bounding deformation of the metric of a simplicial complex only implies a uniformly bounded variation of the Symanzik rational fraction with one parameter}, where uniformly means that the bound does not depend on the chosen metric.

We can extend Propositions~\ref{5:prop:Symanzik_determinantal_formula} and~\ref{5:prop:Symanzik_with_parameters_determinantal_formula} to get
\[ \ratSym_k(u;(\beta);y)=\frac{\norm{\mat w^\tran}_y^k}{\norm{\mat v^\tran}_y^k}. \]
If $U$ is the $d$-th incidence matrix of the simplicial complex $\Delta$ of dimension $d$, we have seen that it is natural to assign the volume of the $i$-th facet of $\Delta$ to $y_i$, for each $i\in\zint1n$. That is why we will deform the metric of $\Delta$ by slightly perturbing $y$.

Let $k$ be any even positive integer. Let $\U$ be some space and $\fF:\U\to\mhyp^k_n(\R)$ be a bounded map (\ie, each entry is bounded).
Let $y_1, \dots, y_n:\U\to\R_+$ be $n$ functions and let
\[ \begin{array}{rrcl}
Y: & \U & \longrightarrow & \mhyp^k_n(\R), \\
  & t & \longmapsto & \diag^k(y_1(t), \dots, y_n(t)).
\end{array} \]
Suppose that $(\mat v,\dots,\mat v)_{y(t)+\f(t)}$ is positive for all $t\in \U$. This is always true for sufficiently large $y_1, \dots, y_n$.

If $\phi$ and $\psi$ are two functions from $\U$ to $\R$, then the notation $\phi=\O_y(\psi)$ means that there exist two positive constants $c$ and $C$ such that, for all $t\in\U$, $y_1(t), \dots, y_n(t)\geq C$ implies that $|\phi(t)|\leq c|\psi(t)|$. Similarly, the notation $\phi=\o_y(\psi)$ means that, for all positive real $\epsilon$, there exists a positive real $C_\epsilon$ such that, for all $t\in\U$, $y_1(t), \dots, y_n(t)\geq C_\epsilon$ implies that $|\phi(t)|\leq\epsilon|\psi(t)|$.

\begin{theorem} \label{5:thm:variation}
With the above notations, we have
\[ \frac{\norm{\mat w^\tran}_y^k}{\norm{\mat v^\tran}_y^k}-\frac{\norm{\mat w^\tran}_{y+\f}^k}{\norm{\mat v^\tran}_{y+\f}^k}=\O_y(1). \]
\end{theorem}

Much of the rest of this section is devoted to the proof of this theorem. The proof essentially follows that of Theorem 1.1 in~\cite{Ami19}. Here are the main steps of the proof. First we find an equivalent statement where the left-hand member is the difference of two polynomials (Claim~\ref{5:claim:variation_1}, and beyond). Then we partially develop the polynomials in order to treat each term separately. Next, we introduce a graph $\vG$ similar to the exchange graph $\G$ of Section~\ref{5:sec:exchange_graph}. Each term corresponds to a vertex of $\vG$. Claim~\ref{5:claim:variation} states some properties of the terms which can be seen on the graph. For instance, if (the vertices corresponding to) two terms are linked, then the difference between them are negligible (more exactly, one has to multiply one of the term by an explicit constant for this to be true). Thus, if a term is negligible, all terms in its connected components are negligible. For other terms, the idea is to use Section~\ref{5:sec:exchange_graph} to associate bijectively each such term with another term in its connected component that compensates the first one (Claim~\ref{5:claim:variation_2}). Before that, Claim~\ref{5:claim:variation_3} explicits the link between $\vG$ and $\G$.

Let us set the following functions
\[ f_1:=\norm{\mat v^\tran}_y^k, \qquad f_2:=\norm{\mat w^\tran}_y^k, \qquad g_1:=\norm{\mat v^\tran}_{y+\f}^k, \qquad g_2:=\norm{\mat w^\tran}_{y+\f}^k. \]

Let $e$ be the canonical basis of $\R^n$. Decomposing along the standard orthonormal basis, we obtain
\begin{align*}
g_1
 & = \norm{\mat v^\tran}_{y+\f}^k \\
 & = \sum_{\substack{K_1,\dots,K_k\subset\zint1n \\ \card{K_1}=\dots=\card{K_k}=n-r}}(\mat v^\tran, \mat e_{K_1})\cdots(\mat v^\tran, \mat e_{K_k})(\mat e_{K_1}, \dots, \mat e_{K_k})_{y+\f}.
\end{align*}
Remind Remark~\ref{5:rem:non_zero_determinants} about a condition for $\det(V_K)$ and $\det(W_L)$ to be nonzero. The previous formula can be rewritten in the form
\begin{equation}
g_1 = \sum_{I_1,\dots,I_k\in\Ind_r}\det(V_{\compl{I_1}})\cdots\det(V_{\compl{I_k}})(\mat e_{\compl{I_1}}, \dots, \mat e_{\compl{I_k}})_{y+\f}. \label{5:eqn:variation:g_1}
\end{equation}
In the same way,
\begin{gather*}
g_2 = \sum_{J_1,\dots,J_k\in\Ind_{r-1}}\det(W_{\compl{J_1}})\cdots\det(W_{\compl{J_k}})(\mat e_{\compl{J_1}}, \dots, \mat e_{\compl{J_k}})_{y+\f}, \\
f_1=\sum_{I\in\Ind_r}\det(V_{\compl I})^ky^{\compl I}, \\
f_2=\sum_{J\in\Ind_{r-1}}\det(W_{\compl J})^ky^{\compl J}.
\end{gather*}

Notice that $f_1$ is a homogeneous polynomial of $\R[y]$ of degree $n-r$ and that all its coefficients are positive. Moreover all coefficients of $g_1$ and $g_2$ as polynomials of $\R[y]$ are bounded.

If $I$ is a subset of $\zint1n$ and if $h\in\R[y]$, let us denote by $[y^I]h$ the coefficient of the monomial $y^I$ in $h$.
For example, if $I\in\Ind_r$, the monomial $y^{\compl I}$ of $g_1$ is only present in the term where all $I_i$s are equal to $I$. Thus
\[ [y^{\compl I}]g_1 = [y^{\compl I}]\Big(\det(V_{\compl I})^k\norm{\mat e_{\compl I}}_{y+\f}^k\Big) = \det(V_{\compl I})^k. \]
We deduce that these coefficients are constant, and that for all $I\in\Ind$,
\begin{equation} \label{5:eqn:variation:same_dominant_monomials}
[y^{\compl I}]g_1 = [y^{\compl I}]f_1.
\end{equation}

The statements of the theorem is that $f_2/f_1-g_2/g_1=\O_y(1)$. Let us simplify this statement thanks to the following claim.
\begin{claim} \label{5:claim:variation_1}
We have
$g_1-f_1=\o_y(f_1)$.
\end{claim}

\begin{proof}[Proof of the claim]
By~\eqref{5:eqn:variation:same_dominant_monomials}, $g_1-f_1$ is a polynomial of degree at most $n-r-1$. Moreover, its coefficients are bounded functions. Let $K\subset\zint1n$ be an arbitrary subset such that $[y^K](g_1-f_1)$ is nonzero. Then, $\card K<n-r$. Since $f_1$ is homogeneous of degree $n-r$, $[y^K](g_1-f_1)=[y^K]g_1$. From~\eqref{5:eqn:variation:g_1} we infer that $K\subset\compl{I_1}\cap\dots\cap\compl{I_k}$ for some sets $I_1, \dots, I_k\in\Ind_r$. Thus, $K\subsetneq\compl{I_1}$. It happens that $[y^{\compl{I_1}}]f_1=\det(V_{\compl{I_1}})^k$ is a positive integer. Since $[y^K](g_1-f_1)$ is bounded,
\[ \Big([y^K](g_1-f_1)\Big)y^K=\o_y(y^{\compl{I_1}}). \]
Since all coefficients of $f_1$ are positive, we can sum all terms of $(g_1-f_1)$, and then conclude the proof.
\end{proof}

Multiplying $f_2/f_1-g_2/g_1$ by $f_1g_1$, and using the claim, it remains to show that
\[ g_1f_2-f_1g_2=\O_y(f_1^2). \]
Notice that the monomials with nonzero coefficients in $f_1^2$ are exactly the monomials of the form $y^{\compl I}y^{\compl{I'}}$ where $(I,I')\in\Ind_r\times\Ind_r$.

Let us rewrite
\[ g_1(t)f_2(t)=\sum_{I_1,\dots,I_k\in\Ind_r}\sum_{J\in\Ind_{r-1}}s(I_1,\dots,I_k,J)h(I_1,\dots,I_k,J;t), \]
where
\[ h(I_1,\dots,I_k,J;t):=(\mat e_{\compl{I_1}}, \dots, \mat e_{\compl{I_k}})_{y+\f}y^{\compl J} \]
is a polynomial whose coefficients are functions, and
\[ s(I_1,\dots,I_k,J)=\det(V_{\compl{I_1}})\cdots\det(V_{\compl{I_k}})\det(W_{\compl J})^k \]
is a real number.
Similarly,
\[ f_1(t)g_2(t)=\sum_{J_1,\dots,J_k\in\Ind_{r-1}}\sum_{I\in\Ind_r}s(J_1,\dots,J_k,I)h(J_1,\dots,J_k,I;t), \]
where
\begin{gather*}
h(J_1,\dots,J_k,I;t):=(\mat e_{\compl{J_1}}, \dots, \mat e_{\compl{J_k}})_{y+\f}y^{\compl I}, \\
s(J_1,\dots,J_k,I):=\det(W_{\compl{J_1}})\cdots\det(W_{\compl{J_k}})\det(V_{\compl I})^k.
\end{gather*}
We will not develop the polynomials further.

It is clear that
\begin{align}
h(K_1,\dots,K_k, L;t) &= \O_y(y^{\compl K_1\cap\dots\cap\compl K_n}y^{\compl L}) \nonumber \\
 &= \O_y(y^{\compl{(K_1\cup\dots\cup K_n)}}y^{\compl L}). \label{5:eqn:variation:h_O_xi}
\end{align}

\medskip

Let us define a new graph which is slightly similar to the exchange graph $\G_{r,r-1}$ of $\Ma$. Let $\vG=(\vV, \vE)$ be a bipartite graph with vertex set $\vV=\vV_{r-1,r}\sqcup\vV_{r,r-1}$ and edge set $\vE$ where
\begin{gather*}
\vV_{r-1,r}:=(\Ind_{r-1})^k\times\Ind_r, \\
\vV_{r,r-1}:=(\Ind_r)^k\times\Ind_{r-1},
\end{gather*}
and where two vertices $(J_1,\dots,J_k,I)\in\vV_{r-1,r}$ and $(I_1,\dots,I_k,J)\in\vV_{r,r-1}$ are connected by an edge if and only if there exists $i\in\Fr(J_1)\cap\dots\cap\Fr(J_k)\cap I$ such that $I=J+i$ and $I_l=J_l+i$, for all $l\in\zint1k$.

One can now see $h$ and $s$ as some functions on $\vG$. Moreover
\begin{align*}
g_1(t)f_2(t)=\sum_{\va\in\vV_{r,r-1}}s(\va)h(\va;t), \\
g_2(t)f_1(t)=\sum_{\va\in\vV_{r-1,r}}s(\va)h(\va;t).
\end{align*}
Thus, each vertex of $\vG$ corresponds to a part of the difference between both polynomials.

Let $\epsilon$ be define as in Lemma~\ref{5:lem:duality_with_signs}. Let $\va=(I_1, \dots, I_k, J)\in\vV_{r,r-1}$ and $\vb=(J_1, \dots, J_k, I)$ be two adjacent vertices of $\vG$. Let $i\in\zint1n$ be such that $I=J+i$. We define
\[ \eta_{\va, \vb}:=\eta_{\vb,\va}:=\epsilon(\compl{I_1}\concat i)\cdots\epsilon(\compl{I_k}\concat i), \]
where we omitted the brackets around $i$.

\begin{definition}
A vertex $(J_1,\dots,J_k,I)$ in $\vV_{r-1,r}$ is said \emph{ordinary} if $\Fr(J_1)=\dots=\Fr(J_k)$. A vertex $(I_1,\dots,I_k,J)$ in $\vV_{r,r-1}$ is said \emph{ordinary} if $I_1\cap\Fr(J)=\dots=I_k\cap\Fr(J)$. A vertex of $\vV$ which is not ordinary is called \emph{special}.
\end{definition}

\begin{claim} \label{5:claim:variation}
Here are some properties of $h$ and $s$.
\begin{enumerate}
\item If $\va$ and $\vb$ are two adjacent vertices of $\vG$, then $h(\va;t)-\eta_{\va,\vb}h(\vb;t)=\O_y(f_1^2)$. \label{5:claim:variation:adjacent}
\item If $\va$ is a special vertex of $\vV$, then $h(\va;t)=\O_y(f_1^2)$. \label{5:claim:variation:special}
\item Let $\vb:=(J_1,\dots,J_k,I)\in\vV_{r-1,r}$ be an ordinary vertex and let $\va:=(I_1,\dots,I_k,J)\in\vV_{r,r-1}$ be one of its neighbors. If $\det(W_{\compl{J_1}})\neq0$ then, for all $\ell\in\zint1k$,
\[ \frac{\det(V_{\compl{I_\ell}})}{\det(V_{\compl{I_1}})} = \frac{\epsilon(\compl{I_\ell}\concat i)\det(W_{\compl{J_\ell}})}{\epsilon(\compl{I_1}\concat i)\det(W_{\compl{J_1}})}, \]
where $i$ verifies $I=J+i$. \label{5:claim:variation:normal_adjacent}
\item With the notations of the third point, if $s(\va)$ and $s(\vb)$ are nonzero, then
\[ \frac{s(\va)}{\det(V_{\compl I_1})^k\det(W_{\compl J})^k}=\eta_{\va,\vb}\frac{s(\vb)}{\det(W_{\compl J_1})^k\det(V_{\compl I})^k}. \] \label{5:claim:variation:normal_adjacent_2}
\end{enumerate}
\end{claim}

\begin{proof}
\begin{enumerate}
\item Let $\va=(J_1,\dots,J_k,I)\in\vV_{k-1,k}$ and $\vb=(I_1,\dots,I_k,J)$ be two adjacent vertices. Let $i\in\zint1n$ be such that $I=J+i$. Let us extract $y_i$ from $h(\va;t)$ and from $h(\vb;t)$. One has
\begin{align*}
h(I_1,\dots,I_k,J;t)
 &=(\mat e_{\compl{I_1}}, \dots, \mat e_{\compl{I_k}})_{y+\f}y^{\compl J} \\
 &=\Big((\mat e_{\compl{I_1}}, \dots, \mat e_{\compl{I_k}})_{y+\f}y^{\compl I}\Big)y_i,
\end{align*}
and, using the Laplace cofactor expansion along the column of the determinant which contains $y_i$,
\begin{align*}
h(J_1,\dots,J_k,I;t)
 &=(\mat e_{\compl{J_1}}, \dots, \mat e_{\compl{J_k}})_{y+\f}y^{\compl I} \\
 &=\epsilon(\compl{I_1}\concat i)\cdots\epsilon(\compl{I_k}\concat i)(\mat e_{\compl{I_1}}\wedge e_i, \dots, \mat e_{\compl{I_k}}\wedge e_i)_{y+\f}y^{\compl I} \\
 &=\Big(\eta_{\va,\vb}(\mat e_{\compl{I_1}}, \dots, \mat e_{\compl{I_k}})_{y+\f}y_i+\O_y(y^{\compl{I_1}\cap\dots\cap\compl{I_k}})\Big)y^{\compl I}.
\end{align*}
Thus,
\begin{align*}
h(\va;t)-\eta_{\va,\vb}h(\vb;t)
 &= \O_y(y^{\compl{I_1}\cap\dots\cap\compl{I_k}}y^{\compl I}) \\
 &= \O_y(y^{\compl{I_1}}y^{\compl I}).
\end{align*}
The monomial $y^{\compl{I_1}}y^{\compl I}$ is present in $f_1^2$ with a positive coefficient. Thus,
\[ h(\va;t)-\eta_{\va,\vb}h(\vb;t)=\O_y(f_1^2). \]

\item Since there are two kinds of special vertices, we will make two cases.

Let $\va=(J_1,\dots,J_k,I)$ be a special vertex of $\vV_{r-1,r}$. Assume, without loss of generality, that $\Fr(J_1)\neq\Fr(J_2)$. We have seen in Equation~\eqref{5:eqn:variation:h_O_xi} that
\[ h(\va;t)=\O_y(y^{\compl{(J_1\cup\dots\cup J_k)}}y^{\compl I}). \]
By Lemma~\ref{5:lem:same_rank_independents}, $\Fr(J_1)\neq\Fr(J_2)$ implies that there exists $j\in J_1\cap\Fr(J_2)$. Since $\rk(J_2)=r-1$, the set $I':=J_2+j$ is in $\Ind_r$. But $I'\subset J_1\cup\dots\cup J_k$, and so
\[ y^{\compl{(J_1\cup\dots\cup J_k)}}=\O_y(y^{\compl{I'}}). \]
Then,
\[ h(\va;t)=\O_y(y^{\compl{I'}}y^{\compl I})=\O_y(f_1^2). \]

\smallskip

In the same way, if $\va=(I_1,\dots,I_k,J)$ is a special vertex of $\vV_{r,r-1}$, we have
\[ h(\va;t)=\O_y(y^{\compl{(I_1\cup\dots\cup I_k)}}y^{\compl J}). \]
Assume, without loss of generality, that there exists an element $i$ in $(I_1\cap\Fr(J))\setminus(I_2\cap\Fr(J))$. One has
\[ i\not\in\compl{(I_1\cup\dots\cup I_k)}, \qquad \compl{(I_1\cup\dots\cup I_k)}\subset\compl{I_2}, \qquad i\in\compl{I_2}. \]
This implies $\compl{(I_1\cup\dots\cup I_k)}+i\subset\compl{I_2}$. Moreover, the set $I:=J+i$ is in $\Ind_r$. One obtains
\begin{align*}
y^{\compl{(I_1\cup\dots\cup I_k)}}y^{\compl J}
 &= y^{\compl{(I_1\cup\dots\cup I_k)}+i}y^{\compl J-i} \\
 &= \O_y(y^{\compl{I_2}}y^{\compl I}).
\end{align*}
Finally,
\[ h(\va;t)=\O_y(f_1^2). \]

\item Let $\~w:=w\concat((0,\dots,0, -1))$ and $\~u$ such that $\~u^\tran$ is a basis of $\vect{(u\concat(\beta))^\tran}$. Notice that $\~U{\~W}^\tran=0$. The statement is equivalent to
\[ \frac{\det(\~W_{\compl{I_\ell}+(n+1)})}{\det(\~W_{\compl{I_1}+(n+1)})} = \frac{\epsilon(\compl{I_\ell}\concat i)\det(\~W_{\compl{J_\ell}})}{\epsilon(\compl{I_1}\concat i)\det(\~W_{\compl{J_1}})}. \]
The families $\~u$ and $\~w$ verify the conditions of Lemma~\ref{5:lem:duality_with_signs}. Applying the lemma, last equation is equivalent to
\begin{equation} \label{5:eqn:claim:variation:normal_adjacent}
\frac{\epsilon((\compl{I_\ell}+(n+1))\concat I_\ell)\det(\~U_{I_\ell})}{\epsilon((\compl{I_1}+(n+1))\concat I_1)\det(\~U_{I_1})}=\frac{\epsilon(\compl{I_\ell}\concat i)\epsilon(\compl{J_\ell}\concat(J_\ell+(n+1)))\det(\~U_{J_\ell+(n+1)})}{\epsilon(\compl{I_1}\concat i)\epsilon(\compl{J_1}\concat(J_1+(n+1)))\det(\~U_{J_1+(n+1)})}.
\end{equation}
We can make some simplifications. For example,
\[ \epsilon((\compl{I_\ell}+(n+1))\concat I_\ell)=(-1)^{\card{I_\ell}}\epsilon(\compl{I_\ell}\concat I_\ell\concat(n+1))=(-1)^{\card{I_\ell}}\epsilon(\compl{I_\ell}\concat I_\ell). \]
After simplification, we get
\[ \frac{\epsilon(\compl{I_\ell}\concat I_\ell)\epsilon(J_\ell\concat i)\det(\~U_{J_\ell}\concat\~U_i)}{\epsilon(\compl{I_1}\concat I_1)\epsilon(J_1\concat i)\det(\~U_{J_1}\concat\~U_i)}=\frac{\epsilon(\compl{I_\ell}\concat i\concat J_\ell)\det(\~U_{J_\ell+(n+1)})}{\epsilon(\compl{J_1}\concat i\concat J_1)\det(\~U_{J_1+(n+1)})}. \]
As $\vb$ is ordinary, $\cl(J_\ell)=\cl(J_1)$. Thus there exists $P\in\M_{r-1}(\R)$ such that $U_{J_\ell}=U_{J_1}P$. Hence, setting
\[ \~P=\pbmatrixV{P}{\begin{smallmatrix}0 \\ \,\rotatebox{90}{$\begin{smallmatrix}\cdots\end{smallmatrix}$} \\ 0\end{smallmatrix}}{\begin{smallmatrix}0 & \cdots & 0\end{smallmatrix}}{\begin{smallmatrix}1\end{smallmatrix}}, \]
one obtains $\~U_{J_\ell}\concat\~U_i=(\~U_{J_1}\concat\~U_i)\~P$ and $\~U_{J_\ell+(n+1)}=\~U_{J_1+(n+1)}\~P$. So, both ratios equal $\det(P)$ up to a sign.

We now compute the sign of the numerators:
\begin{align*}
\epsilon(\compl{I_\ell}\concat I_\ell)\epsilon(J_\ell\concat i)\epsilon(\compl{I_\ell}\concat i\concat J_\ell)
 &= \epsilon(\compl{I_\ell}\concat i\concat J_\ell)\epsilon(i\concat J_\ell)\epsilon(J_\ell\concat i)\epsilon(\compl{I_\ell}\concat i\concat J_\ell) \\
 &= (-1)^{r-1}.
\end{align*}
The signs of the numerators and the denominators simplify, which concludes the proof.

\item Since $s(\va)$ and $s(\vb)$ are nonzero, $\det(W_{\compl J})$ and $\det(W_{\compl{J_1}})$ are nonzero too. Thus we can apply Point~\eqref{5:claim:variation:normal_adjacent} for all $\ell\in\zint1k$. Multiplying left-hand members and right-hand members, we get the equality of Point~\eqref{5:claim:variation:normal_adjacent_2}. \qedhere
\end{enumerate}
\end{proof}

Let $\CC(\vG)$ be the set of connected components of $\vG$. If $\vH=(\vV',\vE')\in\CC(\vG)$, we set
\begin{gather*}
\phantom{\myand}
\vV'_{r-1,r}:=\vV'\cap\vV_{r-1,r}\myand \\
\vV'_{r,r-1}:=\vV'\cap\vV_{r,r-1}.
\end{gather*}
Moreover, we denote by $\OCC(\vG)$ the set of \emph{ordinary connected components} of $\vG$ that only contain ordinary vertices. Let $\SCC(\vG):=\CC(\vG)\setminus\OCC(\vG)$ be the set of \emph{special connected components} of $\vG$.

The equation we wanted to show, namely,
\[ g_1f_2-g_2f_1=\O_y(f_1^2), \]
is equivalent to
\[ \sum_{(\vV',\vE')\in\CC(\vG)}\bigg(\sum_{\va\in\vV'_{r,r-1}}s(\va)h(\va;t)-\sum_{\va\in\vV'_{r-1,r}}s(\va)h(\va;t)\bigg)=\O_y(f_1^2). \]

In the above sum, we can remove the special connected components because of Points~\eqref{5:claim:variation:special} and~\eqref{5:claim:variation:adjacent} of Claim~\ref{5:claim:variation}. Thus, it remains to show
\[ \sum_{(\vV',\vE')\in\OCC(\vG)}\bigg(\sum_{\va\in\vV'_{r,r-1}}s(\va)h(\va;t)-\sum_{\va\in\vV'_{r-1,r}}s(\va)h(\va;t)\bigg)=\O_y(f_1^2). \]

Actually, we will prove that, for all $\vH=(\vV',\vE')\in\OCC(\vG)$,
\begin{equation} \label{5:eqn:variation:compensation}
\sum_{\va\in\vV'_{r,r-1}}s(\va)h(\va;t)-\sum_{\va\in\vV'_{r-1,r}}s(\va)h(\va;t)=\O_y(f_1^2).
\end{equation}

\medskip

We now see that any ordinary connected component of $\vG$ is naturally isomorphic to a connected component of (a subgraph of) the exchange graph of $\Ma$. This will allow us to applied the results of Section~\ref{5:sec:exchange_graph}. More precisely, for each $\va\in\vV'_{r,r-1}$, we will find $\vb\in\vV'_{r-1,r}$ in the same connected component such that
\[ s(\va)h(\va; t)-s(\vb)h(\vb; t)=\O_y(f_1^2). \]

Fix a connected component $\vH=(\vV',\vE')$ in $\OCC(\vG)$. Let $\G=(\V,\gE)$ be the exchange graph of $\Ma$, as defined in Section~\ref{5:sec:exchange_graph}. We define the following projection:
\[ \begin{array}{rrcl}
\pi: & \vV & \to & \V_{r,r-1}\sqcup\V_{r-1,r}, \\
 & (K_1,\dots,K_k,L) & \mapsto & (K_1,L).
\end{array} \]
Let $\H:=\G[\pi(\vV')]$ be the induced subgraph of $\G$ with vertex set the image of $\vV'$ by $\pi$. Let $\V'$ be its vertex set and $\gE'$ be its edge set. We have the following claim.

\begin{claim} \label{5:claim:variation_3}
$\H$ is a connected component of $\G_{r,r-1}$, and the map $\pi$ induces an isomorphism of graphs between $\vH$ and $\H$.
\end{claim}

\begin{proof}
First we prove that $\pi$ is injective on $\vH$. Notice that, as for $\G$, for every $l\in\zint1k$, $K_l\multicup L$ is invariant in the connected components of $\vG$, \ie, if $\va=(K_1,\dots,K_k,L)$ and $\vb=(K'_1,\dots,K'_k,L')$ are two vertices in $\vH$, then $K_l\multicup L=K'_l\multicup L'$ for all $l\in\zint1k$. Thus, if we know $\va$, we can retrieves $\vb$ only knowing $L'$. But $L'$ is encoded in $\pi(\vb)$. That concludes the injectivity. Thus $\pi$ induces a bijection between $\vV'$ and $\V'$.

Next, we prove that $\pi$ induces a natural bijection between $\vE'$ and edges of $\G_{r,r-1}$ which are incident to a vertex of $\H$. Let $\ve\in\vE'$ and let $\va$ and $\vb$ be its two endpoints. Using the definitions, it is clear that $\pi(\va)$ and $\pi(\vb)$ are linked.

Reciprocally, let $e$ be an edge of $\G_{r,r-1}$ which is incident to a vertex $a$ of $\V'$. Let $\va\in\vV'$ such that $\pi(\va)=a$. Let $b$ be the other endpoint of $e$. We want to show that there exists $\vb$ in $\vV$ such that $\pi(\vb)=b$ and that $\va$ and $\vb$ are connected by an edge. There are two cases.
\begin{itemize}
\item If $\va=(J_1,\dots,J_k,I)\in\vV_{r-1,r}$, then $a=(J_1,I)$. There exists $i\in\Fr(J_1)$ such that $b=(J_1+i,I-i)$. Since $\va$ is an ordinary vertex, $i\in\Fr(J_1)$ implies that $i\in\Fr(J_l)$ for all $l\in\zint1k$. Thus, $\vb:=(J_1+i, \dots, J_r+i, I-i)$ is a neighbor of $\va$ such that $\pi(\vb)=b$.
\item If $\va=(I_1,\dots,I_k,J)\in\vV_{r,r-1}$, then $a=(I_1,J)$. There exists $i\in\Fr(J)\cap I_1$ such that $b=(I_1-i,J+i)$. Since $\va$ is a special vertex, $i\in\Fr(J)\cap I_l$ for all $l\in\zint1k$. Finally, $\vb:=(I_1-i, \dots, I_k-i, J+i)$ is a neighbor of $\va$ in $\vH$, and $\pi(\vb)=b$.
\end{itemize}

We have proved that $\pi$ induces a natural bijection between vertices of $\vH$ and vertices of $\H$, and between edges of $\vH$ and edges of $\G_{r,r-1}$ which are incident to a vertex of $\H$. Thus, the claim is true.
\end{proof}
Let us denote by $\pi':\vH\to\H$ the isomorphism induced by $\pi$.

\medskip

We now study a second bijection. The well-definiteness will be justified in Claim~\ref{5:claim:variation_2} below. Set
\[ \begin{array}{rrcl}
\Phi: & \V'_{r,r-1} & \to & \V'_{r-1,r}, \\
 & (I,J) & \mapsto & \Big(A\cup(J\setminus B), B\cup(I\setminus A)\Big),
\end{array} \]
where $(A,B)$ is the MCP of any vertex of $\H$.
This map induces a map on $\vH$:
\[ \begin{array}{rrcl}
\~\Phi: & \vV'_{r,r-1} & \to & \vV'_{r-1,r}, \\
 & \va & \mapsto & \pi'^{-1}\circ\Phi\circ\pi'(\va).
\end{array} \]

\begin{claim} \label{5:claim:variation_2}
$\Phi$ and $\~\Phi$ have the following properties.
\begin{enumerate}
\item $\Phi$ and $\~\Phi$ are well-defined, and both are bijections.
\item If $\va\in\vV'_{r,r-1}$, then $s(\va)h(\va;t)-s(\~\Phi(\va))h(\~\Phi(\va);t)=\O_y(f_1^2)$.
\end{enumerate}
\end{claim}

\begin{proof} We prove the two points independently.
\begin{enumerate}
\item It is enough to show the first point for $\Phi$. Let $(I,J)\in\V'_{r,r-1}$. Let $(A,B):=\MCP(I,J)$. Let $(J',I'):=\Big(A\cup(J\setminus B), B\cup(I\setminus A)\Big)$. We want to apply Corollary~\ref{5:cor:exchange_graph} in order to show that $(J',I')\in\V'$. In the definition of $(J',I')$, the unions are disjoint. Indeed, if $i\in A\cap J$, \resp $i\in B\cap I$, then $(\{i\},\{i\})$ is a codependent pair of $(I,J)$, and so $i\in B$, \resp $i\in A$. Thus
\[ (A\sqcup(J\setminus B))\multicup(B\sqcup(I\setminus A))=I\multicup J. \]
Set $(A',B'):=\MCP(J',I')$. Clearly $(A,B)\subset(A',B')$. Moreover
\begin{align*}
\cl(A') &= \cl(A\sqcup(A'\setminus A)) \\
 &= \cl(\cl(A)\sqcup(A'\setminus A)) \\
 &= \cl(\cl(B)\sqcup(A'\setminus A)) \\
 &= \cl(B\sqcup(A'\setminus A)).
\end{align*}
Similarly,
\[ \cl(B')=\cl(A\sqcup(B'\setminus B)). \]
But $B\sqcup(A'\setminus A)\subset B\sqcup(J'\setminus A)=J$ and $A\sqcup(B'\setminus B)\subset I$. Since
\[ \cl(A\sqcup(B'\setminus B))=\cl(B')=\cl(A')=\cl(B\sqcup(A'\setminus A)), \]
$\Big(A\sqcup(B'\setminus B),B\sqcup(A'\setminus A)\Big)$ is a codependent pair of $(I,J)$. Thus, it is included in $(A,B)$, and so $A'\setminus A$ and $B'\setminus B$ are empty. Finally,
\[ I\multicup J=J'\multicup I'\myand\MCP(I,J)=\MCP(J',I'). \]
We can apply Corollary~\ref{5:cor:exchange_graph}, and we obtain $(J',I')\in\V'_{r-1,r}$. Thus, $\Phi$ is well-defined.

One can easily retrieve $(I,J)$ from $(J',I')$ by $(I,J):=(A\cup(I'\setminus B), B\cup(J'\setminus A))$. Thus, $\Phi$ is a bijection. So is $\~\Phi$.

\item Let $\va=(I,I_2,\dots,I_k,J)\in\vV_{r,r-1}$ be a vertex and let $\vb=(J',J_2,\dots,J_k,I')\in\vV_{r-1,r}$ be the image of $\va$ by $\~\Phi$. Let $(\va=\va_0, \va_1, \dots, \va_m=\vb)$ be a path from $\va$ to $\vb$. Point~\eqref{5:claim:variation:adjacent} of Claim~\ref{5:claim:variation} used $m$ times show that
\begin{gather*}
h(\va; t)-\eta_{\va_0,\va_1}\cdots\eta_{\va_{m-1},\va_m}h(\vb; t)=\O_y(f_1^2).
\end{gather*}

It remains to prove that
\begin{equation} \label{5:eqn:variation:s=s}
s(\va)=\eta_{\va_0,\va_1}\cdots\eta_{\va_{m-1},\va_m}s(\vb).
\end{equation}
We can assume without loss of generality that $s(\va_0), \dots, s(\va_m)$ are nonzero. Indeed, this is true for all $\beta$ belonging to a dense subset of $\vect{u}$. Moreover, both members of Equation~\eqref{5:eqn:variation:s=s} are continuous in $\beta$.

Point~\eqref{5:claim:variation:normal_adjacent_2} of Claim~\ref{5:claim:variation} used $m$ times show that
\begin{gather*}
\frac{s(\va)}{\det(V_{\compl I})^k\det(W_{\compl J})^k}=\eta_{\va_0,\va_1}\cdots\eta_{\va_{m-1},\va_m}\frac{s(\vb)}{\det(W_{\compl{J'}})^k\det(V_{\compl{I'}})^k}.
\end{gather*}
It remains to show that $\det(V_{\compl I})^k\det(W_{\compl J})^k=\det(V_{\compl{I'}})^k\det(W_{\compl{J'}})^k$, \ie, that
\[ \left\vert\frac{\det(V_{\compl{I'}})}{\det(V_{\compl{I}})}\right\vert=\left\vert\frac{\det(W_{\compl{J}})}{\det(W_{\compl{J'}})}\right\vert. \]
The proof will be very similar to the proof of~\eqref{5:claim:variation:normal_adjacent} in Claim~\ref{5:claim:variation}. We use the same notations. It is enough to show that the analogous of Equation~\eqref{5:eqn:claim:variation:normal_adjacent} holds. Since we do not care about signs, this analogous is
\begin{equation} \label{5:eqn:claim:variation_2_1}
\left\vert\frac{\det(\~U_{I'})}{\det(\~U_{I})}\right\vert=\left\vert\frac{\det(\~U_{J+(n+1)})}{\det(\~U_{J'+(n+1)})}\right\vert.
\end{equation}
Set $(A,B):=\MCP(I,J)$. From $\cl(A)=\cl(B)$ we deduce that there exists a (unique) matrix $\~P\in\M_r(\R)$, of the form
\[ \pbmatrix{P}{0}{0}{\Id}, \]
such that $\~U_B\concat\~U_{J\setminus B}\concat\~U_{\{n+1\}}=(\~U_A\concat\~U_{J\setminus B}\concat\~U_{\{n+1\}})\~P$. Therefore, the second ratio of~\eqref{5:eqn:claim:variation_2_1} equals $\abs{\det(P)}$. Moreover, $\~U_B\concat\~U_{I\setminus A}=(\~U_A\concat\~U_{I\setminus A})\~P$. Thus, the first ratio also equals $\abs{\det(P)}$. Finally, Equation~\eqref{5:eqn:variation:s=s} is true, which concludes the proof. \qedhere
\end{enumerate}
\end{proof}

We recall that we wanted to show Equation~\eqref{5:eqn:variation:compensation}, which is:
\[ \sum_{\va\in\vV'_{r,r-1}}s(\va)h(\va;t)-\sum_{\va\in\vV'_{r-1,r}}s(\va)h(\va;t)=\O_y(f_1^2). \]
By Claim~\ref{5:claim:variation_2}, this equation is equivalent to
\[ \sum_{\va\in\vV'_{r,r-1}}s(\va)h(\va;t)-s(\~\Phi(\va))h(\~\Phi(\va);t)=\O_y(f_1^2), \]
which is true by the second point of the claim. Finally,
\[ g_1f_2-f_1g_2=\O_y(f_1^2), \]
which concludes the proof of the theorem. \hfill$\square$

\bigskip

One can easily obtain from Theorem~\ref{5:thm:variation} the following corollary.

\begin{corollary} \label{5:cor:variation}
Let $l$ be a positive integer, $u'$ be a family of $l$ vectors in $\vect u$ and $v'$ be such that $Uv'_i=u'_i$, for every $i\in\zint1l$. Let $\fF$ and $Y$ be two functions as in Theorem~\ref{5:thm:variation}. Then
\[ \frac{\norm{\mat v^\tran\wedge\mat{v'}}_y^k}{\norm{\mat v^\tran}_y^k}-\frac{\norm{\mat v^\tran\wedge\mat{v'}}_{y+\f}^k}{\norm{\mat v^\tran}_{y+\f}^k}=\O_y(\max_{i\in\zint1n}(y_i^{l-1})). \]
\end{corollary}

\begin{proof}
If $u'$ is not free, then both ratios equal zero. Otherwise, one can apply Theorem~\ref{5:thm:variation} to well-chosen families, and obtain
\begin{align*}
\frac{\norm{\mat v^\tran\wedge\mat{v'}}_{y+\f}^k}{\norm{\mat v^\tran}_{y+\f}^k}
 &= \frac{\norm{\mat v^\tran\wedge\mat{v'}}_{y+\f}^k}{\norm{\mat v^\tran\wedge\mat{v'}_{\zint1{l-1}}}_{y+\f}^k}\,\frac{\norm{\mat v^\tran\wedge\mat{v'}_{\zint1{l-1}}}_{y+\f}^k}{\norm{\mat v^\tran\wedge\mat{v'}_{\zint1{l-2}}}_{y+\f}^k}\,\cdots\,\frac{\norm{\mat v^\tran\wedge\mat{v'}_{\zint11}}_{y+\f}^k}{\norm{\mat v^\tran}_{y+\f}^k} \\
 &= \left(\frac{\norm{\mat v^\tran\wedge\mat{v'}}_y^k}{\norm{\mat v^\tran\wedge\mat{v'}_{\zint1{l-1}}}_y^k}+\O_y(1)\right)\,\cdots\,\left(\frac{\norm{\mat v^\tran\wedge\mat{v'}_{\zint11}}_y^k}{\norm{\mat v^\tran}_y^k}+\O_y(1)\right).
\end{align*}
In last member, every ratio is a Symanzik rational fraction. Let us prove that all Symanzik rational fraction always are a $\O_y(\max_{i\in\zint1n}(y_i))$. Indeed, all coefficients are positive, and if $\lambda_Ky^K$ is a monomial of the numerator, then $\compl K\in\Ind_{r-1}$. Therefore, there exists $I\in\Ind_r(\Ma)$ containing $\compl K$. Then, $y^{\compl I}$ appears in the denominator with a positive coefficient. And, $y^K/y^{\compl I}=\O_y(\max_{i\in\zint1n}(y_i))$. Finally, the product equals
\[ \frac{\norm{\mat v^\tran\wedge\mat{v'}}_y^k}{\norm{\mat v^\tran}_y^k}+\O_y(\max_{i\in\zint1n}(y_i^{l-1})). \qedhere \]
\end{proof}

As the following example shows, one cannot expect a better asymptotic for Corollary~\ref{5:cor:variation}.

\begin{example}
We set
\[ U=\begin{pmatrix} 0 & 1 & 0 & 0 \\ 0 & 0 & 1 & 0 \\ 0 & 0 & 0 & 1\end{pmatrix},\quad V^\tran=\begin{pmatrix} 1 \\ 0 \\ 0 \\ 0 \end{pmatrix},\quad V'=\begin{pmatrix} 0 & 0 & 0 \\ 1 & 0 & 0 \\ 0 & 1 & 0 \\ 0 & 0 & 1 \end{pmatrix}. \]
And $\fF$ will be constant equal to
\[ \begin{pmatrix} 0 & 0 & 0 & 0 \\ 0 & 0 & 0 & 0 \\ 0 & 0 & 0 & 0 \\ 0 & 0 & 0 & -1 \end{pmatrix}. \]
Thus,
\[ \norm{\mat v^\tran\wedge\mat v'}^k_y = y_1y_2y_3y_4, \quad
 \norm{\mat v^\tran}_y^k = y_1, \quad
 \norm{\mat v^\tran\wedge\mat{v'}}_{y+\f}^k = y_1y_2y_3(y_4-1), \quad
 \norm{\mat v^\tran}_{y+\f}^k = y_1. \]
Finally,
\[ \frac{\norm{\mat v^\tran\wedge\mat{v'}}_y^k}{\norm{\mat v^\tran}_y^k}-\frac{\norm{\mat v^\tran\wedge\mat{v'}}_{y+\f}^k}{\norm{\mat v^\tran}_{y+\f}^k}
 = \frac{y_1y_2y_3y_4}{y_1}-\frac{y_1y_2y_3(y_4-1)}{y_1}
 = y_2y_3. \qedhere \]
\end{example}

\bibliographystyle{abstract}
\bibliography{$HOME/bibliography/bibliography}

\end{document}